\documentclass[11pt]{article}

\usepackage[english]{babel}
\usepackage{amsmath,amssymb,mathtools}
\usepackage{amsthm}
\usepackage{bm}
\usepackage{booktabs}
\usepackage{enumitem}
\usepackage{tikz}
\usepackage{float}
\usepackage[margin=1in]{geometry}
\usepackage[expansion=false]{microtype}
\usepackage{natbib}
\setcitestyle{semicolon}
\usepackage[pdfencoding=auto,
  pdftitle={Spectral Equality for Novikov Integrability: Recursive Criticality and Unbounded Asymptotic Depth},
  pdfauthor={Jaewoo Lee},
  pdflang={en-GB}]{hyperref}
\hypersetup{colorlinks=true,linkcolor=blue,citecolor=blue,urlcolor=blue}
\usepackage[nameinlink,capitalise]{cleveref}
\usepackage{aliascnt}

\newtheorem{theorem}{Theorem}[section]

\newaliascnt{lemma}{theorem}
\newtheorem{lemma}[lemma]{Lemma}
\aliascntresetthe{lemma}

\newaliascnt{proposition}{theorem}
\newtheorem{proposition}[proposition]{Proposition}
\aliascntresetthe{proposition}

\newaliascnt{corollary}{theorem}
\newtheorem{corollary}[corollary]{Corollary}
\aliascntresetthe{corollary}

\newaliascnt{definition}{theorem}
\newtheorem{definition}[definition]{Definition}
\aliascntresetthe{definition}

\newaliascnt{example}{theorem}
\newtheorem{example}[example]{Example}
\aliascntresetthe{example}

\theoremstyle{remark}
\newaliascnt{remark}{theorem}
\newtheorem{remark}[remark]{Remark}
\aliascntresetthe{remark}

\crefname{theorem}{Theorem}{Theorems}
\Crefname{theorem}{Theorem}{Theorems}
\crefname{lemma}{Lemma}{Lemmas}
\Crefname{lemma}{Lemma}{Lemmas}
\crefname{proposition}{Proposition}{Propositions}
\Crefname{proposition}{Proposition}{Propositions}
\crefname{corollary}{Corollary}{Corollaries}
\Crefname{corollary}{Corollary}{Corollaries}
\crefname{definition}{Definition}{Definitions}
\Crefname{definition}{Definition}{Definitions}
\crefname{example}{Example}{Examples}
\Crefname{example}{Example}{Examples}
\crefname{remark}{Remark}{Remarks}
\Crefname{remark}{Remark}{Remarks}
\crefname{appendix}{Appendix}{Appendices}
\Crefname{appendix}{Appendix}{Appendices}

\renewcommand{\epsilon}{\varepsilon}
\AtBeginDocument{\renewcommand{\d}{\mathrm{d}}}
\newcommand{\dd}{\,\mathrm{d}}
\allowdisplaybreaks
\title{Spectral Equality for Novikov Integrability:
Recursive Criticality and Unbounded Asymptotic Depth}

\author{Jaewoo Lee\thanks{%
  Email: \texttt{jaewoolee7@outlook.com}.}}
\date{}

\begin{document}
\maketitle

\begin{abstract}
We determine the finiteness boundary of the Novikov exponential moment for a
one-dimensional constant-volatility mean-reverting diffusion.  A localised
change of measure cancels the squared drift and leaves a Brownian
Feynman--Kac functional with potential \(q/2\), where \(q=-\lambda'\).  If
\(q(\theta+y)\sim\kappa_\infty y^2\), the exact spectral boundary is
\(\kappa_\infty\sigma^2T^2=\pi^2\); for cubic drift it becomes
\(c\sigma^2T^2=\pi^2/3\), and equality is divergent.  On the spectral
equality surface, the first lower-order transition occurs at power \(4/3\),
where an explicit coefficient separates the two sides.  For symmetric finite
pure-power tails, exact tuning generates the recursion
\(\beta_n=1+3^{-(n+1)}\).  We prove that this recursion gives a complete
classification of the class.  Each individual tail is decided after
finitely many comparisons, but the required depth is unbounded: arbitrarily
long common critical prefixes can lead to opposite outcomes.  The
drift-removing stochastic exponential nevertheless remains a true
martingale; under the physical law it has no higher moments on the
steep-drift class, while the reverse density is essentially bounded.


\end{abstract}

\section{Introduction}\label{sec:introduction}

Let \(X\) solve
\begin{equation}\label{eq:intro-sde}
  \d X_t=\lambda(X_t)\dd t+\sigma\dd W_t,
  \qquad \sigma>0,
\end{equation}
where the drift \(\lambda\) is decreasing and vanishes at a unique equilibrium
\(\theta\).  Put
\begin{equation}\label{eq:intro-q}
  \eta(x):=\frac{\lambda(x)}{\sigma},
  \qquad q(x):=-\lambda'(x)\ge0,
\end{equation}
and consider the pointwise exponential functional
\begin{equation}\label{eq:intro-functional}
  \mathcal N_a(T,x)
  :=\mathbb E_x\bigg[
       \exp\bigg\{a\int_0^T\eta(X_s)^2\dd s\bigg\}
     \bigg],
  \qquad a\ge0.
\end{equation}
At \(a=1/2\), finiteness of \(\mathcal N_a(T,x)\) is the classical
Novikov condition for
\begin{equation}\label{eq:intro-Z}
  Z_t=\mathcal E\bigg(-\int_0^\cdot\eta(X_s)\dd W_s\bigg)_t.
\end{equation}
We study the pointwise finiteness boundary of
\cref{eq:intro-functional} for nonlinear mean reversion, with particular
attention to the spectral equality surface.  This is a stronger question
than martingality.  Novikov finiteness implies that \(Z\) is a martingale,
but failure of the condition need not cause loss of mass.  In the class
considered below, \(Z\) remains a true martingale beyond the moment boundary.

The two directions of the change of measure are sharply asymmetric in the
steep-drift regime.  The physical-to-Brownian density has no higher moments,
whereas the Brownian-to-physical density is essentially bounded.  These
statements hold for every horizon, independently of the critical time found
below.  Moment criticality, martingality, and directional likelihood-ratio
integrability therefore have different thresholds.

A localised change of measure gives an exact cancellation.  If
\(\Lambda'=\lambda\), transfer to the measure under which \(X\) is Brownian,
followed by It\^o's formula, gives
\begin{align}\label{eq:intro-transfer}
 \mathcal N_a(T,x)
 =\mathbb E^{\mathbb Q_0}_x\bigg[
  \exp\bigg\{
   \frac{\Lambda(X_T)-\Lambda(x)}{\sigma^2}
   +\int_0^T\Psi_a(X_s)\dd s
  \bigg\}\bigg],
 \qquad
 \Psi_a:=\frac{q}{2}+
          \bigg(a-\frac12\bigg)\frac{\lambda^2}{\sigma^2},
\end{align}
where \(X_s=x+\sigma W_s\) under \(\mathbb Q_0\).  The localisation is important:
\cref{eq:intro-transfer} is proved without assuming the global martingale
property that the formula is later used to study.

The moment boundary has a concrete likelihood interpretation.  Once
\cref{thm:Z-martingale} identifies \(Z_T\) as the likelihood ratio of the
drift-removed law with respect to the physical law, put
\[
 M_t:=\int_0^t\eta(X_s)\dd W_s,
 \qquad
 \langle M\rangle_T=\int_0^T\eta(X_s)^2\dd s.
\]
Then \(Z_T=\exp\{-M_T-\langle M\rangle_T/2\}\) and
\(\mathcal N_a(T,x)=\mathbb E_x[e^{a\langle M\rangle_T}]\).  Define the
exponential-moment domain
\begin{equation}\label{eq:intro-moment-domain}
 \mathcal A_T(x):=
 \{a\ge0:\mathcal N_a(T,x)<\infty\}.
\end{equation}
For every \(a\in\mathcal A_T(x)\) and \(u\ge0\), Markov's inequality gives
the standard Chernoff bound
\begin{equation}\label{eq:intro-chernoff}
 \mathbb P_x\big(\langle M\rangle_T>u\big)
 \le \mathcal N_a(T,x)e^{-au}.
\end{equation}
On \(\mathcal D_\star\), \cref{cor:critical-coefficient} and monotonicity in
\(a\) imply that, for each fixed \(T\) and \(x\), precisely one of the two
possibilities holds:
\begin{equation}\label{eq:intro-domain-dichotomy}
 \mathcal A_T(x)=[0,1/2)
 \qquad\text{or}\qquad
 \mathcal A_T(x)=[0,1/2].
\end{equation}
Moreover, for \(0\le a<1/2\), \cref{eq:subcritical-bound} makes
\cref{eq:intro-chernoff} explicit:
\begin{equation}\label{eq:intro-explicit-chernoff}
 \mathbb P_x\big(\langle M\rangle_T>u\big)
 \le
 \exp\bigg\{
  \frac{\Lambda(\theta)-\Lambda(x)}{\sigma^2}
  +T\sup_{y\in\mathbb R}\Psi_a(y)-au
 \bigg\}.
\end{equation}
Thus the equality problem asks whether the endpoint \(1/2\) belongs to
\(\mathcal A_T(x)\).  For cubic reversion, \cref{eq:intro-cubic} gives
\begin{equation}\label{eq:intro-cubic-moment-domain}
 \mathcal A_T(x)=
 \begin{cases}
  [0,1/2],&c\sigma^2T^2<\pi^2/3,\\
  [0,1/2),&c\sigma^2T^2\ge\pi^2/3.
 \end{cases}
\end{equation}
Below the critical horizon, the endpoint choice in
\cref{eq:intro-chernoff} yields
\(\mathbb P_x(\langle M\rangle_T>u)
 \le\mathcal N_{1/2}(T,x)e^{-u/2}\).  At and above it, that endpoint
moment-based bound is unavailable, although every bound with \(a<1/2\)
remains available.  This is a statement about endpoint inclusion in the
Laplace-transform domain; it does not rule out tail estimates obtained by
other methods.

The paper has two central contributions.  The first is an exact spectral
criticality theorem, including the equality behaviour of cubic reversion.
The second is a complete equality-surface classification on the class
\(\mathcal P_{\mathrm{sym}}(\kappa_\infty)\) of~\cref{def:Psym}: a recursive sequence of
critical powers and coefficients decides every tail in the class, although
the number of required comparisons is not uniformly bounded.  The remaining
results locate these statements within the off-critical moment problem and
the two directions of the change of measure.

The two terms in \(\Psi_a\) create two different phase transitions.  On the
class \(\mathcal D_\star\) of~\cref{def:Dstar}, \(\lambda^2\) dominates \(q\)
in the tails.  Hence the sign of \(a-1/2\) decides every off-critical case:
\(\mathcal N_a(T,x)\) is finite for every \(T\) when \(a<1/2\), and infinite
for every \(T>0\) when \(a>1/2\).  At \(a=1/2\), however, the entire
\(\lambda^2\)-term cancels and the residual potential is \(q/2\).  Brownian
paths pay a quadratic Cameron--Martin cost to spend time far from the origin,
so subquadratic \(q\) is integrable, windowed superquadratic \(q\) is not, and
quadratic \(q\) is the spectral boundary.

Suppose, in particular, that
\begin{equation}\label{eq:intro-kappa}
  \frac{q(\theta+y)}{y^2}\longrightarrow\kappa_\infty
  \in(0,\infty).
\end{equation}
Conditioning the transferred Brownian motion on its endpoint produces a
Brownian bridge.  Its covariance eigenvalues are
\(T^2/(n^2\pi^2)\), and therefore the strict sides of the boundary are
\begin{equation}\label{eq:intro-spectral-surface}
  \kappa_\infty\sigma^2T^2<\pi^2
  \quad\text{and}\quad
  \kappa_\infty\sigma^2T^2>\pi^2.
\end{equation}
For the cubic drift \(\lambda(\theta+y)=-cy^3\), one has
\(\kappa_\infty=3c\), so the exact result becomes
\begin{equation}\label{eq:intro-cubic}
  \mathcal N_{1/2}(T,x)<\infty
  \quad\Longleftrightarrow\quad
  c\sigma^2T^2<\frac{\pi^2}{3};
\end{equation}
equality belongs to the divergent side.

The leading quadratic coefficient does not decide the equality surface.
Suppose that the first correction has the form
\begin{equation}\label{eq:intro-second-order}
 q(\theta+y)
 =\kappa_\infty y^2-b|y|^\alpha+o(|y|^\alpha),
 \qquad b>0,\quad 0<\alpha<2,
\end{equation}
and impose \(\kappa_\infty\sigma^2T^2=\pi^2\).  We prove
\begin{equation}\label{eq:intro-alpha-transition}
 \alpha<\frac43
 \ \Longrightarrow\ \mathcal N_{1/2}(T,x)=\infty,
 \qquad
 \alpha>\frac43
 \ \Longrightarrow\ \mathcal N_{1/2}(T,x)<\infty.
\end{equation}
The exact asymptotic in \cref{eq:intro-second-order} can be replaced by
one-sided comparisons for both implications.  At \(\alpha=4/3\), define
\begin{equation}\label{eq:intro-b-star}
 b_*=\frac{3^{4/3}\sigma^{1/3}\kappa_\infty^{5/6}}{2J_{4/3}},
 \qquad
 J_{4/3}=\int_0^\pi\sin^{4/3}u\dd u.
\end{equation}
We prove divergence for \(b<b_*\) and finiteness for \(b>b_*\).  At
\(b=b_*\), the next critical power is \(10/9\), with a second explicit
coefficient threshold.  Signed regularly varying corrections above that
power are decided by their sign, while the exact two-term tail is divergent.
The mechanism is a balance between the one-dimensional Legendre growth of
the critical bridge mode and quartic terminal confinement.  This balance
also explains why two-term data at coefficient equality are insufficient:
the next stated term can reverse the answer.

Our second main result makes this hierarchy complete on the symmetric finite
pure-power class \(\mathcal P_{\mathrm{sym}}(\kappa_\infty)\) of
\cref{def:Psym}.  Define
\[
 \beta_0=\frac43,\qquad
 \beta_{n+1}=\frac{\beta_n+2}{3},
 \qquad
 \beta_n=1+3^{-(n+1)}.
\]
A quantitative multiscale bridge expansion shows that the endpoint term at
power \(\beta_n+2\) is matched by the next bridge term at power
\(3\beta_{n+1}\).  The resulting coefficient recursion starts with
\(c_0^*=-b_*\) and satisfies \(c_{n+1}^*/c_n^*\to1/2\).  More than a local
continuation rule is proved: \cref{thm:Psym-complete} classifies every member
of \(\mathcal P_{\mathrm{sym}}(\kappa_\infty)\).  The decision procedure
terminates for each fixed finite expansion, but its depth has no uniform
bound.  For every \(N\), \cref{cor:arbitrary-depth-pairs} constructs two
tails with the same first \(N+1\) critical terms and opposite moment
outcomes.  Thus the equality theory is complete on the stated class, yet no
fixed finite amount of asymptotic data classifies the whole class.
Equivalently, on the spectral equality surface the entire recursive hierarchy
decides whether the endpoint \(1/2\) belongs to the moment domain
\(\mathcal A_T(x)\).

Several companion results complete the phase diagram.  On the steep-drift
class \(\mathcal D_\star\), all off-critical coefficients are decided by
their position relative to \(a=1/2\).  Subquadratic and windowed
superquadratic reversion give the two extreme critical regimes, while sinh
reversion has a coefficient threshold rather than a horizon threshold.  The
density \(Z\) in
\cref{eq:intro-Z} remains a true martingale throughout, but on
\(\mathcal D_\star\) it has no higher-order moment.  In the
opposite direction, \(D_T=\d\mathbb P^x/\d \mathbb Q_0\) is essentially bounded
under Brownian motion.  The latter two conclusions actually hold under the
weaker hypotheses stated in
\cref{thm:no-higher-moments,prop:forward-density-moments};
\(\mathcal D_\star\) is their principal special case.

The closest historical question is due to \citet{stummer1993}, who asked how
Novikov integrability changes when its coefficient \(1/2\) is replaced by
\(1/2\pm\varepsilon\).  On the steep mean-reverting class
\(\mathcal D_\star\), \cref{cor:critical-coefficient} gives a sharp pointwise
answer: every coefficient below \(1/2\) is finite on all horizons, every
coefficient above \(1/2\) is infinite on every positive horizon, and only
\(a=1/2\) retains a nontrivial horizon problem.  Classical Novikov,
Kazamaki, and Bene\v{s}-type criteria instead give sufficient conditions for
a stochastic exponential to be a true martingale
\citep{novikov,kazamaki,benes1971,klebanerliptser2014}.  Diffusion-specific
tests can decide that martingale question directly \citep{mu,ruf,larssonruf2019}.
They concern preservation of mass, whereas we determine the pointwise
finiteness set of a particular exponential moment.  The latter property is
strictly stronger than martingality and is not implied by it.

Integral and Kac functionals of one-dimensional diffusions connect additive
functionals to Schr\"odinger equations and provide convergence, divergence,
and moment criteria \citep{musiela1985,musiela1986,stummersturm2000}.  The
present problem is not merely the finiteness of a prescribed additive
functional.  The localised Girsanov--It\^o identity couples the running
potential \(q/2\) to the terminal factor
\(\exp\{\Lambda(X_T)/\sigma^2\}\).  At spectral equality, the competition
between this endpoint confinement and the shifted first bridge mode produces
the \(4/3\), \(10/9\), and recursively generated lower-order transitions.
Thus the general theory supplies the framework, while the sharp equality
asymptotics require the coupled endpoint analysis developed here.

Classical gauge and conditional-gauge theory relates Feynman--Kac finiteness
to Schr\"odinger Green functions and to critical or subcritical operators
\citep{zhao1986,chungzhao1995,chensong2002,takeda2002}.  The word
``critical'' below has a different, finite-horizon meaning: it refers to
moment and bridge-spectral transition surfaces, not to operator criticality
in the potential-theoretic sense.  The bridge spectrum is classical and
identifies the codimension-one surface
\(\kappa_\infty\sigma^2T^2=\pi^2\); it does not decide what happens on that
surface.  Equality is governed by the lower-order tail and the shifted
endpoint saddle.  Unlike the usual long-time spectral analysis of
Feynman--Kac semigroups, the transition here occurs at a finite horizon.

Sharp asymptotics for Gaussian path functionals form a second adjacent
literature.  Banach-space Laplace methods give exact \(L^p\)-large-deviation
asymptotics for Brownian motion and Brownian bridges \citep{fatalov2003},
while spectral and Laplace-transform methods yield sharp results for
quadratic functionals and \(L^2\)-small balls of integrated Brownian motion
\citep{chenli2003,gaohannigtorcaso2003}.  Those results analyse Gaussian
functionals under a fixed Gaussian law.  Here the relevant object is an
endpoint-dependent translate of a conditional bridge functional, followed
by integration against a terminal potential at the same asymptotic scale.
On the equality surface, successive cancellations also change the active
power recursively.  The proofs below therefore isolate the critical bridge
mode and control the transverse remainder at the endpoint-integration scale,
rather than applying a single Gaussian tail or small-ball asymptotic.

The same Girsanov--It\^o factorisation at \(a=0\) underlies exact diffusion
simulation \citep{beskosroberts2005,
beskospapaspiliopoulosroberts2006,beskospapaspiliopoulosroberts2008}.
Those algorithms control the killing rate
\(\phi=\lambda^2/(2\sigma^2)-q/2=-\Psi_0\) by local bridge envelopes; they do
not require the global physical-law moment in \cref{eq:intro-functional}.

Moment-explosion times in affine stochastic-volatility models are also
called critical horizons \citep{andersenpiterbarg2007,kellerressel2011}.
Here the functional is non-affine: the first Dirichlet bridge mode selects the
critical time, and terminal confinement governs equality.

The classical ingredients above identify the strict-side spectral boundary.
Our contribution begins on that boundary: the exact cubic equality result,
the \(4/3\) and \(10/9\) transitions and their coefficient thresholds, and
the recursive continuation.  Previous Novikov and additive-functional
criteria do not provide a complete classification on this spectral equality
surface, nor the recursively generated critical tail exponents and
coefficient thresholds.  The resulting decision theory is complete on
\(\mathcal P_{\mathrm{sym}}(\kappa_\infty)\): every finite pure-power germ is
decided after finitely many comparisons, but the required depth is unbounded
over the class.  To our knowledge, neither the exact cubic critical horizon
nor this complete unbounded-depth classification has previously been
identified for the Novikov functional of a nonlinear mean-reverting
diffusion.

The paper is organised as follows.  \Cref{sec:setup} fixes the diffusion
class and the martingality framework.  \Cref{sec:effective} derives the
transfer identity and identifies the critical coefficient.
It also treats exponential reversion in \cref{ex:smr-threshold} and, in
\cref{sec:critical-line}, the subquadratic and superquadratic regimes.
\Cref{sec:quadratic} first proves the exact cubic theorem, then establishes
the strict-side quadratic classification, the equality transitions, the
finite-depth cascade, and the complete classification on
\(\mathcal P_{\mathrm{sym}}(\kappa_\infty)\).
\Cref{sec:martingality} records the moment asymmetry of the two
change-of-measure directions, and
\cref{sec:discussion} summarises the phase diagram.  The Brownian tube
and spectral details are collected in the appendices.
\section{Mean-reverting diffusions and stochastic exponentials}
\label{sec:setup}

Throughout, \((\Omega,\mathcal F,(\mathcal F_t)_{t\ge0},\mathbb P)\)
is a filtered probability space carrying a standard one-dimensional Brownian
motion \(W=(W_t)_{t\ge0}\), with the usual conditions.  We consider
\begin{equation}\label{eq:sde}
  \d X_t=\lambda(X_t)\dd t+\sigma\dd W_t,
  \qquad X_0=x\in\mathbb R,\qquad \sigma>0,
\end{equation}
with generator
\(\mathcal Lf=\lambda f'+\frac{\sigma^2}{2}f''\).

\begin{definition}[Mean-reverting drift]\label{def:mean-reverting}
A function \(\lambda\) belongs to \(\mathcal D\) if
\(\lambda\in C^1(\mathbb R)\), \(\lambda\) is strictly decreasing on
\(\mathbb R\), and \(\lambda(\theta)=0\) for some \(\theta\in\mathbb R\).
The equilibrium \(\theta\) is then unique, and
\begin{equation}\label{eq:q-def}
  q(x):=-\lambda'(x)\ge0
\end{equation}
is the \emph{local reversion rate}.
\end{definition}

\begin{remark}[Strict monotonicity rather than a pointwise sign condition]
\label{rem:D-monotone}
For \(\lambda\in C^1\), strict monotonicity is equivalent to \(q\ge0\)
together with the requirement that \(\{q=0\}\) contain no interval.  It is
therefore strictly weaker than the pointwise condition \(\lambda'<0\), and it
is the condition used below: \cref{lem:wellposed} needs only
\((x-\theta)\lambda(x)<0\), \cref{lem:Lambda-concave} only that
\(\Lambda'=\lambda\) is strictly decreasing, and \cref{thm:Z-martingale}
imports conditions on the diffusion coefficients rather than on the sign of
\(\lambda'\).  In particular, it includes the pure cubic drift
\(\lambda(y)=-cy^3\).  Here \(q(y)=3cy^2\), so \(q(0)=0\) and
\(\lambda'<0\) fails at the single point \(y=0\), while \(\lambda\) is
strictly decreasing and hence lies in \(\mathcal D\).  The same applies to every drift
\(-c\operatorname{sgn}(y)|y|^p\) with \(p>1\), and in particular to the
critical family of \cref{sec:quadratic}.
\end{remark}

Membership of \(\mathcal D\) implies
\((x-\theta)\lambda(x)<0\) for \(x\ne\theta\): the drift always points
towards equilibrium.  It also provides global well-posedness without a
linear-growth assumption.

\begin{remark}[Global well-posedness]\label{lem:wellposed}
Every \(\lambda\in\mathcal D\) gives a pathwise unique, non-explosive
global strong solution of \cref{eq:sde}.  Indeed, local Lipschitz continuity
gives uniqueness up to explosion, while for
\(V(x)=1+(x-\theta)^2\),
\[
 \mathcal LV(x)=2(x-\theta)\lambda(x)+\sigma^2
 \le \sigma^2\le \sigma^2V(x).
\]
Khasminskii's test therefore excludes explosion
\citep[Theorem~3.5]{kh}; see also \citet[Chapter~5]{ks}.
\end{remark}

We shall also use the following consequence of
\cref{def:mean-reverting}.  It controls the terminal contribution in the
transfer identity without an additional growth assumption.

\begin{lemma}[Concavity of the potential]\label{lem:Lambda-concave}
Let \(\lambda\in\mathcal D\) and let \(\Lambda\) be an antiderivative of
\(\lambda\).  Then \(\Lambda\) is strictly concave, attains its maximum at
\(\theta\), and
\begin{equation}\label{eq:Lambda-bounded}
 \Lambda(y)\le\Lambda(\theta)\qquad\text{for every }y\in\mathbb R.
\end{equation}
\end{lemma}

\begin{proof}
The derivative \(\Lambda'=\lambda\) is strictly decreasing by
\cref{def:mean-reverting}, and for a differentiable function that is exactly
strict concavity; that \(\Lambda''=-q\) may vanish at some points, as it does
at the origin for the pure cubic drift, is immaterial.  Since
\(\Lambda'(\theta)=\lambda(\theta)=0\) and \(\Lambda'\) is strictly
decreasing, \(\theta\) is the unique maximiser.
\end{proof}

Membership of \(\mathcal D\) alone therefore forces
\(\Lambda(X_T)/\sigma^2\) to be bounded above, uniformly in the terminal
value and with no growth hypothesis on \(\lambda\).  Every sufficiency
argument below reduces accordingly to controlling a running potential.  The
one place where \cref{eq:Lambda-bounded} is too crude is the critical scale
itself, where the terminal factor is precisely what pins the transferred
path; see \cref{rem:spectral} and the finiteness half of
\cref{thm:universality}.

Let \(\eta=\lambda/\sigma\), and for \(0\le t\le s\le T\) define
\begin{equation}\label{eq:Z}
 Z_{t,s}:=\exp\bigg(
   -\int_t^s\eta(X_r)\dd W_r
   -\frac12\int_t^s\eta(X_r)^2\dd r
 \bigg).
\end{equation}
This nonnegative local martingale is the stochastic exponential that removes
the drift from \cref{eq:sde}.  Its martingale property will be used in the
exact calculations below, so we establish it independently of every
exponential-moment estimate.

\begin{theorem}[Martingality of the drift-removing density]
\label{thm:Z-martingale}
Let \(\lambda\in\mathcal D\).  For every starting pair \((t,x)\) and
finite \(T\), \((Z_{t,s})_{s\in[t,T]}\) is a true
\(\mathbb P^{t,x}\)-martingale, \(\mathbb E^{t,x}[Z_{t,T}]=1\), and
\(Z_{t,T}>0\) almost surely.  Consequently
\[
 \frac{\d\mathbb Q^{t,x}}{\d\mathbb P^{t,x}}
 \big|_{\mathcal F_s}=Z_{t,s}
\]
defines a probability measure equivalent to \(\mathbb P^{t,x}\) on
\(\mathcal F_T\), and under \(\mathbb Q^{t,x}\),
\begin{equation}\label{eq:Q-dynamics}
  \d X_s=\sigma\dd W_s^{\mathbb Q},
  \qquad X_s=x+\sigma W_s^{\mathbb Q},
\end{equation}
where \(W_t^{\mathbb Q}=0\).
\end{theorem}

\begin{proof}
Apply \citet[Theorem~2.1 and Corollary~2.2]{mu} on the state space
\(J=\mathbb R\), with diffusion coefficient \(\sigma\) and kernel
\(b=-\lambda/\sigma\).  The local integrability hypotheses hold because
\(\lambda\in C^1\) and \(\sigma>0\); the original diffusion is non-explosive
by \cref{lem:wellposed}, while the auxiliary drift
\(\lambda+\sigma b\) vanishes, so the auxiliary diffusion is non-explosive
Brownian motion.  The cited criterion gives martingality, positivity follows
pathwise, and Girsanov's theorem gives \cref{eq:Q-dynamics}.
\end{proof}

\begin{remark}[Why the theorem does not use Novikov]
\label{rem:why-not-novikov}
The auxiliary drift vanishes by construction.  Its non-explosion, rather than
an exponential moment under \(\mathbb P\), proves martingality.  Thus
\cref{thm:Z-martingale} remains valid in every regime in which the Novikov and
Kazamaki expectations below diverge.  The two non-explosion arguments concern
different processes.  \Cref{lem:wellposed} applies Khasminskii's test to
\(X\), whereas the criterion above applies to the auxiliary diffusion of the
measure change.  For the general
relationship between explosion tests of this kind and the failure of the
martingale property, see \citet{dandapaniprotter2022}.
\end{remark}

We next formulate the two tail conditions used later.  The first is the
windowed superquadratic condition that drives the obstruction.

\begin{definition}[Supercritical reversion]\label{def:supercritical}
For \(R>0\), set
\begin{equation}\label{eq:q-star}
 q_*^{\pm}(R):=\inf\{q(y):|\,\pm(y-\theta)-R\,|\le1\}.
\end{equation}
A drift \(\lambda\in\mathcal D\) is \emph{supercritical} if there is a
sequence \(R_n\uparrow\infty\) such that
\begin{equation}\label{eq:supercritical}
 \frac{q_*^+(R_n)}{R_n^2}\longrightarrow\infty
 \quad\text{or}\quad
 \frac{q_*^-(R_n)}{R_n^2}\longrightarrow\infty.
\end{equation}
\end{definition}

The fixed half-width \(1\) in \cref{eq:q-star} is part of the hypothesis,
not an invariant choice: replacing it by another prescribed positive width
gives an analogous sufficient condition with different tube constants.
What is essential is control on a window of non-vanishing width; a potential
concentrated on arbitrarily thin spikes need not be occupied long enough to
affect the exponential functional.  For eventually monotone or regularly
varying \(q\), the choice of fixed width is immaterial and the condition
reduces to superquadratic growth in at least one tail.

The second condition delimits the class on which the critical coefficient is
determined.  It is stated in terms of a single scalar attached to the drift,
which is introduced first because \cref{thm:subcritical-coefficient} uses it
at values other than \(0\).

\begin{definition}[Steeply reverting drifts]\label{def:Dstar}
For \(\lambda\in\mathcal D\) and \(\sigma>0\), set
\begin{equation}\label{eq:rho}
 \rho_\sigma(\lambda)
 :=\limsup_{|y|\to\infty}\frac{\sigma^2q(y)}{\lambda(y)^2}
 \ \in[0,\infty].
\end{equation}
A drift \(\lambda\in\mathcal D\) is
\begin{enumerate}[label=\textup{(\alph*)}]
\item \emph{superlinear} if
 \(|\lambda(y)|/|y-\theta|\to\infty\) as \(|y-\theta|\to\infty\);
\item \emph{drift-dominated} if \(\rho_\sigma(\lambda)=0\), that is, if the squared
 drift-to-volatility ratio \(\lambda^2/\sigma^2\) eventually dominates the
 local reversion rate.
\end{enumerate}
We write \(\mathcal D_\star\) for the class of drifts in \(\mathcal D\) that
are both superlinear and drift-dominated.
\end{definition}

The ratio in \cref{eq:rho} is well defined away from \(\theta\) because
\(\lambda\) vanishes only there, and it is nonnegative because \(q\ge0\).
The two conditions of \cref{def:Dstar} are logically
independent: superlinearity constrains \(\lambda\) and says nothing about
\(\lambda'\), so one may construct \(\lambda\in\mathcal D\) with
\(|\lambda(y)|\asymp y^2\) whose derivative carries spikes of height
\(|y|^5\) on intervals of width \(|y|^{-4}\), which is superlinear but has
\(\rho_\sigma(\lambda)=\infty\).  Both conditions are used, and for different halves
of \cref{cor:critical-coefficient}.

\begin{example}[Polynomial and sinh drifts]\label{ex:basic-drifts}
For \(\lambda(\theta+y)=-c\operatorname{sgn}(y)|y|^p\) with \(c>0\) and
\(p\ge1\) one has \(\lambda\in C^1\) and
\(q(\theta+y)=cp|y|^{p-1}\), which vanishes at \(y=0\) when \(p>1\);
by \cref{rem:D-monotone} the drift lies in \(\mathcal D\) all the same, and
no smoothing near the origin is required.  The drift is
supercritical exactly when \(p>3\); cubic reversion is the boundary.  It is
superlinear exactly when \(p>1\), and always drift-dominated, since
\(\sigma^2q/\lambda^2\asymp(p\sigma^2/c)|y|^{-p-1}\to0\); hence
\(\lambda\in\mathcal D_\star\) exactly when \(p>1\).  For
\(\lambda(x)=\mu\sinh(\theta-x)\) one has \(q(x)=\mu\cosh(\theta-x)\), so
\(q\) is supercritical by a wide margin, and
\(\sigma^2q/\lambda^2=(\sigma^2/\mu)\cosh/\sinh^2\to0\), so
\(\lambda\in\mathcal D_\star\).
\end{example}

For reference, the symbols that recur across the phase diagram are collected
in \cref{tab:notation}.  Constants denoted by \(C,c\), with or without
subscripts, are local proof constants and may change from line to line.

\begin{table}[htbp]
\centering
\small
\begin{tabular}{@{}ll@{}}
\toprule
Symbol & Meaning \\
\midrule
\(\lambda,\theta,\sigma\) & drift, equilibrium, and constant volatility \\
\(q=-\lambda'\), \(\Lambda'=\lambda\) & local reversion rate and drift potential \\
\(\mathcal D,\mathcal D_\star\) & mean-reverting and steeply reverting drift classes \\
\(\mathcal P_{\mathrm{sym}}(\kappa_\infty)\) & symmetric finite pure-power rate class in \cref{def:Psym} \\
\(\mathcal N_a(T,x)\) & exponential moment in \cref{eq:intro-functional} \\
\(\mathcal A_T(x)\) & exponential-moment domain in \cref{eq:intro-moment-domain} \\
\(Z_{t,T}\), \(D_T=Z_{0,T}^{-1}\) & the two change-of-measure densities \\
\(\kappa_\infty\) & limiting quadratic coefficient of \(q\) \\
\(b_*\) & critical \(|y|^{4/3}\)-coefficient in \cref{eq:b-star} \\\vspace{0.1cm}
\(J_\alpha=\int_0^\pi\sin^\alpha u\dd u\) & sine-power integral used in the bridge constants\\
\(I_p=\int_0^T e_1(s)^p\dd s\) & first-mode power integral \\
\(\mathfrak L(d)\) & critical \(4/3\)-bridge rate in \cref{eq:sharp-bridge-limit} \\
\(\beta_n,c_n^*\) & pure-power cascade exponents and critical coefficients \\
\(\delta_{\mathrm{cp}}(q),\delta_{\mathrm{dec}}(q)\) & critical-prefix and decision depths \\
\bottomrule
\end{tabular}
\caption{Principal notation.  One-use proof constants are intentionally
excluded.}
\label{tab:notation}
\end{table}

\section{The effective potential and the critical coefficient}
\label{sec:effective}

\subsection{Localised transfer}\label{sec:transfer}

Fix an antiderivative \(\Lambda\) of \(\lambda\), so
\begin{equation}\label{eq:Lambda}
  \Lambda'=\lambda,
  \qquad \Lambda''=\lambda'=-q.
\end{equation}
Let \(\mathbb Q_0\) denote the law on \(C([0,T])\) of
\(x+\sigma W\), and retain \(X\) for the coordinate process.  Here and below,
\(W\) denotes a Brownian motion under whichever measure is in force; where two
measures appear in the same statement, the Brownian motion of the changed
measure carries that measure as a superscript, as in \cref{eq:Q-dynamics}.

\begin{lemma}[Localised transfer and residual potential]
\label{lem:local-transfer}
Let \(n\in\mathbb N\) and \(A\in\mathcal F_T\) satisfy
\(A\subseteq\{\sup_{s\le T}|X_s|\le n\}\).  For every
\(\mathcal F_T\)-measurable \(F\) such that \(F\mathbf1_A\) is bounded,
\begin{multline}\label{eq:transfer}
 \mathbb E^{\mathbb P}\big[e^F\mathbf1_A\big]
 =\mathbb E^{\mathbb Q_0}\bigg[\exp\bigg(
  F+\frac{\Lambda(X_T)-\Lambda(x)}{\sigma^2}
  +\frac12\int_0^Tq(X_s)\dd s
  -\frac1{2\sigma^2}\int_0^T\lambda(X_s)^2\dd s
 \bigg)\mathbf1_A\bigg].
\end{multline}
Moreover, with
\(W_s=\sigma^{-1}(X_s-x-\int_0^s\lambda(X_r)\dd r)\),
\begin{equation}\label{eq:pathwise-identity}
 \int_0^T\eta(X_s)\dd W_s
 =\frac1{\sigma^2}\int_0^T\lambda(X_s)\dd X_s
  -\frac1{\sigma^2}\int_0^T\lambda(X_s)^2\dd s.
\end{equation}
\end{lemma}

\begin{proof}
Let \(\tau_{n+1}:=\inf\{s:|X_s|\ge n+1\}\).  On
\(\mathcal F_{T\wedge\tau_{n+1}}\), the stopped kernel \(\eta(X)\) is
bounded, so Novikov's condition holds for the stopped exponentials and
Girsanov gives mutual absolute continuity with density
\[
 \frac{\d\mathbb P}{\d \mathbb Q_0}\Big|_{\mathcal F_{T\wedge\tau_{n+1}}}
 =\exp\bigg(
  \frac1{\sigma^2}\int_0^{T\wedge\tau_{n+1}}\lambda(X_s)\dd X_s
  -\frac1{2\sigma^2}\int_0^{T\wedge\tau_{n+1}}\lambda(X_s)^2\dd s
 \bigg).
\]
On \(A\), the stopping time exceeds \(T\).  Under \(\mathbb Q_0\), It\^o's
formula and \cref{eq:Lambda} give
\[
 \frac1{\sigma^2}\int_0^T\lambda(X_s)\dd X_s
 =\frac{\Lambda(X_T)-\Lambda(x)}{\sigma^2}
  +\frac12\int_0^Tq(X_s)\dd s,
\]
which proves \cref{eq:transfer}.  Identity
\cref{eq:pathwise-identity} follows by substituting
\(\d W_s=\sigma^{-1}(\d X_s-\lambda(X_s)\dd s)\).
\end{proof}

\begin{definition}[Effective potential]\label{def:effective-potential}
For \(a\ge0\) set
\begin{equation}\label{eq:effective-potential}
 \Psi_a(y):=\frac{q(y)}{2}
  +\bigg(a-\frac12\bigg)\frac{\lambda(y)^2}{\sigma^2}.
\end{equation}
\end{definition}

\begin{corollary}[Extended-valued transfer identity]
\label{cor:extended-transfer}
For every \(a\ge0\), \(x\in\mathbb R\) and \(T>0\),
\begin{equation}\label{eq:transfer-Psi}
 \mathcal N_a(T,x)
 =e^{-\Lambda(x)/\sigma^2}\,
  \mathbb E^{\mathbb Q_0}\bigg[
   \exp\bigg\{\frac{\Lambda(X_T)}{\sigma^2}
    +\int_0^T\Psi_a(X_s)\dd s\bigg\}\bigg],
 \qquad X_s=x+\sigma W_s,
\end{equation}
as an identity in \([0,\infty]\).  In particular,
\cref{eq:transfer-Psi} does not presuppose that the drift-removing stochastic
exponential is a true martingale.
\end{corollary}

\begin{proof}
Let
\[
 A_n:=\big\{\sup_{0\le s\le T}|X_s|\le n\big\},
 \qquad
 F:=a\int_0^T\eta(X_s)^2\dd s.
\]
Continuity of \(\eta\) makes \(F\mathbf1_{A_n}\) bounded, so
\cref{lem:local-transfer} applies.  Under both \(\mathbb P\) and \(\mathbb Q_0\), the
coordinate process is continuous and non-explosive on \([0,T]\); hence
\(A_n\uparrow\Omega\) almost surely under both measures.  The two localised
integrands are nonnegative, and monotone convergence therefore gives
\cref{eq:transfer-Psi}, allowing either side to equal \(+\infty\).
\end{proof}

Only \(q\) and \(\Lambda\) appear on the right of
\cref{eq:transfer-Psi}, and \(\Lambda\) is determined by \(\lambda\) up to a
constant that cancels.  This is why the classification of
\cref{sec:critical-line} can be stated in terms of the growth of \(q\) alone,
and why the strict-side spectral surface of \cref{thm:universality} depends
on \(\lambda\) only through \(\lim q(\theta+y)/y^2\).

\subsection{The critical coefficient}\label{sec:critical-coefficient}

\begin{theorem}[Subcritical coefficients]\label{thm:subcritical-coefficient}
Let \(\lambda\in\mathcal D\) and let \(\rho_\sigma(\lambda)\in[0,\infty]\) be as in
\cref{eq:rho}.  If
\(a<\big(1-\rho_\sigma(\lambda)\big)/2\), then \(\Psi_a\) is bounded above on
\(\mathbb R\) and
\begin{equation}\label{eq:subcritical-bound}
 \mathcal N_a(T,x)\le
 \exp\bigg\{\frac{\Lambda(\theta)-\Lambda(x)}{\sigma^2}
  +T\sup_{y\in\mathbb R}\Psi_a(y)\bigg\}<\infty
\end{equation}
for every \(T>0\) and \(x\in\mathbb R\).  The hypothesis is vacuous unless
\(\rho_\sigma(\lambda)<1\); in particular, if \(\lambda\) is
drift-dominated then \(\mathcal N_a(T,x)<\infty\) for every \(a<1/2\), every
horizon and every initial state.
\end{theorem}

\begin{proof}
The hypothesis \(a<\big(1-\rho_\sigma(\lambda)\big)/2\) with \(a\ge0\) forces
\(\rho_\sigma(\lambda)<1-2a\le1\), and in particular \(\rho_\sigma(\lambda)<\infty\).
Choose \(\rho'\) with \(\rho_\sigma(\lambda)<\rho'<1-2a\).  By \cref{eq:rho} there is
\(R>0\) with
\(\sigma^2q(y)\le\rho'\lambda(y)^2\) whenever \(|y-\theta|\ge R\).  For such
\(y\),
\[
 \Psi_a(y)
 =\frac{q(y)}{2}-\bigg(\frac12-a\bigg)\frac{\lambda(y)^2}{\sigma^2}
 \le\frac{\lambda(y)^2}{\sigma^2}
   \bigg[\frac{\rho'}{2}-\frac{1-2a}{2}\bigg]<0 .
\]
On the compact set \(\{|y-\theta|\le R\}\) the function \(\Psi_a\) is
continuous, hence bounded, so \(K:=\sup_{\mathbb R}\Psi_a<\infty\).
Inserting \(K\) and \cref{eq:Lambda-bounded} into \cref{eq:transfer-Psi}
gives \cref{eq:subcritical-bound}.  Drift-domination is the case
\(\rho_\sigma(\lambda)=0\).
\end{proof}

No moment estimate under \(\mathbb P\), no spectral input and no growth
hypothesis on \(\lambda\) enter \cref{thm:subcritical-coefficient}: once
\(\sup\Psi_a\) is finite the bound is deterministic.  The matching direction
uses superlinearity instead, through the excursion of
\cref{lem:excursion}.

\begin{theorem}[Supercritical coefficients]
\label{thm:supercritical-coefficient}
Let \(\lambda\in\mathcal D\) be superlinear.  Then
\(\mathcal N_a(T,x)=\infty\) for every \(a>1/2\), every \(x\in\mathbb R\) and
every \(T>0\).
\end{theorem}

\begin{proof}
Assume without loss of generality that superlinearity is realised in the
right tail.  Let \(A_R\) be the excursion event of \cref{lem:excursion}.  On
\(A_R\) we have \(|X_s-\theta-R|<1\) for \(s\in[T/3,2T/3]\), and since
\(|\lambda|\) is increasing away from \(\theta\),
\[
 |\lambda(X_s)|\ge\lambda_*(R):=\inf_{|y-\theta-R|\le1}|\lambda(y)|
 =|\lambda(\theta+R-1)| .
\]
Because \(q\ge0\), the running term in \cref{eq:transfer-Psi} obeys, on
\(A_R\),
\[
 \int_0^T\Psi_a(X_s)\dd s
 \ge\bigg(a-\frac12\bigg)\frac1{\sigma^2}
   \int_{T/3}^{2T/3}\lambda(X_s)^2\dd s
 \ge\frac{(2a-1)T}{6\sigma^2}\,\lambda_*(R)^2 .
\]
Restricting \cref{eq:transfer-Psi} to \(A_R\) and using the two estimates of
\cref{lem:excursion},
\[
 \mathcal N_a(T,x)\ge p_{\mathrm{tube}}
 \exp\bigg\{\frac{(2a-1)T}{6\sigma^2}\lambda_*(R)^2
  -\frac{12R^2+24R}{\sigma^2T}
  -\frac{C_\Lambda}{\sigma^2}\bigg\}.
\]
Superlinearity gives
\(\lambda_*(R)/R=|\lambda(\theta+R-1)|/R\to\infty\), so for
fixed \(a>1/2\) and fixed \(T>0\) the exponent tends to \(+\infty\) as
\(R\to\infty\).
\end{proof}

\begin{corollary}[Critical coefficient]\label{cor:critical-coefficient}
Let \(\lambda\in\mathcal D_\star\).  Then for every \(x\in\mathbb R\),
\[
 \mathcal N_a(T,x)<\infty\ \text{ for all }T>0\quad\text{if }a<\frac12,
 \qquad
 \mathcal N_a(T,x)=\infty\ \text{ for all }T>0\quad\text{if }a>\frac12 .
\]
The coefficient \(a=1/2\) is therefore the unique value at which
\(\mathcal N_a(\cdot,x)\) can exhibit horizon dependence.
\end{corollary}

\begin{proof}
Drift-domination and \cref{thm:subcritical-coefficient} give the first
assertion; superlinearity and \cref{thm:supercritical-coefficient} give the
second.
\end{proof}

On \(\mathcal D_\star\), Novikov's coefficient is the only value not already
decided by the preceding estimates; see
\cref{cor:critical-coefficient}.  The remaining analysis is therefore
concentrated on that line in \cref{sec:critical-line,sec:quadratic}.

\begin{remark}[Both hypotheses are used, and neither is removable]
\label{rem:OU-sharpness}
The two halves of \cref{cor:critical-coefficient} use different hypotheses.
The Ornstein--Uhlenbeck drift is drift-dominated but not superlinear, and for
\(a>1/2\) retains a classical horizon-dependent Gaussian-quadratic boundary
\citep{cameronmartin1944}.  Conversely, the spiked construction following
\cref{def:Dstar} is superlinear but not drift-dominated, so
\cref{thm:subcritical-coefficient} gives no positive-coefficient conclusion.
\end{remark}

\begin{figure}[!ht]
\centering
\begin{tikzpicture}[x=1cm,y=1cm,scale=0.9]
\foreach \i/\lab/\type in {0/{$p=1$ (OU)}/A, 1/{$1<p<3$}/B, 2/{$p=3$}/C,
                           3/{$p>3$}/D}{
  \begin{scope}[xshift=\i*3.75cm]
    \node[font=\small] at (1.45,3.28) {\lab};
    \begin{scope}
      \clip (-0.15,-0.15) rectangle (2.95,2.95);
      \if\type A
        \fill[gray!18] (0,0) rectangle (1.4,2.9);
        \fill[gray!18] plot[domain=1.42:2.9,samples=120,smooth]
          (\x,{min(2.9,0.62/sqrt(\x-1.4))}) -- (2.9,0) -- (1.4,0) -- cycle;
        \draw[very thick,domain=1.4477:2.9,samples=120,smooth]
          plot (\x,{0.62/sqrt(\x-1.4)});
      \else
        \fill[gray!18] (0,0) rectangle (1.4,2.9);
      \fi
      \if\type B \draw[very thick] (1.4,0) -- (1.4,2.9); \fi
      \if\type D \draw[very thick,gray!45] (1.4,0) -- (1.4,2.9); \fi
      \if\type C
        \draw[very thick,gray!45] (1.4,1.15) -- (1.4,2.9);
        \draw[very thick] (1.4,0) -- (1.4,1.15);
        \fill[white,draw=black,thick] (1.4,1.15) circle (2.2pt);
      \fi
    \end{scope}
    \draw[->] (0,0) -- (3.05,0) node[right,font=\scriptsize] {$a$};
    \draw[->] (0,0) -- (0,3.02) node[above,font=\scriptsize] {$T$};
    \draw[dashed,gray] (1.4,0) -- (1.4,2.9);
    \node[font=\scriptsize] at (1.4,-0.33) {$\tfrac12$};
    \node[font=\scriptsize] at (0.7,1.5) {finite};
    \if\type A
      \node[font=\scriptsize] at (2.3,2.1) {$\infty$};
    \else
      \node[font=\scriptsize] at (2.2,1.5) {$\infty$};
    \fi
    \if\type C
      \node[font=\scriptsize,anchor=west] at (1.52,1.15) {$T_c$};
    \fi
  \end{scope}
}
\end{tikzpicture}
\caption{Finiteness of \(\mathcal N_a(T,x)\) for
\(\lambda(\theta+y)\asymp-c\operatorname{sgn}(y)|y|^p\); shaded means finite.
For \(p>1\), horizon dependence is confined to \(a=1/2\): the critical line
is finite for \(p<3\), infinite for \(p>3\), and for \(p=3\) splits at
\(c\sigma^2T_c^2=\pi^2/3\), with divergence at the open circle.  For the
Ornstein--Uhlenbeck case \(p=1\), the black curve in \(a>1/2\) remains
nondegenerate; see \cref{rem:OU-sharpness}.  For cubic reversion and
\(a<1/2\), the explicit bound \cref{eq:subcritical-bound} uses
\(\sup\Psi_a=\sigma\sqrt{c/(1-2a)}\).}
\label{fig:phase-diagram}
\end{figure}
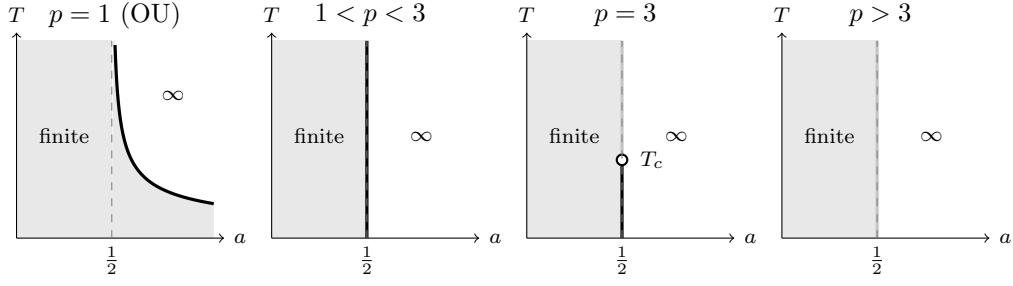

\subsection{Critical behaviour at \texorpdfstring{\(a=1/2\)}{a=1/2}}
\label{sec:critical-line}

At \(a=1/2\) the effective potential reduces to
\(\Psi_{1/2}=\frac12q\), and \cref{eq:transfer-Psi} becomes a Feynman--Kac
functional of Brownian motion with potential \(\frac12q\) and terminal weight
\(e^{\Lambda(X_T)/\sigma^2}\).  This section settles the two extreme regimes.

Both this section and \cref{sec:quadratic} reduce a Gaussian quadratic
functional to a product over Karhunen--Lo\`eve coordinates, so we isolate the
one-dimensional factor.

\begin{lemma}[Shifted Gaussian quadratic factor]
\label{lem:gaussian-factor}
For \(\chi>0\), \(\nu>0\), \(m\in\mathbb R\) and
\(\xi\sim\mathcal N(0,\nu)\),
\begin{equation}\label{eq:gaussian-factor}
 \mathbb E\big[e^{\chi(m+\xi)^2}\big]
 =\begin{cases}
   \dfrac{1}{\sqrt{1-2\chi\nu}}\,
     \exp\bigg\{\dfrac{\chi m^2}{1-2\chi\nu}\bigg\},
     & 2\chi\nu<1,\\[3mm]
   +\infty, & 2\chi\nu\ge1 .
  \end{cases}
\end{equation}
\end{lemma}

\begin{proof}
This is the standard scalar Gaussian integral
\citep[Chapter~1]{bogachev1998}: after inserting the density, complete the
square when \(1-2\chi\nu>0\).  If \(1-2\chi\nu<0\) the integrand grows
quadratically, while at equality it is an exponential affine function or a
positive constant and is still non-integrable.  Thus \(2\chi\nu=1\) lies in
the divergent branch, which makes every equality statement below uniform.
\end{proof}

We next record the classical input of this section in the shifted form used
below.

\begin{lemma}[Shifted Cameron--Martin criterion]\label{lem:cameron-martin}
Let \(W\) be a standard Brownian motion, \(\gamma\ge0\), \(u\in\mathbb R\)
and \(T>0\).  The relevant covariance problem has a Dirichlet condition at
\(0\) and a Neumann condition at \(T\), and
\[
 \mathbb E\bigg[
  \exp\bigg\{\frac{\gamma^2}{2}\int_0^T(u+W_s)^2\dd s\bigg\}\bigg]
 <\infty
 \quad\Longleftrightarrow\quad
 \gamma T<\frac{\pi}{2}.
\]
\end{lemma}

\begin{proof}
The spectral-product proof is given in \cref{app:cameron-martin-proof}.
\end{proof}

\begin{remark}[Why analytic continuation is insufficient]
The classical negative-exponent identity
\citep{cameronmartin1944} contains \(\cosh(\gamma T)\).  Formally replacing
\(\gamma\) by an imaginary parameter predicts the first zero of the
corresponding cosine, but analytic continuation alone does not certify that a
positive exponential moment exists before that zero.  The appendix instead
derives finiteness from the nonnegative Karhunen--Lo\`eve product.
\end{remark}

The remaining bulk reward at \(a=1/2\) is \(\frac12\int q\), while keeping a
Brownian path near distance \(R\) for positive time costs only order \(R^2\).
This proves the following general obstruction.

\begin{theorem}[Failure of classical exponential-moment criteria]
\label{thm:obstruction}
Let \(\lambda\in\mathcal D\) be supercritical, let \(X\) solve
\cref{eq:sde}, and put \(\eta=\lambda/\sigma\).  For every
\(x\in\mathbb R\) and \(T>0\):
\begin{enumerate}[label=\textup{(\roman*)}]
\item \emph{Novikov} \citep{novikov}.
\begin{equation}\label{eq:novikov-diverges}
 \mathbb E_x\bigg[
  \exp\bigg\{\frac12\int_0^T\eta(X_s)^2\dd s\bigg\}
 \bigg]=\infty.
\end{equation}
The same holds with \(1/2\) replaced by any \(a\ge1/2\).  It also
holds on every deterministic subinterval from every restart state, so
partitioned refinements of Novikov's condition fail
\citep[Chapter~6]{ls}.
\item \emph{Kazamaki} \citep{kazamaki}.
\begin{equation}\label{eq:kazamaki-diverges}
 \mathbb E_x\bigg[
  \exp\bigg\{-\frac12\int_0^T\eta(X_s)\dd W_s\bigg\}
 \bigg]=\infty.
\end{equation}
Thus Kazamaki's criterion for
\(\mathcal E(-\int\eta(X)\dd W)\) fails already at the deterministic
time \(T\).
\end{enumerate}
\end{theorem}

The proof is deferred to \cref{app:tubes}.  It uses the local transfer lemma,
so it does not presuppose \cref{thm:Z-martingale}.  The theorem concerns the
classical pure Novikov and Kazamaki tests and their partitioned versions; it
does not claim that every generalised martingale criterion must fail.  Mixed
criteria involving both a stochastic exponential and its quadratic variation
may remain applicable \citep{ruf}, as may Bene\v{s}-type conditions
\citep{benes1971,klebanerliptser2014}.

\begin{remark}[Why quadratic \(q\) is the boundary]\label{rem:q-boundary}
Under the Brownian law the residual quantity is the Feynman--Kac functional
\(\mathbb E^{\mathbb Q_0}\big[\exp\big\{\frac12
  \int_0^Tq(x+\sigma W_s)\dd s\big\}\big]\),
up to the terminal antiderivative.  Quadratic \(q\) is the
Cameron--Martin quadratic functional: its expectation has a finite critical
time.  Faster-than-quadratic \(q\) makes the functional infinite for every
positive time, whereas a subquadratic potential is absorbable on every fixed
horizon.  Since a polynomial drift of order \(p\) has
\(q\) of order \(p-1\), cubic restoring drift is the exact boundary.
\end{remark}

\begin{theorem}[Dichotomy away from the critical scale]\label{thm:q-dichotomy}
Let \(\lambda\in\mathcal D\).
\begin{enumerate}[label=\textup{(\roman*)}]
\item \textup{(Subquadratic.)}  If
 \(q(\theta+y)=o(y^2)\) as \(|y|\to\infty\), then
 \(\mathcal N_{1/2}(T,x)<\infty\) for every \(T>0\) and every
 \(x\in\mathbb R\).
\item \textup{(Supercritical.)}  If \(\lambda\) is supercritical in the
 windowed sense of \cref{def:supercritical}, then
 \(\mathcal N_{1/2}(T,x)=\infty\) for every \(T>0\) and every
 \(x\in\mathbb R\).
\end{enumerate}
\end{theorem}

\begin{proof}
(ii) is \cref{thm:obstruction}(i).

(i)  Let \(T>0\) and choose \(\varepsilon>0\) with
\(\varepsilon\sigma^2T^2<\pi^2/4\).  By hypothesis there is
\(C_\varepsilon<\infty\) with
\(q(\theta+y)\le\varepsilon y^2+C_\varepsilon\) for all \(y\).  By
\cref{eq:Lambda-bounded,eq:transfer-Psi} at \(a=1/2\),
\[
 \mathcal N_{1/2}(T,x)\le
 e^{(\Lambda(\theta)-\Lambda(x))/\sigma^2}e^{C_\varepsilon T/2}\,
 \mathbb E\bigg[\exp\bigg\{\frac{\varepsilon}{2}
  \int_0^T(x-\theta+\sigma W_s)^2\dd s\bigg\}\bigg].
\]
Writing the exponent as
\(\frac{\varepsilon\sigma^2}{2}\int_0^T
  \big(\frac{x-\theta}{\sigma}+W_s\big)^2\dd s\)
and applying \cref{lem:cameron-martin} with
\(\gamma=\sigma\sqrt\varepsilon\), the expectation is finite precisely when
\(\sigma\sqrt\varepsilon\,T<\pi/2\), that is
\(\varepsilon\sigma^2T^2<\pi^2/4\), which is how \(\varepsilon\) was chosen.
As \(T\) was arbitrary, (i) follows.
\end{proof}

\begin{corollary}[Polynomial trichotomy]
\label{cor:polynomial-trichotomy}
Let
\[
 \lambda(\theta+y)=-c\operatorname{sgn}(y)|y|^p,
 \qquad c>0,\quad p\ge1.
\]
At Novikov's coefficient, for every initial state \(x\),
\[
 \begin{array}{ccl}
  1\le p<3
   &\Longrightarrow& \mathcal N_{1/2}(T,x)<\infty
      \text{ for every }T>0,\\[1mm]
  p=3
   &\Longrightarrow& \mathcal N_{1/2}(T,x)<\infty
      \Longleftrightarrow c\sigma^2T^2<\pi^2/3,\\[1mm]
  p>3
   &\Longrightarrow& \mathcal N_{1/2}(T,x)=\infty
      \text{ for every }T>0.
 \end{array}
\]
At \(p=3\), equality belongs to the divergent side.
\end{corollary}

\begin{proof}
For \(p<3\), the local reversion rate
\(q(\theta+y)=cp|y|^{p-1}\) is subquadratic, whereas for \(p>3\) it is
windowed superquadratic; apply \cref{thm:q-dichotomy}.  The cubic case is
\cref{thm:cubic-boundary}.
\end{proof}

\begin{remark}[The two hypotheses are not complementary]
\label{rem:not-complementary}
\Cref{thm:q-dichotomy} is a dichotomy between two extremes, not a partition.
Condition (i) is the relation \(q(\theta+y)=o(y^2)\), whereas (ii) is the
windowed condition \cref{eq:supercritical}, which is strictly stronger
than \(\limsup q(\theta+y)/y^2=\infty\): a reversion rate carrying
arbitrarily thin spikes has infinite \(\limsup\) but need not be occupied long
enough for the tube argument to bite.  For eventually monotone or regularly
varying \(q\) the two coincide, and then \cref{thm:q-dichotomy} leaves exactly
the intermediate scale \(\limsup q(\theta+y)/y^2\in(0,\infty)\) undecided.
That scale is the subject of \cref{sec:quadratic}.
\end{remark}

\begin{example}[Exponential reversion]\label{ex:smr-threshold}
Let \(\lambda(x)=\mu\sinh(\theta-x)\), where \(\mu>0\).  Then
\[
 \eta(x)=\frac{\mu}{\sigma}\sinh(\theta-x),\qquad
 q(x)=\mu\cosh(\theta-x),\qquad
 \Lambda(x)=-\mu\cosh(\theta-x).
\]
For every starting state and positive horizon,
\begin{equation}\label{eq:smr-threshold}
 \mathcal N_a(T,x)<\infty\quad\Longleftrightarrow\quad a<\frac12.
\end{equation}
Indeed,
\[
 \Psi_a(x)=\frac{\mu}{2}\cosh(\theta-x)
 -\bigg(\frac12-a\bigg)\frac{\mu^2}{\sigma^2}
   \sinh^2(\theta-x),
\]
which is bounded above when \(a<1/2\), as is \(\Lambda\); hence
\cref{eq:transfer-Psi} gives finiteness.  At and above \(1/2\), the
windowed-supercritical obstruction in \cref{thm:obstruction} gives
divergence.

The conclusion is pointwise in the initial state.  In fact, for every
\(c,T>0\),
\[
 \sup_x\mathbb E\Big[
  e^{c\int_0^T\eta(x+\sigma W_s)^2\dd s}\Big]=\infty.
\]
To see this, fix \(0<\delta\le T\) and restrict to
\(\{\sup_{s\le\delta}|\sigma W_s|\le1\}\), whose probability is positive
and independent of \(x\); the resulting lower bound contains
\(\sup_x\exp\{c\delta\inf_{|y-x|\le1}\eta(y)^2\}=\infty\).
Thus natural uniform additive-functional bounds cannot recover the pointwise
threshold, consistent with the distinction in
\citet{stummer1993,stummer1997,stummersturm2000}.
\end{example}

\section{Exact quadratic criticality}\label{sec:quadratic}

\subsection{The cubic boundary}\label{sec:cubic-boundary}

\begin{theorem}[Exact Novikov boundary for cubic reversion]
\label{thm:cubic-boundary}
Let
\begin{equation}\label{eq:cubic-sde}
 \d Y_t=-cY_t^3\dd t+\sigma\dd W_t,
 \qquad c,\sigma>0,\qquad Y_0=y_0,
\end{equation}
and define
\begin{equation}\label{eq:cubic-functional}
 N_T:=\mathbb E_{y_0}\bigg[
  \exp\bigg\{\frac1{2\sigma^2}\int_0^Tc^2Y_s^6\dd s\bigg\}
 \bigg].
\end{equation}
Then, for every \(y_0\in\mathbb R\),
\begin{equation}\label{eq:cubic-critical}
 N_T<\infty
 \quad\Longleftrightarrow\quad
 c\sigma^2T^2<\frac{\pi^2}{3}.
\end{equation}
At equality, as well as above it, \(N_T=\infty\).  The same equivalence
holds for affine--cubic drift \(-\kappa Y_t-cY_t^3\) with any
\(\kappa\ge0\).
\end{theorem}

\begin{proof}
The drift \(\lambda(y)=-cy^3\) is \(C^1\) and strictly decreasing with
\(\lambda(0)=0\), so \(\lambda\in\mathcal D\) by
\cref{def:mean-reverting,rem:D-monotone}; the same holds for
\(-\kappa y-cy^3\) with \(\kappa\ge0\).  By \cref{thm:Z-martingale}, the
measure removing the drift is therefore well defined.
Under that measure, which we denote by \(\mathbb Q\), the state satisfies
\(Y_s=y_0+\sigma W_s\).  Dividing the Novikov functional by the density
and applying It\^o's formula to \(Y^4\) gives
\begin{equation}\label{eq:cubic-transfer}
 N_T=e^{cy_0^4/(4\sigma^2)}
 \mathbb E^{\mathbb Q}_{y_0}\bigg[
  \exp\bigg\{-\frac{cY_T^4}{4\sigma^2}
       +\frac{3c}{2}\int_0^TY_s^2\dd s\bigg\}
 \bigg].
\end{equation}
Condition on \(Y_T=y\).  The conditional path is its linear interpolation
\(\ell_{y_0,y}\) plus \(\sigma\beta\), where \(\beta\) is a standard
Brownian bridge independent of \(y\).  Its covariance operator on
\(L^2[0,T]\) has kernel
\[
 K(s,r)=s\wedge r-\frac{sr}{T},
\]
eigenfunctions \(e_n(s)=\sqrt{2/T}\sin(n\pi s/T)\), and eigenvalues
\begin{equation}\label{eq:bridge-eigenvalues}
 \nu_n=\frac{T^2}{n^2\pi^2},\qquad n\ge1.
\end{equation}
Indeed, \((Kf)''=-f\) with \((Kf)(0)=(Kf)(T)=0\), so the eigenproblem
reduces to the Dirichlet spectrum of \(-\partial_s^2\).

Set
\[
 m_n:=\bigg\langle\frac{\ell_{y_0,y}}{\sigma},e_n\bigg\rangle,
 \qquad
 \xi_n:=\langle\beta,e_n\rangle,
 \qquad \chi:=\frac32c\sigma^2.
\]
Exactly as in the proof of \cref{lem:cameron-martin}, the \(\xi_n\) are
jointly Gaussian with
\(\operatorname{Cov}(\xi_n,\xi_k)=\langle Ke_n,e_k\rangle
 =\nu_n\delta_{nk}\), hence independent with
\(\xi_n\sim\mathcal N(0,\nu_n)\), and Parseval gives the pathwise identity
\begin{equation}\label{eq:parseval-bridge}
 \frac{3c}{2}\int_0^TY_s^2\dd s
 =\chi\big\|\sigma^{-1}\ell_{y_0,y}+\beta\big\|_{L^2[0,T]}^2
 =\chi\sum_{n\ge1}(m_n+\xi_n)^2 .
\end{equation}
The partial sums \(\chi\sum_{n\le N}(m_n+\xi_n)^2\) are nonnegative and
increase to the right-hand side of \cref{eq:parseval-bridge}, so monotone
convergence and independence turn the conditional bridge expectation into a
limit of finite products,
\begin{equation}\label{eq:bridge-product}
 \mathbb E\bigg[
  \exp\bigg\{\frac{3c}{2}\int_0^TY_s^2\dd s\bigg\}\biggm| Y_T=y\bigg]
 =\lim_{N\to\infty}\prod_{n=1}^N
  \mathbb E\big[\exp\big\{\chi(m_n+\xi_n)^2\big\}\big]
 \ \in(0,\infty] ,
\end{equation}
every factor being at least \(1\).  Each factor is \cref{eq:gaussian-factor}
with \(\nu=\nu_n\), and \((\nu_n)\) is decreasing, so the whole product is
governed by \(2\chi\nu_1=3c\sigma^2\nu_1\), and
\[
 3c\sigma^2\nu_1<1
 \quad\Longleftrightarrow\quad
 c\sigma^2T^2<\frac{\pi^2}{3}.
\]
If \(c\sigma^2T^2\ge\pi^2/3\), then \(2\chi\nu_1\ge1\) and the second branch
of \cref{eq:gaussian-factor} makes the first factor infinite, so
\cref{eq:bridge-product} diverges; this covers equality and strict excess at
once.  The conclusion holds for every endpoint \(y\), so Tonelli's theorem
gives divergence of \cref{eq:cubic-transfer}.

Suppose now that the inequality is strict, and put
\(\delta=1-2\chi\nu_1>0\).  Then
\(1-2\chi\nu_n\ge\delta\) for all \(n\).  Since
\(\sum_n\nu_n=T^2/6<\infty\), the determinant product converges.  Moreover,
\[
 \sum_nm_n^2
 =\sigma^{-2}\|\ell_{y_0,y}\|_{L^2}^2
 \le\frac{T}{\sigma^2}\big(|y_0|+|y|\big)^2.
\]
The conditional bridge expectation is therefore bounded by
\(C\exp\{C_1(1+y^2)\}\), with constants independent of \(y\).  The
endpoint density is Gaussian, and the factor
\(\exp\{-cy^4/(4\sigma^2)\}\) in \cref{eq:cubic-transfer} dominates
this quadratic bound.  Integration over \(y\) proves finiteness.

Adding \(-\kappa Y\) to the drift contributes only
\[
 \frac{\kappa T}{2}
 -\frac{\kappa(Y_T^2-y_0^2)}{2\sigma^2}
\]
to the transferred exponent.  It does not alter the bridge quadratic
coefficient, so the critical value is unchanged.
\end{proof}

\begin{proposition}[Blow-up rate at the cubic critical horizon]
\label{prop:cubic-blowup-rate}
In the pure cubic model of \cref{thm:cubic-boundary}, put
\[
 T_c:=\frac{\pi}{\sigma\sqrt{3c}},\qquad
 \omega:=\sigma\sqrt{3c}.
\]
For every fixed starting state \(y_0\),
\begin{equation}\label{eq:cubic-blowup-log-rate}
 \lim_{T\uparrow T_c}(T_c-T)^2\log N_T
 =\frac{1}{4c\sigma^2}.
\end{equation}
At the equilibrium the sharper equivalent is
\begin{equation}\label{eq:cubic-blowup-equivalent}
 N_T\sim
 \sqrt{\frac{2}{\cos\{\omega(T_c-T)\}}}
 \exp\bigg\{\frac34\cot^2\!\{\omega(T_c-T)\}\bigg\}.
\end{equation}
Thus the divergence is super-exponential in \((T_c-T)^{-1}\), with
\(\log N_T\sim[4c\sigma^2(T_c-T)^2]^{-1}\).
\end{proposition}

\begin{proof}
Write \(a=c/(4\sigma^2)\).  The transfer formula
\cref{eq:cubic-transfer} and the Mehler kernel for the inverted harmonic
oscillator give, for \(T<T_c\),
\begin{equation}
\label{eq:cubic-mehler-integral}
 N_T=e^{ay_0^4}
 \sqrt{\frac{\omega}{2\pi\sigma^2\sin(\omega T)}}
 \int_{\mathbb R}\exp\bigg\{
  -\frac{\omega\{(y_0^2+y^2)\cos(\omega T)-2y_0y\}}
         {2\sigma^2\sin(\omega T)}-ay^4
 \bigg\}\dd y.
\end{equation}
Equivalently, this follows by solving
\(\partial_Tv=(\sigma^2/2)v''+(3c/2)y^2v\) with
\(v(0,y)=e^{-ay^4}\), and then using \(N_T=e^{ay_0^4}v(T,y_0)\).

Let \(\delta=T_c-T\) and \(\varepsilon=\omega\delta\).  When \(y_0=0\),
the integral in \cref{eq:cubic-mehler-integral} is
\[
 \int_{\mathbb R}e^{B_\varepsilon y^2-ay^4}\dd y,
 \qquad
 B_\varepsilon:=\frac{\omega\cot\varepsilon}{2\sigma^2}.
\]
Its two nondegenerate maxima are
\(y=\pm\sqrt{B_\varepsilon/(2a)}\), each has value
\(B_\varepsilon^2/(4a)\), and the second derivative there is
\(-4B_\varepsilon\).  Laplace's method therefore gives
\[
 \int_{\mathbb R}e^{B_\varepsilon y^2-ay^4}\dd y
 \sim\sqrt{\frac{2\pi}{B_\varepsilon}}
       e^{B_\varepsilon^2/(4a)}.
\]
Combining the prefactors and using \(\omega^2=3c\sigma^2\) gives
\cref{eq:cubic-blowup-equivalent}.

For fixed \(y_0\ne0\), the additional linear term in the exponent of
\cref{eq:cubic-mehler-integral} is \(O(\delta^{-1}|y|)\).  The maximising
scale remains \(|y|\asymp\delta^{-1/2}\), so this term and the displacement
of the maximiser contribute \(O(\delta^{-3/2})\) to the logarithm.  The
terms independent of \(y\) are \(O(\delta^{-1})\).  Hence
\[
 \log N_T=\frac{1}{4c\sigma^2\delta^2}
           +O(\delta^{-3/2}),
\]
which proves \cref{eq:cubic-blowup-log-rate}.
\end{proof}

\begin{remark}[Why the constant is \(\pi^2/3\)]\label{rem:spectral}
The quartic terminal factor in \cref{eq:cubic-transfer} suppresses path
families with unbounded terminal values.  The divergent competitors are
therefore pinned at \(T\), and the relevant spectrum is that of
\(-\partial_s^2\) with Dirichlet conditions at both endpoints.  Its first
eigenvalue is \(\pi^2/T^2\).  If one discards the terminal factor, the right
endpoint is free and the classical Brownian quadratic-functional threshold is
\(c\sigma^2T^2=\pi^2/12\).  Retaining the endpoint penalty changes the
boundary condition and moves the threshold by the ratio of the two first
eigenvalues, from \(\pi^2/12\) to \(\pi^2/3\).  The bridge decomposition is
the probabilistic expression of that pinning, and is the conditioning
underlying the conditional-gauge theory of \citet{chungzhao1995}.
Analytically, the residual operator
\(\frac{\sigma^2}{2}\partial_y^2+\frac{3c}{2}y^2\) is the inverted harmonic
oscillator.  Formula \cref{eq:cubic-mehler-integral} is its Mehler kernel,
and \(T_c\) is its first kernel blow-up time against the quartically decaying
terminal datum.  The values \(\pi^2/12\) and \(\pi^2/3\) are therefore the
Dirichlet--Neumann and Dirichlet--Dirichlet versions, respectively, of the
same spectral mechanism.
\end{remark}

\subsection{The quadratic spectral surface away from equality}
\label{sec:universality}

\Cref{thm:cubic-boundary} treats the pure cubic and affine--cubic families.
The next result extracts what survives for a general asymptotically quadratic
reversion rate.  It is deliberately a strict-side theorem: the leading
curvature locates the spectral surface but, as \cref{sec:equality} will show,
does not decide equality.  Write
\begin{equation}\label{eq:kappa-infty}
 \kappa_\infty:=\lim_{|y|\to\infty}\frac{q(\theta+y)}{y^2},
\end{equation}
when the limit exists through both tails; the symbol \(\kappa\) remains
reserved for the affine coefficient of \cref{thm:cubic-boundary}.

\begin{theorem}[Strict-side quadratic classification]\label{thm:universality}
Let \(\lambda\in\mathcal D\) and suppose the limit \cref{eq:kappa-infty}
exists in \((0,\infty)\).  Then for every \(x\in\mathbb R\),
\begin{equation}\label{eq:universal-boundary}
 \kappa_\infty\sigma^2T^2<\pi^2
 \ \Longrightarrow\ \mathcal N_{1/2}(T,x)<\infty,
 \qquad
 \kappa_\infty\sigma^2T^2>\pi^2
 \ \Longrightarrow\ \mathcal N_{1/2}(T,x)=\infty.
\end{equation}
No conclusion at \(\kappa_\infty\sigma^2T^2=\pi^2\) follows from the leading
limit alone; see \cref{sec:equality}.
\end{theorem}

The hypotheses imply
\(\lambda(\theta+y)\sim-\kappa_\infty y^3/3\), and hence
\(\lambda\in\mathcal D_\star\).  Thus \cref{thm:universality} describes the
only coefficient at which this class can retain horizon dependence.

\begin{proof}
Assume \(\theta=0\).  Fix \(\varepsilon\in(0,\kappa_\infty)\).  By
\cref{eq:kappa-infty} and continuity of \(q\) there is
\(C_\varepsilon<\infty\) with
\begin{equation}\label{eq:sandwich}
 \frac12(\kappa_\infty-\varepsilon)y^2-C_\varepsilon
 \le\frac12q(y)
 \le\frac12(\kappa_\infty+\varepsilon)y^2+C_\varepsilon
 \qquad\text{for all }y\in\mathbb R,
\end{equation}
the compact part being absorbed into \(C_\varepsilon\).  Integrating
\cref{eq:kappa-infty} twice gives \(\lambda(y)/y^3\to-\kappa_\infty/3\) and
\begin{equation}\label{eq:Lambda-asymptotic}
 \frac{\Lambda(y)}{y^4}\longrightarrow-\frac{\kappa_\infty}{12},
\end{equation}
so there are \(c_\pm>0\) and \(C'<\infty\) with
\(-c_+y^4-C'\le\Lambda(y)\le-c_-y^4+C'\).

\emph{Finiteness.}  Let \(\kappa_\infty\sigma^2T^2<\pi^2\) and choose
\(\varepsilon\) with \((\kappa_\infty+\varepsilon)\sigma^2T^2<\pi^2\).  By
\cref{eq:sandwich,eq:Lambda-asymptotic} the exponent in
\cref{eq:transfer-Psi} at \(a=1/2\) is at most
\[
 \frac{-c_-X_T^4+C'-\Lambda(x)}{\sigma^2}
 +\frac{\kappa_\infty+\varepsilon}{2}\int_0^TX_s^2\dd s
 +C_\varepsilon T.
\]
This is the functional treated in the proof of \cref{thm:cubic-boundary},
with quadratic coefficient \(\frac12(\kappa_\infty+\varepsilon)\) in place of
\(\frac32c\) and quartic penalty coefficient \(c_-\) in place of \(c/4\).
Conditioning on \(X_T\) and diagonalising the bridge as in
\cref{eq:bridge-eigenvalues}, with the passage to the infinite product
justified as in \cref{eq:parseval-bridge,eq:bridge-product}, the
Karhunen--Lo\`eve product converges because
\((\kappa_\infty+\varepsilon)\sigma^2\nu_1<1\) with \(\nu_1=T^2/\pi^2\); the
conditional expectation is bounded by \(C\exp\{C_1(1+y^2)\}\) uniformly in
the endpoint \(y\), and the endpoint integration converges because
\(c_->0\) makes the quartic factor dominate.

\emph{Divergence.}  Let \(\kappa_\infty\sigma^2T^2>\pi^2\) and choose
\(\varepsilon\) with \((\kappa_\infty-\varepsilon)\sigma^2T^2>\pi^2\).  By
\cref{eq:sandwich} the exponent is at least
\[
 \frac{-c_+X_T^4-C'-\Lambda(x)}{\sigma^2}
 +\frac{\kappa_\infty-\varepsilon}{2}\int_0^TX_s^2\dd s
 -C_\varepsilon T.
\]
Condition on \(X_T=y\).  In Karhunen--Lo\`eve coordinates the first bridge
factor is
\(\mathbb E[\exp\{\chi_\varepsilon(m_1+\beta_1)^2\}]\) with
\(\chi_\varepsilon=\frac12(\kappa_\infty-\varepsilon)\sigma^2\) and
\(\beta_1\sim\mathcal N(0,\nu_1)\).  Since \(2\chi_\varepsilon\nu_1>1\), the
second branch of \cref{eq:gaussian-factor} makes that factor infinite.  The
remaining
factors are positive and the terminal weight
\(\exp\{(-c_+y^4-C')/\sigma^2\}\) is a finite positive constant for each
fixed \(y\).  The conditional expectation is therefore \(+\infty\) for every
endpoint, and Tonelli's theorem concludes.
\end{proof}

\begin{remark}[A one-sided version under \(\limsup\) only]
\label{rem:limsup-version}
The finiteness half of \cref{thm:universality} does not need the limit to
exist.  If \(\bar\kappa:=\limsup_{|y|\to\infty}q(\theta+y)/y^2<\infty\) and
\(\liminf_{|y|\to\infty}(-\Lambda(\theta+y))/|y|^{2+\delta}>0\) for some
\(\delta>0\), the same argument with \(\bar\kappa+\varepsilon\) in place of
\(\kappa_\infty+\varepsilon\) gives \(\mathcal N_{1/2}(T,x)<\infty\) whenever
\(\bar\kappa\sigma^2T^2<\pi^2\).  The extra hypothesis on \(\Lambda\) is
automatic under \cref{eq:kappa-infty} by \cref{eq:Lambda-asymptotic}, but must
be assumed when only the \(\limsup\) is controlled.
\end{remark}

\begin{corollary}[The cubic scale in intrinsic form]
\label{cor:intrinsic-form}
For \(\lambda(y)=-cy^3\) one has \(\kappa_\infty=3c\), so
\cref{eq:universal-boundary} reads \(3c\sigma^2T^2\lessgtr\pi^2\), that is
\(c\sigma^2T^2\lessgtr\pi^2/3\).  The same strict-side classification is
obtained for every \(\lambda\in\mathcal D\) whose reversion rate
satisfies
\begin{equation}\label{eq:cubic-intrinsic}
 q(\theta+y)\sim3cy^2\qquad\text{as }|y|\to\infty .
\end{equation}
This statement alone makes no claim at equality.  In particular,
\cref{eq:cubic-intrinsic} holds for the affine--cubic drift
\(-\kappa y-cy^3\) with \(\kappa\ge0\), and more generally for every
\(\lambda\in\mathcal D\) of the form
\begin{equation}\label{eq:cubic-remainder}
 \lambda(\theta+y)=-cy^3+r(y),
 \qquad r\in C^1(\mathbb R),\qquad r'(y)=o(y^2).
\end{equation}
The pure and affine--cubic cases diverge at equality by
\cref{thm:cubic-boundary}.  Equality for general lower-order corrections is
treated in \cref{thm:alpha-transition,cor:equality-nonuniversal}.
\end{corollary}

\begin{proof}
Under \cref{eq:cubic-remainder}, \(q(\theta+y)=3cy^2-r'(y)=3cy^2+o(y^2)\),
which is \cref{eq:cubic-intrinsic}; the affine--cubic case is
\(r(y)=-\kappa y\), with \(r'\equiv-\kappa\).  \Cref{thm:universality} then
applies with \(\kappa_\infty=3c\).
\end{proof}

\begin{example}[Drift asymptotics do not control the critical datum]
\label{ex:derivative-bumps}
The following construction explains why the principal hypotheses must be
stated through \(q=-\lambda'\), rather than through the size of
\(\lambda\).  The remainder condition \cref{eq:cubic-remainder} is imposed on \(r'\) rather
than on \(r\), and it cannot be relaxed to \(r(y)=O(y^2)\), or even to
\(r(y)=O(y)\).  The classifying quantity in \cref{thm:universality} is
\(q=-\lambda'\), and a \(C^1\) remainder may be uniformly small in value while
carrying large derivatives on short intervals.  Let \(r\in C^1\) with
\(r(0)=0\) have \(-r'\) equal to a smooth bump of height \(n^5\) and width
\(n^{-5}\) centred at \(y=n\) for each \(n\ge1\), and \(r'=0\) elsewhere.
Each bump contributes \(O(1)\) to \(r\), so \(r(y)=O(y)\); yet
\(\lambda(\theta+y)=-cy^3+r(y)\) has \(q(\theta+y)=3cy^2+n^5\) at \(y=n\), so
\(q(\theta+y)/y^2\to\infty\) along the integers and the limit
\cref{eq:kappa-infty} does not exist.  Note that this \(\lambda\) does lie in
\(\mathcal D\), since \(\lambda'=-3cy^2+r'\le0\) vanishes only at \(y=0\);
and it is not supercritical in the windowed sense of
\cref{def:supercritical}, the bumps being too thin, so it falls in the gap
described in \cref{rem:not-complementary}.  Only a hypothesis controlling
\(\lambda'\) transfers to the boundary, which is why
\cref{thm:universality,cor:intrinsic-form} are stated through \(q\): the
polynomial representation of \(\lambda\) is not the intrinsic datum.
\end{example}

\subsection{Equality surface: one-sided comparison and first thresholds}
\label{sec:equality}

The strict-side comparison in \cref{thm:universality} necessarily loses an
\(\varepsilon y^2\) term and therefore cannot be continued to equality.
Exact two-term asymptotics are not needed away from the second-order
borderline: one-sided bounds suffice.  We first retain the soft critical
bridge estimate in a form adapted to those comparisons.

\begin{lemma}[First-mode representation at equality]
\label{lem:first-mode-apparatus}
Fix \(x\in\mathbb R\) and \(T,\sigma,\kappa_\infty>0\) with
\(\kappa_\infty\sigma^2T^2=\pi^2\).  Let
\(\ell_{x,y}(s)=x+(y-x)s/T\), let \(\beta\) be a standard Brownian bridge,
and set
\[
 e_n(s)=\sqrt{\frac2T}\sin\frac{n\pi s}{T},\qquad
 \nu_n=\frac{T^2}{n^2\pi^2}.
\]
Write \(e=e_1\),
\[
 \beta=ze+\beta^\perp,\qquad
 \ell_{x,y}=l_1e+\ell^\perp,\qquad
 u=l_1+\sigma z,\qquad
 g=\ell^\perp+\sigma\beta^\perp.
\]
Then \(z\sim\mathcal N(0,\nu_1)\) is independent of \(\beta^\perp\), and
the symmetry \(e_1(T-s)=e_1(s)\) gives the midpoint identity
\begin{equation}\label{eq:first-mode-midpoint}
 \langle\ell_{x,y},e_1\rangle
 =\ell_{x,y}(T/2)\langle1,e_1\rangle
 =\frac{x+y}{2}I_1,
 \qquad I_1:=\int_0^T e_1(s)\dd s
 =\frac{2\sqrt{2T}}{\pi}.
\end{equation}
In particular,
\begin{equation}\label{eq:first-mode-data}
 l_1=\frac{\sqrt{2T}}{\pi}(x+y),\qquad
 A_y=\kappa_\infty\sigma l_1
     =\frac{\sqrt2\pi}{\sigma T^{3/2}}(x+y),\qquad
 B_y=\frac{A_y}{\sigma}.
\end{equation}
The critical cancellation is
\begin{equation}\label{eq:critical-cancellation}
 -\frac{z^2}{2\nu_1}
 +\frac{\kappa_\infty}{2}(l_1+\sigma z)^2
 =\frac{\kappa_\infty l_1^2}{2}+A_yz.
\end{equation}
For every measurable real-valued functional \(V\), with both sides
understood in \([0,\infty]\),
\begin{align}
\notag
 & \quad\,\, \mathbb E\Big[
  e^{\frac{\kappa_\infty}{2}\|\ell_{x,y}+\sigma\beta\|_2^2
     -V(\ell_{x,y}+\sigma\beta)}\Big]\\ \label{eq:first-mode-representation}
 &=C_0e^{-\kappa_\infty l_1^2/2}
  \mathbb E\bigg[
   e^{\frac{\kappa_\infty}{2}\|g\|_2^2}
   \int_{\mathbb R}e^{B_yu-V(ue+g)}\dd u\bigg]\\ \notag
 &=C_0e^{-\kappa_\infty l_1^2/2}\widehat H_y
  \widehat{\mathbb E}_y\bigg[
   \int_{\mathbb R}e^{B_yu-V(ue+g)}\dd u\bigg],
 \qquad C_0=(\sigma\sqrt{2\pi\nu_1})^{-1}, 
\end{align}
where
\[
 \widehat H_y
 :=\mathbb E\bigg[e^{\frac{\kappa_\infty}{2}\|g\|_2^2}\bigg]
 \le Ce^{C(1+y^2)},
 \qquad
 \frac{\dd\widehat{\mathbb P}_y}{\dd\mathbb P}
 :=\widehat H_y^{-1}e^{\frac{\kappa_\infty}{2}\|g\|_2^2}.
\]
Under \(\widehat{\mathbb P}_y\),
\(g=\widehat m_y+\widehat\beta\), where, for \(n\ge2\),
\begin{equation}\label{eq:tilted-mode-law}
 \langle\widehat m_y,e_n\rangle
 =\frac{\langle\ell_{x,y},e_n\rangle}{1-n^{-2}},
 \quad
 \operatorname{Var}\langle\widehat\beta,e_n\rangle
 =\frac{\sigma^2\nu_n}{1-n^{-2}}.
\end{equation}
Moreover, for every \(r\ge0\),
\begin{equation}\label{eq:tilted-fernique}
 \|\widehat m_y\|_\infty\le C(1+|y|),\quad
 \widehat{\mathbb P}_y(\|\widehat\beta\|_\infty>r)
 \le C e^{-c_F r^2},
\end{equation}
with constants independent of \(y\).
\end{lemma}

\begin{proof}
The Karhunen--Lo\`eve decomposition gives independence.  Since
\(\ell_{x,y}(s)+\ell_{x,y}(T-s)=x+y\) and
\(e_1(T-s)=e_1(s)\), a change of variables gives
\[
 2\langle\ell_{x,y},e_1\rangle
 =\int_0^T\bigl(\ell_{x,y}(s)+\ell_{x,y}(T-s)\bigr)e_1(s)\dd s
 =(x+y)I_1,
\]
which proves \cref{eq:first-mode-midpoint}.  Direct integration then
gives \cref{eq:first-mode-data}, and
\(\kappa_\infty\sigma^2\nu_1=1\) gives
\cref{eq:critical-cancellation}.  Substituting \(u=l_1+\sigma z\), using
orthogonality, and applying Tonelli proves
\cref{eq:first-mode-representation}.  Completing the square mode by mode,
with \(\kappa_\infty\sigma^2\nu_n=n^{-2}\), gives
\cref{eq:tilted-mode-law} and
\(\widehat H_y\le Ce^{C(1+y^2)}\).

The affine bridge has sine coefficients
\(O((1+|y|)/n)\).  Since
\((1-n^{-2})^{-1}-1=O(n^{-2})\), the correction to its transverse mean has
an absolutely and uniformly convergent sine series of size \(O(1+|y|)\),
which proves the supremum bound in \cref{eq:tilted-fernique}.  The centred
covariance in \cref{eq:tilted-mode-law} is independent of \(y\), defines a
continuous Gaussian bridge, and has finite trace; Fernique's theorem gives
the stated uniform tail bound.
\end{proof}

\begin{lemma}[Critical bridge bounds]\label{lem:critical-bridge}
Fix \(x\in\mathbb R\), \(T,\sigma,\kappa_\infty,d>0\) with
\(\kappa_\infty\sigma^2T^2=\pi^2\).  Let
\(\ell_{x,y}(s)=x+(y-x)s/T\), let \(\beta\) be a standard Brownian bridge on
\([0,T]\), and set
\begin{equation}\label{eq:critical-bridge-functional}
 G_{\alpha,d}(y):=
 \mathbb E\bigg[
  \exp\bigg\{
    \frac{\kappa_\infty}{2}\int_0^T(\ell_{x,y}+\sigma\beta)^2\dd s
    -d\int_0^T|\ell_{x,y}+\sigma\beta|^\alpha\dd s
  \bigg\}\bigg].
\end{equation}
If \(\alpha>1\) and \(p=\alpha/(\alpha-1)\), then constants
\(c,C\in(0,\infty)\), depending on the displayed parameters but not on \(y\),
satisfy
\begin{align}
 G_{\alpha,d}(y)
 &\le C\exp\{C(1+y^2+|y|^p)\},
 \label{eq:critical-bridge-upper}\\
 G_{\alpha,d}(y)
 &\ge c\exp\{c|y|^p-C(1+y^2+|y|^\alpha)\}
 \qquad\text{for }|y|\text{ sufficiently large}.
 \label{eq:critical-bridge-lower}
\end{align}
If \(0<\alpha<1\), then \(G_{\alpha,d}(y)=\infty\) whenever \(x+y\ne0\).
If \(\alpha=1\), it is infinite for all sufficiently large \(|y|\).
\end{lemma}

\begin{proof}
Use the notation and transverse estimate of
\cref{lem:first-mode-apparatus}.  Suppose first that \(\alpha>1\), and recall
\(p=\alpha/(\alpha-1)\).  H\"older's inequality gives
\begin{equation}\label{eq:projection-damping}
 \int_0^T|f(s)|^\alpha\dd s
 \ge\frac{|\langle f,e_1\rangle|^\alpha}
          {\|e_1\|_{L^p}^\alpha}.
\end{equation}
Put \(K_{\alpha,d}:=d/\|e_1\|_{L^p}^\alpha\).  After the substitution
\(u=l_1+\sigma z\), the remaining first-mode factor is, up to a fixed
constant,
\begin{equation}\label{eq:first-mode-integral}
 e^{-\kappa_\infty l_1^2/2}
 \int_{\mathbb R}
   \exp\bigg\{\frac{A_y}{\sigma}u-K_{\alpha,d}|u|^\alpha\bigg\}\dd u.
\end{equation}
Young's inequality yields
\[
 \int_{\mathbb R}e^{Bu-K_{\alpha,d}|u|^\alpha}\dd u
 \le C\exp\{C|B|^p\}.
\]
Together with
\cref{eq:first-mode-data,eq:first-mode-representation,eq:tilted-fernique},
this
proves \cref{eq:critical-bridge-upper}.

For the lower bound, choose \(M\) so that
\[
 p_M:=\mathbb P(\|\beta^\perp\|_\infty\le M)>0.
\]
On this event,
\[
 \int_0^T|\ell_{x,y}+\sigma ze_1+\sigma\beta^\perp|^\alpha\dd s
 \le C(1+|y|^\alpha+|z|^\alpha).
\]
The transverse part of the quadratic norm is nonnegative, so it may be
discarded in a lower bound.  By \cref{eq:critical-cancellation},
\[
 G_{\alpha,d}(y)
 \ge c e^{-C(1+|y|^\alpha)}
       \int_{\mathbb R}e^{A_yz-C|z|^\alpha}\dd z.
\]
For \(\alpha>1\), the maximiser has the sign of \(A\) and modulus
\((|A|/(C\alpha))^{1/(\alpha-1)}\).  Scaling about that maximiser gives
\[
 \int_{\mathbb R}e^{Az-C|z|^\alpha}\dd z
 \ge c\exp\{c|A|^p\}
\]
for large \(|A|\), which proves
\cref{eq:critical-bridge-lower}.  If \(0<\alpha<1\), the same integral is
infinite for every \(A\ne0\), since its linear term dominates in one tail.
For \(\alpha=1\), it is infinite once \(|A|>C\).
\Cref{eq:first-mode-data} gives the final assertions.
\end{proof}

\begin{lemma}[Perturbed Laplace principle]
\label{lem:perturbed-laplace}
Let \(n_y>0\), with \(n_y\to\infty\) as \(|y|\to\infty\), and let
\(F:\mathbb R\to\mathbb R\) be continuous with a unique maximiser \(v_*\).
Assume that \(F\) is twice differentiable near \(v_*\) and
\(F''(v_*)<0\).  Let \(R_y:\mathbb R\to\mathbb R\) be measurable and
suppose the following two quantified conditions hold:
\begin{enumerate}[label=\textup{(L\arabic*)}]
\item for every compact \(K\subset\mathbb R\),
\[
 \sup_{v\in K}\frac{|R_y(v)|}{n_y}\longrightarrow0;
\]
\item for every \(M>0\) there are a compact set \(K_M\) and
\(Y_M<\infty\) such that, for \(|y|\ge Y_M\),
\begin{equation}\label{eq:laplace-tightness-basic}
 \int_{K_M^{\mathsf c}}e^{n_yF(v)+R_y(v)}\dd v
 \le e^{n_y\{F(v_*)-M\}}.
\end{equation}
\end{enumerate}
Then
\[
 \log\int_{\mathbb R}e^{n_yF(v)+R_y(v)}\dd v
 =n_yF(v_*)+o(n_y).
\]

For the refined assertion, let \(P_y:\mathbb R\to\mathbb R\) be measurable,
let \(P\) be continuous, and suppose \(P_y\to P\) locally uniformly.
Assume \(\varepsilon_y\to0\),
\begin{equation}\label{eq:laplace-log-scale}
 \frac{n_y|\varepsilon_y|}{\log n_y}\longrightarrow\infty,
\end{equation}
and replace \textup{(L1)--(L2)} by the stronger conditions
\begin{enumerate}[label=\textup{(R\arabic*)}]
\item for every compact \(K\subset\mathbb R\),
\[
 \sup_{v\in K}\frac{|R_y(v)|}{n_y|\varepsilon_y|}
 \longrightarrow0;
\]
\item for every \(M>0\) there are a compact set \(K_M\) and
\(Y_M<\infty\) such that, for \(|y|\ge Y_M\),
\begin{equation}\label{eq:laplace-tightness-refined}
 \int_{K_M^{\mathsf c}}
 e^{n_y\{F(v)+\varepsilon_yP_y(v)\}+R_y(v)}\dd v
 \le e^{n_y\{F(v_*)-M\}}.
\end{equation}
\end{enumerate}
Then
\begin{equation}\label{eq:perturbed-laplace}
 \log\int_{\mathbb R}
 e^{n_y\{F(v)+\varepsilon_yP_y(v)\}+R_y(v)}\dd v
 =n_yF(v_*)+n_y\varepsilon_yP(v_*)
  +o(n_y|\varepsilon_y|).
\end{equation}
\end{lemma}

\begin{proof}
Condition \cref{eq:laplace-tightness-basic} restricts the first integral to
a compact set with an arbitrarily prescribed loss on the \(n_y\)-scale.
On that compact set, uniqueness gives a strict gap outside every
neighbourhood of \(v_*\), while \textup{(L1)} controls \(R_y\).  A fixed
small neighbourhood gives the matching lower bound by continuity of \(F\).
Letting the neighbourhood shrink proves the first assertion.

For the refined assertion, put \(a_y=|\varepsilon_y|\) and integrate the
lower bound over \([v_*-a_y,v_*+a_y]\).  Taylor's theorem gives a loss
\(O(n_ya_y^2)=o(n_ya_y)\); continuity of \(P\), local uniform convergence
of \(P_y\), and the hypothesis on \(R_y\) contribute only
\(o(n_ya_y)\).  The interval length contributes \(\log(2a_y)\).  Since
\cref{eq:laplace-log-scale} implies \(-\log a_y=O(\log n_y)\), this term is
also \(o(n_ya_y)\).  For the upper bound, a fixed neighbourhood of \(v_*\)
on which \(F(v)\le F(v_*)-c|v-v_*|^2\) gives a Gaussian integral and hence
only an \(O(\log n_y)\) normalisation.  Given \(\eta>0\), shrink this
neighbourhood so that \(|P(v)-P(v_*)|\le\eta\), and then use local uniform
convergence of \(P_y\) and \textup{(R1)}.  Its compact complement loses a
fixed multiple of \(n_y\), and \cref{eq:laplace-tightness-refined} handles
the tails.  The upper error is therefore at most
\(\eta n_ya_y+o(n_ya_y)\).  Letting \(\eta\downarrow0\) proves
\cref{eq:perturbed-laplace}.  Thus all uniformity is over the explicitly
displayed family \(|y|\ge Y_M\); no unquantified tail condition is used.
\end{proof}

At the power \(4/3\), the H\"older projection
\cref{eq:projection-damping} has the correct order but loses the leading
coefficient.  The next lemma replaces it by an expansion around the first
bridge mode.

\begin{lemma}[Sharp bridge asymptotic at power \(4/3\)]
\label{lem:sharp-four-thirds}
Under the hypotheses and notation of \cref{lem:critical-bridge}, put
\[
 J_{4/3}:=\int_0^\pi\sin^{4/3}v\dd v
 =\frac{\sqrt\pi\,\Gamma(7/6)}{\Gamma(5/3)}.
\]
Then, for every \(d>0\),
\begin{equation}\label{eq:sharp-bridge-limit}
 \log G_{4/3,d}(y)
 =\mathfrak L(d)|x+y|^4+O\big(1+y^2+\log(2+|y|)\big),
 \qquad
 \mathfrak L(d)
 :=\frac{27\kappa_\infty^{7/2}}{256d^3\sigma J_{4/3}^3}.
\end{equation}
\end{lemma}

\begin{proof}
Use the notation and representation of \cref{lem:first-mode-apparatus} with
\(V(f)=d\|f\|_{4/3}^{4/3}\), and set
\[
 I_{4/3}:=\int_0^T e(s)^{4/3}\dd s
 =\frac{2^{2/3}T^{1/3}}{\pi}J_{4/3},
 \qquad \mathsf D:=dI_{4/3},
 \qquad B_0:=\frac{\sqrt2\pi}{\sigma^2T^{3/2}}.
\]

Because \(e=e_1\) is nonnegative on \([0,T]\),
\(\operatorname{sgn}(ue)=\operatorname{sgn}(u)\) and
\(|ue|^{1/3}=|u|^{1/3}e^{1/3}\).  Convexity of
\(v\mapsto|v|^{4/3}\) therefore yields
\begin{equation}\label{eq:four-thirds-convexity}
 \|ue+g\|_{4/3}^{4/3}
 \ge I_{4/3}|u|^{4/3}
 +\frac43\operatorname{sgn}(u)|u|^{1/3}
   \langle e^{1/3},g\rangle.
\end{equation}
The linear functional \(\langle e^{1/3},g\rangle\) is therefore Gaussian
under \(\widehat{\mathbb P}_y\), with mean \(O(1+|y|)\) and variance bounded
uniformly in \(y\).  The Gaussian exponential bound
\(\mathbb E[e^{t|Z|}]\le2e^{t|\mathbb EZ|+t^2\operatorname{Var}(Z)/2}\)
then implies, from \cref{eq:first-mode-representation},
\begin{equation}
\label{eq:sharp-upper-integral}
 G_{4/3,d}(y)
 \le e^{C(1+y^2)}
 \int_{\mathbb R}\exp\big\{
   B_yu-\mathsf D|u|^{4/3}
   +C(1+|y|)|u|^{1/3}+C|u|^{2/3}
 \big\}\dd u .
\end{equation}

For the lower bound, use the elementary remainder estimate
\begin{equation}\label{eq:four-thirds-remainder}
 |a+b|^{4/3}
 \le |a|^{4/3}
 +\frac43\operatorname{sgn}(a)|a|^{1/3}b
 +C|b|^{4/3}.
\end{equation}
Choose \(M\) such that
\(E_M=\{\|\beta^\perp\|_{4/3}\le M\}\) has positive probability.  On
\(E_M\),
\[
 |\langle e^{1/3},g\rangle|\le C(1+|y|),
 \qquad
 \|g\|_{4/3}^{4/3}\le C(1+|y|^{4/3}).
\]
Restricting the first expectation on the right of
\cref{eq:first-mode-representation} to \(E_M\), and discarding the
nonnegative transverse quadratic exponent, gives
\begin{equation}\label{eq:sharp-lower-integral}
 G_{4/3,d}(y)
 \ge e^{-C(1+y^2)}
 \int_{\mathbb R}\exp\big\{
   B_yu-\mathsf D|u|^{4/3}
   -C(1+|y|)|u|^{1/3}-C(1+|y|^{4/3})
 \big\}\dd u .
\end{equation}

Put \(r=|x+y|\) and substitute \(u=r^3v\).  The leading exponent in both
\cref{eq:sharp-upper-integral,eq:sharp-lower-integral}, divided by
\(r^4\), is exactly
\[
 \operatorname{sgn}(x+y)B_0v-\mathsf D|v|^{4/3},
\]
while, after also absorbing the prefactors outside the integrals, the
displayed transverse errors are bounded above and below by
\(\pm Cr^2(1+|v|^{2/3})\).  To justify the required error scale explicitly,
apply the refined part of \cref{lem:perturbed-laplace} to the two comparison
integrals with
\[
 n_y=r^4,\qquad \varepsilon_y=r^{-2},\qquad
 P_\pm(v)=\pm C(1+|v|^{2/3}),\qquad R_y=0.
\]
Condition \textup{(R1)} is immediate, and
\(n_y\varepsilon_y/\log n_y=r^2/(4\log r)\to\infty\).  For
\textup{(R2)}, the negative \(-\mathsf D|v|^{4/3}\) term dominates
\(r^{-2}P_\pm(v)\) outside a fixed compact set, uniformly for large \(r\).
Thus each comparison integral equals its leading Laplace value up to
\(O(r^2)\); the change of variables contributes only \(3\log r\).  Hence
\begin{align*}
 \log G_{4/3,d}(y)
 &=r^4\sup_{v\in\mathbb R}\{B_0v-\mathsf D|v|^{4/3}\}
  +O\big(1+y^2+\log(2+r)\big)\\
 &=\frac{27B_0^4}{256\mathsf D^3}r^4
  +O\big(1+y^2+\log(2+r)\big).
\end{align*}
Insert the displayed values of \(B_0\) and \(I_{4/3}\), and use
\(\kappa_\infty\sigma^2T^2=\pi^2\), to obtain
\cref{eq:sharp-bridge-limit}.
\end{proof}

We call a measurable function \(h\) \emph{signed regularly varying of
index \(\rho\)} if it is eventually nonzero and of one sign and if
\(L:=|h|\) is regularly varying of index \(\rho\).  Thus
\(L(tr)/L(r)\to t^\rho\) for every \(t>0\).  The case \(\rho=0\) is signed
slow variation.  We always take a locally bounded Borel extension to
\([0,\infty)\).  We use the uniform convergence theorem and Potter bounds
in their standard form; see
\citet[Theorems~1.2.1 and~1.5.6]{binghamgoldieteugels1987}.

\begin{lemma}[Stability under transverse paths]
\label{lem:slow-pathwise-stability}
Let \(e=e_1\), let \(h\) be signed regularly varying of index
\(\rho\in(-2/3,0]\), with \(h(r)\to0\), and put
\[
 \Phi_h(f):=\int_0^T|f(s)|^{4/3}h(|f(s)|)\dd s.
\]
Then, uniformly over continuous functions \(g=g_u\) satisfying
\(\|g\|_\infty\le |u|^{2/3}\),
\begin{equation}\label{eq:pathwise-slow-stability}
 \Phi_h(ue+g)
 =I_{4/3+\rho}|u|^{4/3}h(|u|)\big(1+o(1)\big),
 \qquad |u|\to\infty,
\end{equation}
where
\[
 I_{4/3+\rho}:=\int_0^Te(s)^{4/3+\rho}\dd s.
\]
\end{lemma}

\begin{proof}
It is enough to work with \(L=|h|\), since the sign is eventually
constant.  Put \(\delta_u=|u|^{-1/6}\) and split the interval according to
\(e\ge\delta_u\) and \(e<\delta_u\).  On the first set,
\[
 \sup_{e\ge\delta_u}\frac{|g|}{|u|e}
 \le |u|^{-1/6}\longrightarrow0.
\]
The uniform convergence theorem for the regularly varying function
\(r^{4/3}L(r)\), whose index is \(4/3+\rho\), therefore gives, uniformly
in the stated class of \(g\),
\[
 |ue+g|^{4/3}L(|ue+g|)
 =|u|^{4/3}e^{4/3}L(|u|e)\big(1+o(1)\big).
\]
The uniform convergence theorem on fixed compact subintervals of
\((0,\infty)\), followed by monotone enlargement of those intervals, shows
that the integral of the right-hand side is
\(I_{4/3+\rho}|u|^{4/3}L(|u|)(1+o(1))\).

For completeness, the shrinking endpoint pieces are controlled directly.
For any fixed \(0<\varepsilon<4/3+\rho\), Potter's bound gives
\[
 \frac{L(r t)}{L(r)}\le C_\varepsilon
 \max\{t^{\rho+\varepsilon},t^{\rho-\varepsilon}\},
 \qquad r,rt\ \text{large}.
\]
The first eigenfunction vanishes linearly at both endpoints.  Since
\(|ue+g|/|u|\le e+|u|^{-1/3}\), the contribution of \(e<\delta_u\), apart
from a set on which the argument remains in a fixed compact interval, is at
most
\[
 C|u|^{4/3}L(|u|)
 \int_{\{e<\delta_u\}}
   (e+|u|^{-1/3})^{4/3+\rho-\varepsilon}\dd s
 =o\big(|u|^{4/3}L(|u|)\big).
\]
The compact-argument part is \(O(\delta_u)\) by local boundedness, which is
again negligible because \(|u|^{4/3}L(|u|)\to\infty\).  This proves
\cref{eq:pathwise-slow-stability}.
\end{proof}

\begin{lemma}[Suppression of the transverse complement]
\label{lem:transverse-complement}
Under \(\widehat{\mathbb P}_y\), write
\(g=\widehat m_y+\widehat\beta\) as in
\cref{lem:first-mode-apparatus}.  For \(u=|y|^3v\), with \(v\) in a fixed
compact subset of \(\mathbb R\setminus\{0\}\), put
\begin{equation}\label{eq:main-transverse-event}
 E_{y,u}:=\{\|\widehat\beta\|_\infty\le |u|^{2/3}/2\}.
\end{equation}
If \(h\) is signed regularly varying of index \(\rho\in(-2/3,0]\) and
\(h(r)\to0\), then constants
\(c,C>0\), independent of \(y,u\), satisfy
\begin{equation}\label{eq:complement-suppression}
 \widehat{\mathbb E}_y\big[
   e^{|\Phi_h(ue+g)|}\mathbf1_{E_{y,u}^{\mathsf c}}\big]
 \le C e^{-c|u|^{4/3}}
\end{equation}
for all sufficiently large \(|y|\).
\end{lemma}

\begin{proof}
Choose \(r_0\) and set
\(\varepsilon_0:=\sup_{r\ge r_0}|h(r)|\).  Local boundedness and
\(|a+b|^{4/3}\le C(|a|^{4/3}+|b|^{4/3})\) give
\begin{equation}\label{eq:Phi-coarse-bound}
 |\Phi_h(ue+g)|
 \le C\varepsilon_0\big(|u|^{4/3}+\|g\|_{4/3}^{4/3}\big)+C_{r_0}.
\end{equation}
This is also the Potter estimate
\(r^{4/3}|h(r)|\le C(1+r^2)\), using any exponent below \(2/3\), but
\cref{eq:Phi-coarse-bound} retains the small tail coefficient needed here.
By Cauchy--Schwarz and \cref{eq:tilted-fernique},
\begin{align}
 &\qquad\,\widehat{\mathbb E}_y\big[
   e^{|\Phi_h(ue+g)|}\mathbf1_{E_{y,u}^{\mathsf c}}\big]
 \notag\\
 &\quad\le
 \big(\widehat{\mathbb E}_y e^{2|\Phi_h(ue+g)|}\big)^{1/2}
 \widehat{\mathbb P}_y(E_{y,u}^{\mathsf c})^{1/2}
 \notag\\
 &\quad\le C\exp\bigg\{
   \bigg(C\varepsilon_0-\frac{c_F}{8}\bigg)|u|^{4/3}
   +C\varepsilon_0|u|^{4/9}\bigg\}.
 \label{eq:complement-CS}
\end{align}
Here \(\|\widehat m_y\|_{4/3}^{4/3}=O(|y|^{4/3})=O(|u|^{4/9})\), and
Fernique is used once for the probability and once for the exponential
moment; the exponent \(4/3<2\) is essential.  Since
\(\varepsilon_0\downarrow0\) as \(r_0\uparrow\infty\), choose \(r_0\) so that
\(C\varepsilon_0<c_F/16\).  This is non-circular:
\(\varepsilon_0\) depends only on \(h\), whereas \(c_F\) depends only on the
centred bridge covariance.  The \(|u|^{4/9}\) term is lower order, proving
\cref{eq:complement-suppression}.  Thus the complement loses a full
\(e^{-c|y|^4}\), while the signal at \(b=b_*\) has size
\(|y|^4|h(|y|^3v_*)|\); the separation is by a genuine power gap.
\end{proof}

\begin{lemma}[Regularly varying refinement of the critical bridge]
\label{lem:slowly-varying-bridge}
Under the hypotheses and notation of \cref{lem:sharp-four-thirds}, let
\(h\) be signed regularly varying of index \(\rho\in(-2/3,0]\), with
\(h(r)\to0\), and, for \(d>0\), set
\[
 G^h_d(y):=
 \mathbb E\bigg[
  \exp\bigg\{
   \frac{\kappa_\infty}{2}\int_0^T(\ell_{x,y}+\sigma\beta)^2\dd s
   -d\int_0^T|\ell_{x,y}+\sigma\beta|^{4/3}\dd s
   +\Phi_h(\ell_{x,y}+\sigma\beta)
  \bigg\}\bigg].
\]
Let
\[
 B_0:=\frac{\sqrt2\pi}{\sigma^2T^{3/2}},\qquad
 I_{4/3}:=\frac{2^{2/3}T^{1/3}}{\pi}J_{4/3},\qquad
 I_{4/3+\rho}:=\int_0^Te(s)^{4/3+\rho}\dd s,
 \qquad
 v_d:=\bigg(\frac{3B_0}{4dI_{4/3}}\bigg)^3.
\]
Then
\begin{equation}\label{eq:slowly-varying-bridge-limit}
 \log G^h_d(y)
 =\mathfrak L(d)|x+y|^4
  +I_{4/3+\rho} v_d^{4/3}|x+y|^4
    h\big(|x+y|^3v_d\big)\big(1+o(1)\big),
 \qquad |y|\to\infty.
\end{equation}
\end{lemma}

\begin{proof}
Apply \cref{eq:first-mode-representation} with
\(V(f)=d\|f\|_{4/3}^{4/3}-\Phi_h(f)\).  Put \(r=|x+y|\).  On the event
\cref{eq:main-transverse-event}, the mean bound in
\cref{eq:tilted-fernique} gives \(\|g\|_\infty\le|u|^{2/3}\) for
\(u=r^3v\), locally uniformly away from \(v=0\).  Hence
\cref{lem:slow-pathwise-stability} yields
\begin{equation}\label{eq:slow-main-event}
 \Phi_h(ue+g)
 =I_{4/3+\rho}|u|^{4/3}h(|u|)\big(1+o(1)\big).
\end{equation}
The complement is negligible by
\cref{lem:transverse-complement}; this is the normalisation step for which a
purely pointwise substitution of the first mode would not suffice.

The ordinary transverse errors in
\cref{eq:sharp-upper-integral,eq:sharp-lower-integral} are \(O(y^2)\).
They are negligible because regular variation with \(\rho>-2/3\) implies
\[
 y^2|h(|y|^3)|\longrightarrow\infty.
\]
Also, \(|h|\to0\) leaves a fixed fraction of the negative
\(-dI_{4/3}|u|^{4/3}\) term, so the first-mode tails remain uniformly
controlled.  After \(u=r^3v\), the leading rate
\[
 F_\pm(v)=\pm B_0v-dI_{4/3}|v|^{4/3}
\]
has the unique maximiser \(v_\pm=\pm v_d\).  The uniform convergence
theorem, with Potter's bound controlling a neighbourhood of the origin,
gives locally uniformly on \(\mathbb R\),
\[
 P_y(v):=I_{4/3+\rho}|v|^{4/3}
 \frac{h(r^3|v|)}{h(r^3v_d)}
 \longrightarrow I_{4/3+\rho}v_d^{-\rho}|v|^{4/3+\rho}.
\]
Put \(n_y=r^4\) and \(\varepsilon_y=h(r^3v_d)\).  Regular variation and
the preceding estimate give
\[
 \frac{n_y|\varepsilon_y|}{\log n_y}
 =\frac{r^4|h(r^3v_d)|}{4\log r}
 \longrightarrow\infty,
\]
and make the \(O(y^2)\) errors \(o(n_y|\varepsilon_y|)\).  The remaining
negative \(-dI_{4/3}|v|^{4/3}\) term, together with Potter bounds for the
perturbation, verifies \cref{eq:laplace-tightness-refined}: for each
\(M>0\), a sufficiently large fixed interval is a valid \(K_M\), uniformly
for all large \(|y|\).  Therefore \cref{lem:perturbed-laplace}, applied
separately to the two endpoint tails, gives
\cref{eq:slowly-varying-bridge-limit}.
\end{proof}

\begin{theorem}[One-sided classification on the equality surface]
\label{thm:alpha-transition}
Let \(\lambda\in\mathcal D\), translate the equilibrium to zero, and fix
\(\kappa_\infty,\sigma,T>0\) with
\begin{equation}\label{eq:critical-surface}
 \kappa_\infty\sigma^2T^2=\pi^2.
\end{equation}
For every starting state \(x\), the following statements hold.
\begin{enumerate}[label=\textup{(\roman*)}]
\item \emph{Divergence.}  Suppose \(q(y)=O(y^2)\), and for some
 \(C<\infty\), either
 \begin{equation}\label{eq:one-sided-divergence-zero}
 q(y)\ge\kappa_\infty y^2-C\qquad(y\in\mathbb R),
 \end{equation}
 or, for some \(b>0\) and \(0<\alpha<4/3\),
 \begin{equation}\label{eq:one-sided-divergence}
 q(y)\ge\kappa_\infty y^2-b|y|^\alpha-C
  \qquad(y\in\mathbb R).
 \end{equation}
 Then \(\mathcal N_{1/2}(T,x)=\infty\).  The first alternative includes
 eventual domination \(q(y)\ge\kappa_\infty y^2\) outside a compact set;
 no vacuous power parameter is attached to that case.
\item \emph{Finiteness.}  Suppose that for some
 \(b>0\), \(4/3<\alpha<2\), and \(R<\infty\),
 \begin{equation}\label{eq:one-sided-finiteness}
  q(y)\le\kappa_\infty y^2-b|y|^\alpha
  \qquad(|y|\ge R).
 \end{equation}
 Assume also the terminal-confinement condition
 \begin{equation}\label{eq:one-sided-Lambda-confinement}
  \limsup_{|y|\to\infty}\frac{\Lambda(y)}{y^4}<0.
 \end{equation}
 Then \(\mathcal N_{1/2}(T,x)<\infty\).
\item \emph{The power \(4/3\).}  Suppose that
 \begin{equation}\label{eq:four-thirds-assumption}
  q(y)=\kappa_\infty y^2-b|y|^{4/3}+o(|y|^{4/3}),
  \qquad b>0.
 \end{equation}
 Define
 \begin{equation}\label{eq:b-star}
  b_*:=\frac{3^{4/3}\sigma^{1/3}\kappa_\infty^{5/6}}{2J_{4/3}},
  \qquad
  J_{4/3}=\frac{\sqrt\pi\,\Gamma(7/6)}{\Gamma(5/3)}.
 \end{equation}
 Then
 \[
  b<b_*\ \Longrightarrow\ \mathcal N_{1/2}(T,x)=\infty,
  \qquad
  b>b_*\ \Longrightarrow\ \mathcal N_{1/2}(T,x)<\infty.
 \]
 Assumption \cref{eq:four-thirds-assumption} alone makes no assertion when
 \(b=b_*\); see \cref{thm:bstar-nonuniversal}.
\end{enumerate}
\end{theorem}

\begin{proof}
Normalise \(\Lambda(0)=0\).  Conditioning the Brownian representation
\cref{eq:transfer-Psi} on \(X_T=y\) gives
\begin{equation}\label{eq:equality-endpoint-integral}
 \mathcal N_{1/2}(T,x)
 =\frac{e^{-\Lambda(x)/\sigma^2}}{\sqrt{2\pi\sigma^2T}}
  \int_{\mathbb R}
   e^{-(y-x)^2/(2\sigma^2T)+\Lambda(y)/\sigma^2}
   \mathbb E\bigg[
     e^{\frac12\int_0^Tq(\ell_{x,y}+\sigma\beta)\dd s}
   \bigg]\dd y.
\end{equation}

For (i), \(q(y)=O(y^2)\) implies
\(\Lambda(y)\ge-C(1+y^4)\).  Under
\cref{eq:one-sided-divergence-zero}, the conditional lower bound contains
the critical first bridge factor from
\cref{eq:parseval-bridge,eq:bridge-product}; it is infinite by
\cref{eq:gaussian-factor}.  Under the second alternative, by
\cref{eq:one-sided-divergence}, the conditional expectation is bounded below,
up to a fixed positive factor, by \(G_{\alpha,b/2}(y)\).  When
\(0<\alpha\le1\), \cref{lem:critical-bridge} makes this conditional
expectation infinite for almost every sufficiently large endpoint.  When
\(1<\alpha<4/3\), put \(p=\alpha/(\alpha-1)>4\).  Then
\cref{eq:critical-bridge-lower} shows that the logarithm of the endpoint
integrand in \cref{eq:equality-endpoint-integral} is bounded below by
\[
 c|y|^p-C(1+|y|^4)
\]
for large \(|y|\).  Hence the endpoint integral diverges.

For (ii), the compact part of \cref{eq:one-sided-finiteness} can be absorbed
into constants:
\[
 q(y)\le\kappa_\infty y^2-b|y|^\alpha+C,
\]
Condition \cref{eq:one-sided-Lambda-confinement} gives constants
\(c_0,C>0\) such that, after enlarging \(C\),
\[
 \Lambda(y)\le-\frac{c_0}{12}y^4+C(1+y^2).
\]
The first inequality and \cref{eq:critical-bridge-upper} bound the
conditional expectation in \cref{eq:equality-endpoint-integral} by
\[
 C\exp\{C(1+y^2+|y|^p)\},
 \qquad p=\frac{\alpha}{\alpha-1}<4.
\]
Since \(p<4\), all positive terms in the endpoint exponent have order
strictly below four, while the terminal contribution is negative quartic.

For (iii), \cref{eq:four-thirds-assumption} implies, for every
\(\varepsilon\in(0,b)\),
\[
 \kappa_\infty y^2-(b+\varepsilon)|y|^{4/3}-C_\varepsilon
 \le q(y)
 \le\kappa_\infty y^2-(b-\varepsilon)|y|^{4/3}+C_\varepsilon.
\]
It also gives, by two integrations,
\[
 \frac{\Lambda(y)}{y^4}\longrightarrow-\frac{\kappa_\infty}{12}.
\]
Apply \cref{lem:sharp-four-thirds} to the two conditional comparisons and
let \(\mathfrak h_x(y)\) denote the integrand in
\cref{eq:equality-endpoint-integral}, including the endpoint Gaussian.  For
every \(\varepsilon\in(0,b)\),
\[
 \liminf_{|y|\to\infty}\frac{\log \mathfrak h_x(y)}{|y|^4}
 \ge-\frac{\kappa_\infty}{12\sigma^2}
      +\mathfrak L\bigg(\frac{b+\varepsilon}{2}\bigg),
\]
whereas
\[
 \limsup_{|y|\to\infty}\frac{\log \mathfrak h_x(y)}{|y|^4}
 \le-\frac{\kappa_\infty}{12\sigma^2}
      +\mathfrak L\bigg(\frac{b-\varepsilon}{2}\bigg).
\]
Introduce the continuous rate
\begin{equation}\label{eq:four-thirds-endpoint-rate}
 \mathcal R(b):=-\frac{\kappa_\infty}{12\sigma^2}
  +\mathfrak L\bigg(\frac b2\bigg)
 =-\frac{\kappa_\infty}{12\sigma^2}
  +\frac{27\kappa_\infty^{7/2}}{32b^3\sigma J_{4/3}^3}.
\end{equation}
It is positive precisely when \(b<b_*\) and negative precisely when
\(b>b_*\).  In either case choose \(\varepsilon\) small enough that the
relevant one-sided bound has the same strict sign as \(\mathcal R(b)\).
This proves divergence or finiteness without interchanging
\(\varepsilon\downarrow0\) with the endpoint limit.  Solving
\(\mathcal R(b)=0\) gives \cref{eq:b-star}.
\end{proof}

\begin{corollary}[Asymmetric corrections at power \(4/3\)]
\label{cor:asymmetric-four-thirds}
Let \(\lambda\in\mathcal D\), translate the equilibrium to zero, and assume
the critical relation \cref{eq:critical-surface}.  Suppose that for some
\(b_+,b_->0\),
\[
 q(y)=
 \begin{cases}
  \kappa_\infty y^2-b_+y^{4/3}+o(y^{4/3}),&y\to+\infty,\\
  \kappa_\infty y^2-b_-|y|^{4/3}+o(|y|^{4/3}),&y\to-\infty.
 \end{cases}
\]
Then
\[
 b_+>b_*\ \text{and}\ b_->b_*
 \quad\Longrightarrow\quad
 \mathcal N_{1/2}(T,x)<\infty,
\]
whereas
\[
 \min\{b_+,b_-\}<b_*
 \quad\Longrightarrow\quad
 \mathcal N_{1/2}(T,x)=\infty.
\]
No assertion is made if one coefficient equals \(b_*\) and neither is
smaller.
\end{corollary}

\begin{proof}
Split the endpoint integral \cref{eq:equality-endpoint-integral} at zero and
repeat the proof of \cref{thm:alpha-transition}(iii) separately on the two
tails.  In the first-mode representation, the near-maximising coordinate has
\(\operatorname{sgn}(u)=\operatorname{sgn}(y)\) and \(|u|\asymp|y|^3\).
Use the event \(E_{y,u}\) from \cref{eq:main-transverse-event}.  The mean
bound in \cref{eq:tilted-fernique} gives
\(\|g\|_\infty\le|u|^{2/3}\) there.  Let \(S_u\) be the set on which
\(ue_1+g\) has sign opposite to \(u\).  Linear vanishing of \(e_1\) gives
\(|S_u|=O(\|g\|_\infty/|u|)\), and hence
\[
 \int_{S_u}|ue_1+g|^{4/3}\dd s
 \le C\frac{\|g\|_\infty^{7/3}}{|u|}
 =O(|u|^{5/9})=O(|y|^{5/3})=o(|y|^4).
\]
Off \(E_{y,u}\), the Fernique tail in
\cref{eq:tilted-fernique} gives the same full
\(e^{-c|u|^{4/3}}\) suppression as in
\cref{lem:transverse-complement}, now without any additional regularly
varying perturbation.  Hence on
\(y\to+\infty\) the rate
\cref{eq:four-thirds-endpoint-rate} contains \(b_+\), and on
\(y\to-\infty\) it contains \(b_-\).  The starting point does not alter
either threshold because
\[
 \frac{|x+y|^4}{|y|^4}\longrightarrow1
 \qquad (y\to\pm\infty).
\]
Thus both tail rates must be negative for finiteness, while a positive rate
in either tail forces divergence.
\end{proof}

\begin{theorem}[Signed regularly varying corrections above power \(10/9\)]
\label{thm:bstar-nonuniversal}
Let \(\lambda\in\mathcal D\), translate the equilibrium to zero, assume
\cref{eq:critical-surface}, and let \(h\) be signed regularly varying of
index \(\rho\in(-2/9,0]\), with \(h(r)\to0\).  Suppose
\begin{equation}\label{eq:bstar-slow-assumption}
 q(y)=\kappa_\infty y^2-b_*|y|^{4/3}
      +|y|^{4/3}h(|y|)
      +o\big(|y|^{4/3}|h(|y|)|\big),
 \qquad |y|\to\infty.
\end{equation}
Then, for every \(x\in\mathbb R\),
\[
 h>0\ \text{eventually}
 \quad\Longrightarrow\quad
 \mathcal N_{1/2}(T,x)=\infty,
 \qquad
 h<0\ \text{eventually}
 \quad\Longrightarrow\quad
 \mathcal N_{1/2}(T,x)<\infty.
\]
Consequently, the two-term assumption
\cref{eq:four-thirds-assumption} with \(b=b_*\) does not determine the
critical moment.  Within the class of eventually signed slowly varying
third-order corrections, and more generally for every
\(\rho\in(-2/9,0]\), the sign gives a complete classification.  This index
range corresponds to third-order powers strictly above \(10/9\), together
with slowly varying perturbations of the \(4/3\) term.
\end{theorem}

\begin{proof}
Let \(d_*=b_*/2\) and let \(v_*=v_{d_*}\) be as in
\cref{lem:slowly-varying-bridge}; also put
\(I_{4/3+\rho}=\int_0^Te_1(s)^{4/3+\rho}\dd s\).  For every fixed
\(\varepsilon\in(0,1)\), the remainder in
\cref{eq:bstar-slow-assumption} sandwiches the conditional bridge
expectation in \cref{eq:equality-endpoint-integral}, up to finite positive
factors, between the functionals in
\cref{lem:slowly-varying-bridge} with perturbations
\((1-\varepsilon)h/2\) and \((1+\varepsilon)h/2\), with the order
reversed when \(h<0\).  Moreover,
\cref{eq:bstar-slow-assumption} gives
\[
 \Lambda(y)=-\frac{\kappa_\infty}{12}y^4
 +\frac{9b_*}{70}|y|^{10/3}+o(|y|^{10/3}).
\]
Regular variation with \(\rho>-2/9\) and Potter's bound imply
\begin{equation}\label{eq:slow-endpoint-negligible}
 |y|^{10/3}+|y|^3+y^2
 =o\big(|y|^4|h(|y|^3v_*)|\big).
\end{equation}
At \(b=b_*\), the quartic endpoint term cancels by the definition of
\(b_*\).  Fix \(\varepsilon=1/2\).  If \(h>0\), the lower comparison,
\cref{eq:slowly-varying-bridge-limit}, and
\cref{eq:slow-endpoint-negligible} give eventually
\[
 \log \mathfrak h_x(y)\ge
 \frac{I_{4/3+\rho}v_*^{4/3}}4|y|^4h(|y|^3v_*)>0.
\]
If \(h<0\), the upper comparison gives the reverse strict bound with the
same coefficient.  Its magnitude is regularly varying with index
\(4+3\rho>10/3\), so it dominates \(|y|^{10/3+\delta}\) for some
\(\delta>0\).  Thus the positive case makes
the endpoint integral diverge, while the negative case gives an integrable
stretched-polynomial upper bound.  No limit interchange in \(\varepsilon\)
is required.
\end{proof}

\begin{corollary}[The second critical power at \(b_*\)]
\label{cor:second-critical-power}
Let \(\lambda\in\mathcal D\), translate its equilibrium to zero, and assume
the critical relation \cref{eq:critical-surface}.
Let \(d_*=b_*/2\), let \(v_*=v_{d_*}\) be as in
\cref{lem:slowly-varying-bridge}, and put
\[
 I_{10/9}:=\int_0^Te_1(s)^{10/9}\dd s,
 \qquad
 \gamma_*:=-\frac{9b_*}
 {35\sigma^2I_{10/9}v_*^{10/9}}<0.
\]
Then the following statements hold for every starting state \(x\).
\begin{enumerate}[label=\textup{(\roman*)}]
\item If \(10/9<\beta<4/3\), \(\gamma\ne0\), and
\[
 q(y)=\kappa_\infty y^2-b_*|y|^{4/3}
      +\gamma|y|^\beta+o(|y|^\beta),
\]
then the moment diverges for \(\gamma>0\) and is finite for \(\gamma<0\).
In particular, the correction \(\gamma|y|^{6/5}\) is classified by its
sign.
\item At the boundary power,
\[
 q(y)=\kappa_\infty y^2-b_*|y|^{4/3}
      +\gamma|y|^{10/9}+o(|y|^{10/9}),
\]
the moment diverges for \(\gamma>\gamma_*\) and is finite for
\(\gamma<\gamma_*\).  No assertion is made at \(\gamma=\gamma_*\).
\item If
\begin{equation}\label{eq:below-second-critical-power}
 q(y)=\kappa_\infty y^2-b_*|y|^{4/3}+o(|y|^{10/9}),
\end{equation}
then \(\mathcal N_{1/2}(T,x)=\infty\).  Hence the exact two-term tail, with
no further correction, lies on the divergent side.
\end{enumerate}
\end{corollary}

\begin{proof}
Part (i) is \cref{thm:bstar-nonuniversal} with
\(h(r)=\gamma r^{\beta-4/3}\), whose regular-variation index lies in
\((-2/9,0)\).

For (ii)--(iii), two integrations of the common first two terms give
\begin{equation}\label{eq:bstar-Lambda-third-scale}
 \Lambda(y)=-\frac{\kappa_\infty}{12}y^4
 +\frac{9b_*}{70}|y|^{10/3}+o(|y|^{10/3});
\end{equation}
the \(|y|^{10/9}\) correction itself contributes only order
\(|y|^{28/9}\) to \(\Lambda\).  Apply
\cref{lem:slowly-varying-bridge} at \(\rho=-2/9\) to the conditional
potential, which contains one half of the correction in \(q\).  Together
with \cref{eq:sharp-bridge-limit}, this gives, in case (ii),
\begin{equation}\label{eq:ten-ninth-endpoint-rate}
 \log\mathfrak h_x(y)
 =\bigg\{\frac{9b_*}{70\sigma^2}
   +\frac{\gamma}{2}I_{10/9}v_*^{10/9}\bigg\}|y|^{10/3}
  +o(|y|^{10/3}).
\end{equation}
Indeed, the cancelled quartic bridge term is naturally expressed through
\(|x+y|^4\); replacing it by \(|y|^4\) costs \(O(|y|^3)\), the transverse
error in \cref{eq:sharp-bridge-limit} is \(O(y^2)\), and the endpoint
Gaussian is also lower order.  The coefficient in braces vanishes exactly
at \(\gamma_*\), proving (ii).

Under \cref{eq:below-second-critical-power}, comparison with
\(\pm\varepsilon|y|^{10/9}\), followed by \(\varepsilon\downarrow0\),
turns \cref{eq:ten-ninth-endpoint-rate} into
\[
 \log\mathfrak h_x(y)
 =\frac{9b_*}{70\sigma^2}|y|^{10/3}
  +o(|y|^{10/3}).
\]
The coefficient is strictly positive, so the endpoint integral diverges.
\end{proof}

\begin{remark}[Exact endpoint refinement]
\label{rem:exact-slow-endpoint-rate}
The two one-sided comparisons also squeeze the endpoint rate to
\begin{equation}\label{eq:bstar-slow-endpoint-rate}
 \log \mathfrak h_x(y)
 =\frac{I_{4/3+\rho}v_*^{4/3}}2|y|^4
   h(|y|^3v_*)\big(1+o(1)\big).
\end{equation}
This coefficient is not needed for the sign classification, but it recovers
the explicit logarithmic scale in \cref{cor:bstar-logarithmic}.
\end{remark}

If the normalised correction in \cref{eq:bstar-slow-assumption} instead
converges to a nonzero constant \(a_{4/3}\), simply replace \(b_*\) by the
effective coefficient \(b_{\mathrm{eff}}=b_*-a_{4/3}\).  If
\(b_{\mathrm{eff}}>0\), apply \cref{thm:alpha-transition}(iii); if
\(b_{\mathrm{eff}}<0\), use \cref{eq:one-sided-divergence-zero}.  When
\(b_{\mathrm{eff}}=0\), for every \(\varepsilon\in(0,b_*)\) the potential
is bounded below, up to an additive constant, by
\(\kappa_\infty y^2-\varepsilon|y|^{4/3}\).  The lower comparison used in
the proof of \cref{thm:alpha-transition}(iii) then gives divergence.  The
notation \(a_{4/3}\) avoids
confusing this coefficient with the threshold \(b_*\).  This observation
contains the former five-case restatement and is not repeated as a separate
corollary.

\begin{corollary}[Logarithmic specialisation]
\label{cor:bstar-logarithmic}
Define
\[
 R_{4/3}(r):=\frac{r^{4/3}}{\log(e+r)},\qquad r\ge0.
\]
Let \(I_{4/3}=\int_0^Te_1(s)^{4/3}\dd s\) and retain \(v_*\) from
\cref{thm:bstar-nonuniversal}.
Under the critical relation \cref{eq:critical-surface}, suppose
\begin{equation}\label{eq:bstar-log-assumption}
 q(y)=\kappa_\infty y^2-b_*|y|^{4/3}
      +\gamma R_{4/3}(|y|)
      +o\big(R_{4/3}(|y|)\big),
 \qquad |y|\to\infty,
\end{equation}
where \(\gamma\ne0\).  Then the critical moment diverges for
\(\gamma>0\) and is finite for \(\gamma<0\).  More precisely, the endpoint
integrand in \cref{eq:equality-endpoint-integral} satisfies
\begin{equation}\label{eq:bstar-log-endpoint-rate}
 \log \mathfrak h_x(y)
 =\frac{\gamma I_{4/3}v_*^{4/3}}6
   \frac{|y|^4}{\log|y|}
  +o\bigg(\frac{|y|^4}{\log|y|}\bigg).
\end{equation}
\end{corollary}

\begin{proof}
Apply \cref{thm:bstar-nonuniversal,eq:bstar-slow-endpoint-rate} with
\(h(r)=\gamma/\log(e+r)\), and use
\(\log(e+|y|^3v_*)=3\log|y|+O(1)\).
\end{proof}

Both outcomes are realised within \(\mathcal D_\star\): for any
\(\varepsilon>0\), take
\[
 q_\pm(y)=\kappa_\infty y^2-b_*|y|^{4/3}
          \pm\varepsilon R_{4/3}(|y|)
 \qquad (|y|\text{ sufficiently large}),
\]
complete \(q_\pm\) smoothly and positively on a compact set, and define
\(\lambda_\pm(y)=-\int_0^yq_\pm(u)\dd u\).  Then
\(q_\pm(y)/y^2\to\kappa_\infty\), so
\(\lambda_\pm\in\mathcal D_\star\), while their critical moments have
opposite behaviour.

\subsection{Finite-depth pure-power cascade}
\label{sec:finite-cascade}

\begin{lemma}[Finite lower-power saddle expansion]
\label{lem:finite-saddle-expansion}
Let \(B,D>0\), let
\[
 F_\pm(v):=\pm Bv-D|v|^{4/3},
 \qquad
 v_D:=\bigg(\frac{3B}{4D}\bigg)^3,
\]
and fix a finite family
\[
 1<p_m<\cdots<p_1<\frac43,
 \qquad A_1,\ldots,A_m\in\mathbb R.
\]
For \(r\ge1\), put
\[
 Q_r(v):=\sum_{j=1}^m A_jr^{-(4-3p_j)}|v|^{p_j}.
\]
Then
\begin{equation}\label{eq:finite-saddle-supremum}
 \sup_{v\in\mathbb R}\{r^4F_\pm(v)+r^4Q_r(v)\}
 =r^4F_\pm(\pm v_D)
  +\sum_{j=1}^m A_jv_D^{p_j}r^{3p_j}
  +O(r^{6p_1-4}).
\end{equation}
The same expansion holds for the logarithm of the corresponding integral
if its exponent is augmented by a measurable \(E_r\) satisfying, for some
fixed \(s<4/3\),
\begin{equation}\label{eq:finite-saddle-error}
 |E_r(v)|\le Cr^2(1+|v|^s),
 \qquad v\in\mathbb R,\ r\ge1.
\end{equation}
More precisely,
\begin{equation}
\label{eq:finite-saddle-integral}
 \log\int_{\mathbb R}
 \exp\{r^4F_\pm(v)+r^4Q_r(v)+E_r(v)\}\dd v
 =r^4F_\pm(\pm v_D)
  +\sum_{j=1}^m A_jv_D^{p_j}r^{3p_j}
  +O(r^{6p_1-4}).
\end{equation}
The constants may depend on the displayed finite family but not on \(r\).
\end{lemma}

\begin{proof}
We give the proof for \(F_+\); reflection gives \(F_-\).  The function
\(F_+\) has the unique maximiser \(v_D>0\), and
\(F_+''(v_D)<0\).  Set
\[
 \delta_r:=r^{-(4-3p_1)}.
\]
On every compact neighbourhood of \(v_D\),
\(\|Q_r\|_{C^2}=O(\delta_r)\).  Outside a fixed compact interval,
the negative term \(-D|v|^{4/3}\) dominates all the powers
\(|v|^{p_j}\), uniformly for large \(r\).  Uniform convergence
\(Q_r\to0\) on compact sets and the strict gap away from \(v_D\) therefore
localise every maximiser of \(F_++Q_r\) to a fixed neighbourhood of
\(v_D\).  The implicit-function theorem and \(F_+'(v_D)=0\) give a unique
maximiser \(v_r\) there and
\[
 v_r-v_D=O(\delta_r).
\]
Taylor expansion, first for \(F_+\) and then for \(Q_r\), yields
\[
 F_+(v_r)+Q_r(v_r)
 =F_+(v_D)+Q_r(v_D)+O(\delta_r^2).
\]
Multiplication by \(r^4\) proves
\cref{eq:finite-saddle-supremum}, because
\(r^4\delta_r^2=r^{6p_1-4}\).

For the integral, the same tail domination first restricts it to a fixed
compact interval: since \(s<4/3\), the term in
\cref{eq:finite-saddle-error} is \(o(r^4|v|^{4/3})\) as
\(|v|\to\infty\), uniformly for large \(r\).  On that compact interval
\(E_r=O(r^2)\).  The upper bound follows from the supremum estimate and
the finite length of the interval.  For the lower bound, integrate over an
interval of length \(r^{-2}\) about \(v_r\).  On that interval the loss in
\(r^4(F_++Q_r)\) is \(O(1)\), while
\cref{eq:finite-saddle-error} is bounded below by \(-O(r^2)\).
Thus the logarithmic normalisation is \(O(\log r)\).  Since \(p_1>1\),
\(6p_1-4>2\), both \(O(r^2)\) and \(O(\log r)\) are absorbed by the
remainder in \cref{eq:finite-saddle-integral}.
\end{proof}

\begin{lemma}[Finite multiscale critical-bridge expansion]
\label{lem:multiscale-critical-bridge}
Assume the critical relation \cref{eq:critical-surface}, fix \(d>0\), and
let \(1<p_m<\cdots<p_1<4/3\) and
\(a_1,\ldots,a_m\in\mathbb R\).  With
\(f=\ell_{x,y}+\sigma\beta\), define
\begin{equation}\label{eq:multiscale-bridge-functional}
 G_{d,\boldsymbol a}(y)
 :=\mathbb E\exp\bigg\{
   \frac{\kappa_\infty}{2}\|f\|_2^2
   -d\|f\|_{4/3}^{4/3}
   +\frac12\sum_{j=1}^ma_j\|f\|_{p_j}^{p_j}
 \bigg\}.
\end{equation}
Put \(I_p:=\int_0^Te_1(s)^p\dd s\), let \(v_d\) be as in
\cref{lem:slowly-varying-bridge}, and set \(r=|x+y|\).  Then
\begin{equation}\label{eq:multiscale-bridge-expansion}
 \log G_{d,\boldsymbol a}(y)
 =\mathfrak L(d)r^4
  +\frac12\sum_{j=1}^m
    a_jI_{p_j}v_d^{p_j}r^{3p_j}
  +O(r^{6p_1-4}).
\end{equation}
In particular, if \(p_1\le10/9\), the final remainder is \(O(r^{8/3})\).
\end{lemma}

\begin{proof}
Use the representation in \cref{lem:first-mode-apparatus}.  Write
\(g=\widehat m_y+\widehat\beta\), set
\(Z=\|\widehat\beta\|_\infty\), and recall that
\[
 \|\widehat m_y\|_\infty\le C(1+|y|),
 \qquad
 \widehat{\mathbb P}_y(Z>z)\le Ce^{-c_Fz^2}.
\]
For every \(1<p\le4/3\), the elementary inequality
\begin{equation}\label{eq:power-transverse-difference}
 \big||a+b|^p-|a|^p\big|
 \le C_p\big(|a|^{p-1}|b|+|b|^p\big)
\end{equation}
gives
\begin{equation}\label{eq:norm-transverse-difference}
 \big|\|ue+g\|_p^p-I_p|u|^p\big|
 \le C_p\big(|u|^{p-1}\|g\|_\infty
                  +\|g\|_\infty^p\big).
\end{equation}
Consequently, after extracting the scalar terms
\[
 B_yu-dI_{4/3}|u|^{4/3}
 +\frac12\sum_{j=1}^ma_jI_{p_j}|u|^{p_j},
\]
the absolute value of the remaining transverse exponent is bounded by
\begin{equation}\label{eq:multiscale-transverse-remainder}
 C\bigg(1+|u|^{1/3}+\sum_j|u|^{p_j-1}\bigg)(1+r+Z)
 +C(1+r+Z)^{4/3}.
\end{equation}
Here and below constants may depend on the fixed finite family.

The Gaussian tail implies the uniform exponential estimate
\begin{equation}\label{eq:fernique-linear-subquadratic}
 \log\widehat{\mathbb E}_y
  \exp\{tZ+CZ^{4/3}\}\le C(1+t^2),
 \qquad t\ge0.
\end{equation}
Indeed, after decreasing \(c_F\) if necessary, the displayed Gaussian tail
gives \(\widehat{\mathbb E}_y e^{c_FZ^2/2}<\infty\).  Young's inequality
gives
\[
 CZ^{4/3}\le \frac{c_F}{4}Z^2+C,
 \qquad
 tZ\le \frac{t^2}{c_F}+\frac{c_F}{4}Z^2,
\]
which proves \cref{eq:fernique-linear-subquadratic}.  Applying
\cref{eq:fernique-linear-subquadratic} to
\cref{eq:multiscale-transverse-remainder}, and using \(p_j<4/3\), bounds
the conditional expectation from above by the scalar first-mode integral
with an additional exponent which, after \(u=r^3v\), is bounded by
\[
 Cr^2(1+|v|^{2/3}).
\]
The factors \(e^{-\kappa_\infty l_1^2/2}\widehat H_y\) in
\cref{eq:first-mode-representation} contribute at most \(O(r^2)\) to the
logarithm.

For the matching lower bound, fix \(M\) for which
\(\widehat{\mathbb P}_y(Z\le M)>0\); this probability is independent of
\(y\), because the centred law in \cref{eq:tilted-mode-law} is independent
of \(y\).  Restrict the tilted expectation to \(\{Z\le M\}\) and use the
negative of the right-hand side of
\cref{eq:multiscale-transverse-remainder}.  After \(u=r^3v\), this again
has absolute value at most \(Cr^2(1+|v|^{2/3})\).  Also
\(\widehat H_y\ge1\), so the prefactor costs at most \(O(r^2)\) from below.

We may therefore apply \cref{lem:finite-saddle-expansion} to both bounds,
with
\[
 B=B_0,\qquad D=dI_{4/3},\qquad
 A_j=\frac{a_j}{2}I_{p_j}.
\]
The maximisers are \(\pm v_d\), according to the endpoint tail, and their
absolute values agree.  Moreover
\(F_\pm(\pm v_d)=\mathfrak L(d)\).  The change of variables contributes
only \(3\log r\), which is absorbed by \(O(r^{6p_1-4})\).  The resulting
upper and lower bounds have identical displayed coefficients and prove
\cref{eq:multiscale-bridge-expansion}.
\end{proof}

\begin{lemma}[Finite-power endpoint expansion]
\label{lem:finite-power-endpoint}
Suppose \(1<p_m<\cdots<p_1<2\) and
\[
 q(y)=\kappa_\infty y^2+\sum_{j=1}^m a_j|y|^{p_j}
      +o(|y|^{p_m}),
 \qquad |y|\to\infty.
\]
If \(\Lambda'=\lambda\) and \(\lambda'=-q\), then
\begin{equation}\label{eq:finite-power-Lambda}
 \Lambda(y)=-\frac{\kappa_\infty}{12}y^4
 -\sum_{j=1}^m
   \frac{a_j}{(p_j+1)(p_j+2)}|y|^{p_j+2}
 +o(|y|^{p_m+2}).
\end{equation}
\end{lemma}

\begin{proof}
Subtract the displayed explicit antiderivative from \(\Lambda\).  The
second derivative of the difference is \(o(|y|^{p_m})\).  For every
\(\varepsilon>0\), its absolute value is at most
\(\varepsilon|y|^{p_m}\) outside a fixed compact interval.  Two integrations
then bound the difference by
\(C_\varepsilon(1+|y|)+C\varepsilon|y|^{p_m+2}\).  Divide by
\(|y|^{p_m+2}\), take the two tail limits, and then let
\(\varepsilon\downarrow0\).
\end{proof}

\begin{theorem}[Finite-depth pure-power critical cascade]
\label{thm:finite-critical-cascade}
Let \(\lambda\in\mathcal D\), translate its equilibrium to zero, and assume
the critical relation \cref{eq:critical-surface}.  Define
\begin{equation}\label{eq:cascade-beta-recursion}
 \beta_0:=\frac43,
 \qquad
 \beta_{n+1}:=\frac{\beta_n+2}{3}\quad(n\ge0),
\end{equation}
and, with \(I_p=\int_0^Te_1(s)^p\dd s\), define
\begin{equation}\label{eq:cascade-coefficient-recursion}
 c_0^*:=-b_*,
 \qquad
 c_{n+1}^*:=
 \frac{2c_n^*}
 {\sigma^2(\beta_n+1)(\beta_n+2)
  I_{\beta_{n+1}}v_*^{\beta_{n+1}}}.
\end{equation}
Then
\begin{equation}\label{eq:cascade-beta-closed-form}
 \beta_n=1+3^{-(n+1)},
 \qquad c_n^*<0.
\end{equation}

Fix an integer \(N\ge0\), let \(1<\beta<\beta_N\), and suppose
\begin{equation}\label{eq:finite-cascade-assumption}
 q(y)=\kappa_\infty y^2
      +\sum_{j=0}^Nc_j^*|y|^{\beta_j}
      +c|y|^\beta+o(|y|^\beta),
 \qquad |y|\to\infty.
\end{equation}
For every starting state \(x\), the following classification holds.
\begin{enumerate}[label=\textup{(\roman*)}]
\item If \(\beta>\beta_{N+1}\) and \(c\ne0\), then
\[
 c>0\Longrightarrow\mathcal N_{1/2}(T,x)=\infty,
 \qquad
 c<0\Longrightarrow\mathcal N_{1/2}(T,x)<\infty.
\]
\item If \(1<\beta<\beta_{N+1}\), then
\(\mathcal N_{1/2}(T,x)=\infty\) for every \(c\in\mathbb R\).
\item If \(\beta=\beta_{N+1}\), then
\[
 c>c_{N+1}^*\Longrightarrow\mathcal N_{1/2}(T,x)=\infty,
 \qquad
 c<c_{N+1}^*\Longrightarrow\mathcal N_{1/2}(T,x)<\infty.
\]
No assertion is made at \(c=c_{N+1}^*\) without a finer asymptotic
expansion.
\end{enumerate}
Thus the recursion holds at every prescribed finite depth within the
finite pure-power class; it makes no claim for arbitrary little-\(o\)
remainders at coefficient equality or at the accumulation point
\(\lim_n\beta_n=1\).
\end{theorem}

\begin{proof}
The closed form for \(\beta_n\) follows from
\(\beta_{n+1}-1=(\beta_n-1)/3\).  The sign assertion for \(c_n^*\) follows
inductively from \cref{eq:cascade-coefficient-recursion}.

Let \(\mathfrak h_x\) be the endpoint integrand in
\cref{eq:equality-endpoint-integral}.  For the moment ignore the final
little-\(o\) term in \cref{eq:finite-cascade-assumption}.  Apply
\cref{lem:multiscale-critical-bridge} with \(d=d_*=b_*/2\), with the
coefficients \(c_j^*\), \(1\le j\le N\), and with the final coefficient
\(c\).  When \(N\ge1\), the largest perturbing power is
\(\beta_1=10/9\), so the bridge remainder is \(O(|y|^{8/3})\).  When
\(N=0\), it is \(O(|y|^{6\beta-4})\), which is smaller than
\(|y|^{3\beta}\) and, if \(\beta\le10/9\), smaller than
\(|y|^{10/3}\).  For \(N=0\) and \(\beta\ge11/9\), this remainder may
therefore absorb the displayed \(A_0|y|^{10/3}\) term below, but it remains
\(o(|y|^{3\beta})\), so the classification is unchanged.

By \cref{lem:finite-power-endpoint}, the terminal factor contributes
the corresponding twice-integrated powers.  The quartic terms cancel through
\(\mathfrak L(d_*)=\kappa_\infty/(12\sigma^2)\), which is the defining
identity for \(b_*\).  For \(1\le j\le N\), the bridge contribution
from \(c_j^*\) and the endpoint contribution from \(c_{j-1}^*\) have the
common power
\[
 3\beta_j=\beta_{j-1}+2
\]
and cancel because \cref{eq:cascade-coefficient-recursion} is equivalent
to
\[
 -\frac{c_{j-1}^*}
 {\sigma^2(\beta_{j-1}+1)(\beta_{j-1}+2)}
 +\frac{c_j^*}{2}I_{\beta_j}v_*^{\beta_j}=0.
\]

Replacing \(|x+y|^4\) by \(|y|^4\) costs \(O(|y|^3)\).  Every analogous
replacement in a lower-power bridge term costs \(o(|y|^3)\); the endpoint
Gaussian is \(O(y^2)\).  The new coefficient \(c\) also contributes an
endpoint term of order \(|y|^{\beta+2}\), which is lower than its bridge
term because \(\beta+2<3\beta\).  Since \(\beta_N+2>3\) and
\(3\beta>3\), all these terms, together with the bridge remainders above,
are lower order.  Hence
\begin{equation}\label{eq:cascade-endpoint-rate}
 \log\mathfrak h_x(y)
 =A_N|y|^{\beta_N+2}
  +\frac c2I_\beta v_*^\beta|y|^{3\beta}
  +o\big(|y|^{\max\{\beta_N+2,3\beta\}}\big),
\end{equation}
where
\begin{equation}\label{eq:cascade-positive-residual}
 A_N:=-\frac{c_N^*}
 {\sigma^2(\beta_N+1)(\beta_N+2)}>0.
\end{equation}

To restore the little-\(o\) remainder in
\cref{eq:finite-cascade-assumption}, fix \(\varepsilon>0\).  After changing
a constant on a compact set, \(q\) is bounded above and below by the same
finite sum with \(c\) replaced by \(c+\varepsilon\) and
\(c-\varepsilon\), respectively.  The endpoint expansion is unaffected at
the displayed leading scale, and the conditional bridge expectation is
squeezed between the two corresponding multiscale functionals, up to fixed
positive factors.  Divide by the dominant power in
\cref{eq:cascade-endpoint-rate} and let \(\varepsilon\downarrow0\).  This
proves \cref{eq:cascade-endpoint-rate} for the stated asymptotic assumption.

If \(3\beta>\beta_N+2\), equivalently
\(\beta>\beta_{N+1}\), the term with coefficient \(c\) decides the sign.
If \(3\beta<\beta_N+2\), the positive coefficient \(A_N\) decides it.  At
equality, the coefficient is
\[
 A_N+\frac c2I_{\beta_{N+1}}v_*^{\beta_{N+1}},
\]
which vanishes precisely at \(c=c_{N+1}^*\).  A positive leading endpoint
rate makes the integral in \cref{eq:equality-endpoint-integral} diverge,
whereas a negative rate gives an integrable stretched-exponential upper
bound.  This proves all three assertions.
\end{proof}

\begin{corollary}[The third critical power]
\label{cor:third-critical-power}
Under the hypotheses of \cref{thm:finite-critical-cascade}, suppose
\[
 q(y)=\kappa_\infty y^2-b_*|y|^{4/3}
      +\gamma_*|y|^{10/9}
      +c|y|^\beta+o(|y|^\beta),
 \qquad 1<\beta<\frac{10}{9}.
\]
Then
\[
 \beta_2=\frac{28}{27},
 \qquad
 c_2^*=\frac{81\gamma_*}
 {266\sigma^2I_{28/27}v_*^{28/27}}<0.
\]
For \(\beta>28/27\), the sign of nonzero \(c\) decides divergence or
finiteness; for \(1<\beta<28/27\), the moment diverges; and at
\(\beta=28/27\), the moment diverges for \(c>c_2^*\) and is finite for
\(c<c_2^*\).  No assertion is made at \(c=c_2^*\).
\end{corollary}

\begin{proof}
Apply \cref{thm:finite-critical-cascade} with \(N=1\).  The displayed
values follow from
\[
 \beta_2=\frac{10/9+2}{3}=\frac{28}{27},
 \qquad
 (\beta_1+1)(\beta_1+2)=\frac{532}{81},
\]
and \cref{eq:cascade-coefficient-recursion}.
\end{proof}

\subsection{Complete finite-power classification and asymptotic depth}
\label{sec:finite-power-completeness}

The preceding theorem is local in the hierarchy: it assumes that a prescribed
critical prefix has already been tuned.  We now turn it into a complete
classification on a fixed finite-power class.  The finite expansion is part
of the hypothesis; no infinite formal series is being imposed.

Throughout this subsection, fix
\(\kappa_\infty,\sigma,T>0\) satisfying the critical relation
\cref{eq:critical-surface}, and retain the sequences
\((\beta_n)_{n\ge0}\) and \((c_n^*)_{n\ge0}\) from
\cref{thm:finite-critical-cascade}.  All depth notation below is relative to
this fixed critical tuple.

\begin{definition}[Symmetric finite pure-power class]
\label{def:Psym}
A reversion rate \(q\) belongs to
\(\mathcal P_{\mathrm{sym}}(\kappa_\infty)\) if
\(q=-\lambda'\) for a drift \(\lambda\in\mathcal D_\star\) whose equilibrium
has been translated to zero, and if \(q\) admits an expansion
\begin{equation}\label{eq:Psym-expansion}
 q(y)=\kappa_\infty y^2+\sum_{j=1}^m a_j|y|^{p_j}+r(y),
 \qquad |y|\to\infty,
\end{equation}
where \(m\ge0\),
\[
 2>p_1>\cdots>p_m>1,\qquad a_j\ne0,
 \qquad r(y)=o(|y|).
\]
For \(m=0\), the sum is empty.  The word \emph{symmetric} refers to the
finite power germ in \cref{eq:Psym-expansion}; the remainder itself need not
be even.  The associated translated drift is unique, because its equilibrium
condition gives \(\lambda(0)=0\), and is therefore
\[
 \lambda(y)=-\int_0^y q(u)\dd u.
\]
\end{definition}

The lower cutoff at one is intrinsic to the cascade:
\(p\mapsto(p+2)/3\) has the unique fixed point \(p=1\).  It also ensures
that every finite-depth endpoint scale is strictly larger than three.  The
class is natural for the recursive theorem, but it is not claimed to be
maximal; slowly varying and oscillatory corrections need separate hypotheses.

\begin{lemma}[Canonical finite-power germ]
\label{lem:Psym-canonical}
The integer \(m\), the ordered exponents \(p_j\), and the coefficients
\(a_j\) in \cref{eq:Psym-expansion} are uniquely determined by \(q\).
\end{lemma}

\begin{proof}
Suppose two such representations are given and subtract them.  If their
finite sums differ, let \(p>1\) be the largest exponent whose combined
coefficient \(a\) is nonzero.  All remaining finite powers are
\(o(|y|^p)\), and the difference of the two remainders is
\(o(|y|)=o(|y|^p)\).  Division by \(|y|^p\) would therefore give
\(a+o(1)=0\), a contradiction.  Thus the finite sums coincide term by term.
\end{proof}

For \(q\in\mathcal P_{\mathrm{sym}}(\kappa_\infty)\), define its
\emph{critical-prefix length} by
\begin{equation}\label{eq:critical-prefix-depth}
 \delta_{\mathrm{cp}}(q):=\max\bigg(\{0\}\cup
 \bigg\{k\ge1:q(y)=\kappa_\infty y^2
   +\sum_{j=0}^{k-1}c_j^*|y|^{\beta_j}
   +o\big(|y|^{\beta_{k-1}}\big)\bigg\}\bigg).
\end{equation}
This maximum is finite: by \cref{lem:Psym-canonical}, a prefix of length
\(k\) requires the first \(k\) canonical terms to be present with the stated
coefficients, while \cref{eq:Psym-expansion} contains only finitely many
terms.  Put \(\delta=\delta_{\mathrm{cp}}(q)\) and subtract the corresponding
prefix.  If the
remaining finite sum is nonempty, its largest term is uniquely of the form
\begin{equation}\label{eq:residual-power-data}
 q(y)-\kappa_\infty y^2-
   \sum_{j=0}^{\delta-1}c_j^*|y|^{\beta_j}
 =a(q)|y|^{p(q)}+o\big(|y|^{p(q)}\big),
 \qquad a(q)\ne0.
\end{equation}
If it is empty, we use the convention
\begin{equation}\label{eq:missing-power-convention}
 p(q):=\beta_\delta,\qquad a(q):=0.
\end{equation}
Maximality of \(\delta\) implies that, when \(p(q)=\beta_\delta\), one always
has \(a(q)\ne c_\delta^*\).

\begin{theorem}[Complete classification on the finite pure-power class]
\label{thm:Psym-complete}
Let \(q\in\mathcal P_{\mathrm{sym}}(\kappa_\infty)\), let
\(\lambda(y)=-\int_0^yq(u)\dd u\) be its associated translated drift, and
assume the critical relation \cref{eq:critical-surface}.
Let \(\delta=\delta_{\mathrm{cp}}(q)\), \(p=p(q)\), and \(a=a(q)\) be defined by
\cref{eq:critical-prefix-depth,eq:residual-power-data,eq:missing-power-convention}.
Then, for every \(x\in\mathbb R\), exactly
one of the following alternatives applies:
\begin{enumerate}[label=\textup{(\roman*)}]
\item If \(p>\beta_\delta\), then
\[
 a>0\Longrightarrow\mathcal N_{1/2}(T,x)=\infty,
 \qquad
 a<0\Longrightarrow\mathcal N_{1/2}(T,x)<\infty.
\]
\item If \(1<p<\beta_\delta\), then
\(\mathcal N_{1/2}(T,x)=\infty\), irrespective of the sign of \(a\).
\item If \(p=\beta_\delta\), then
\[
 a>c_\delta^*\Longrightarrow\mathcal N_{1/2}(T,x)=\infty,
 \qquad
 a<c_\delta^*\Longrightarrow\mathcal N_{1/2}(T,x)<\infty.
\]
\end{enumerate}
In particular, the convention \cref{eq:missing-power-convention} always
places a terminated critical prefix on the divergent side, because
\(0>c_\delta^*\).
\end{theorem}

\begin{proof}
We first consider \(\delta=0\).  Then \(\beta_\delta=4/3\).  If \(p>4/3\) and
\(a>0\), \cref{eq:residual-power-data} gives
\(q(y)\ge\kappa_\infty y^2-C\), after changing \(C\) on a compact set;
\cref{thm:alpha-transition}(i) gives divergence.  If \(p>4/3\) and
\(a<0\), then for some \(b>0\) and all sufficiently large \(|y|\),
\[
 q(y)\le\kappa_\infty y^2-b|y|^p.
\]
Moreover, two integrations of \cref{eq:Psym-expansion} give
\(\Lambda(y)/y^4\to-\kappa_\infty/12\), so the terminal-confinement
condition \cref{eq:one-sided-Lambda-confinement} holds.  Finiteness follows
from \cref{thm:alpha-transition}(ii).

If \(1<p<4/3\), a positive \(a\) again gives
\(q(y)\ge\kappa_\infty y^2-C\), whereas a negative \(a\) gives, for a
suitable \(b>0\),
\[
 q(y)\ge\kappa_\infty y^2-b|y|^p-C.
\]
Both signs are therefore divergent by \cref{thm:alpha-transition}(i).
At \(p=4/3\), a coefficient \(a>0\) is covered by the first divergence
comparison.  If \(a=0\), this is the missing-power convention, so the
residual is \(o(|y|)\).  Fix any \(\alpha\in(1,4/3)\).  Outside a compact
set the residual is bounded below by \(-|y|^\alpha\), and hence
\cref{thm:alpha-transition}(i) again gives divergence.  For \(a<0\), write
\(b=-a\) and apply
\cref{thm:alpha-transition}(iii).  Since \(c_0^*=-b_*\), this is precisely
the comparison of \(a\) with \(c_0^*\).  The missing-power convention has
\(a=0\), so it is included in the divergent branch.  This proves all three
alternatives when \(\delta=0\).

Now suppose \(\delta\ge1\) and put \(N=\delta-1\).  If the residual finite sum is
nonempty, \cref{eq:residual-power-data} is exactly
\cref{eq:finite-cascade-assumption} with its tuned prefix through \(N\),
final exponent \(p\), and coefficient \(a\).  If the residual sum is empty,
then \(r=o(|y|)=o(|y|^{\beta_\delta})\), so the same hypothesis holds with
\(p=\beta_\delta\) and \(a=0\).  The three assertions are therefore respectively
parts (i), (ii), and (iii) of \cref{thm:finite-critical-cascade}.  The
coefficient-equality case cannot occur by the maximal definition of \(\delta\),
so the alternatives are exhaustive.  Every cited result holds for every
starting state, completing the proof.
\end{proof}

\begin{corollary}[Tail-germ invariance]
\label{cor:tail-germ-invariance}
Let \(q_1,q_2\in\mathcal P_{\mathrm{sym}}(\kappa_\infty)\), with associated
translated drifts \(\lambda_1,\lambda_2\), satisfy
\cref{eq:critical-surface}, and suppose \(q_1=q_2\) outside a compact
set.  Then, for every \(x\in\mathbb R\),
\[
 \mathcal N^{(1)}_{1/2}(T,x)<\infty
 \quad\Longleftrightarrow\quad
 \mathcal N^{(2)}_{1/2}(T,x)<\infty.
\]
\end{corollary}

\begin{proof}
The two canonical finite-power germs are identical by
\cref{lem:Psym-canonical}, so \cref{thm:Psym-complete} assigns the same
outcome.  The transfer identity also displays why compact changes do not
create a hidden scale.  Since \(q_1-q_2\) has compact support,
\(\lambda_1-\lambda_2\) is constant on each sufficiently remote tail and
any antiderivatives satisfy
\[
 \Lambda_1(y)-\Lambda_2(y)=O(1+|y|).
\]
The running-potential difference is uniformly bounded.  Whenever the
conditional bridge expectations are finite, the logarithms of the two
endpoint integrands therefore differ by at most \(O(1+|y|)\), whereas every
strict deciding endpoint rate in the proof of \cref{thm:Psym-complete} has
power greater than one.  In the branches where the critical bridge
expectation is already infinite, multiplication by a bounded running factor
preserves that fact.  This also gives a direct transfer-level explanation of
the tail-germ conclusion.
\end{proof}

The theorem gives a finite algorithm.  At stage \(n\), after the first
\(n\) critical terms have been removed, compare the largest remaining
power with \(\beta_n\).  A strict exponent comparison stops the algorithm;
at equality compare its coefficient with \(c_n^*\).  Only exact coefficient
equality advances to stage \(n+1\).  If there is no remaining power, use the
zero-coefficient convention \cref{eq:missing-power-convention}.

\begin{proposition}[Finite pointwise depth and no uniform depth bound]
\label{prop:decision-depth}
Let \(\delta_{\mathrm{dec}}(q)\) be the number of stages executed by the
preceding decision algorithm, including its terminal stage.  Then
\begin{equation}\label{eq:decision-depth-identity}
 \delta_{\mathrm{dec}}(q)=\delta_{\mathrm{cp}}(q)+1<\infty
 \qquad\text{for every }q\in\mathcal P_{\mathrm{sym}}(\kappa_\infty),
\end{equation}
but
\begin{equation}\label{eq:decision-depth-unbounded}
 \sup_{q\in\mathcal P_{\mathrm{sym}}(\kappa_\infty)}
 \delta_{\mathrm{dec}}(q)=\infty.
\end{equation}
\end{proposition}

\begin{proof}
By \cref{eq:critical-prefix-depth}, the first
\(\delta_{\mathrm{cp}}(q)\) stages are exactly the equalities
\(a=c_n^*\), \(0\le n<\delta_{\mathrm{cp}}(q)\).  Maximality of
\(\delta_{\mathrm{cp}}(q)\) says
that the next stage is strict, or has no remaining power and hence compares
\(0\) with the negative number \(c_{\delta_{\mathrm{cp}}(q)}^*\).  This proves
\cref{eq:decision-depth-identity}.

For any integer \(N\ge0\), prescribe outside a sufficiently large compact
set
\[
 q_N(y)=\kappa_\infty y^2+
       \sum_{j=0}^{N-1}c_j^*|y|^{\beta_j},
\]
with the sum empty for \(N=0\).  The quadratic term dominates the finite
lower-order sum, so \(q_N\) has a smooth strictly positive completion.
Defining \(\lambda_N(y)=-\int_0^yq_N(u)\dd u\) gives a member of
\(\mathcal D_\star\): indeed,
\(q_N(y)\sim\kappa_\infty y^2\),
\(\lambda_N(y)\sim-\kappa_\infty y^3/3\), and
\(\sigma^2q_N/\lambda_N^2\to0\).  Its canonical germ has
\(\delta_{\mathrm{cp}}(q_N)=N\), and hence
\(\delta_{\mathrm{dec}}(q_N)=N+1\).  Since \(N\) is arbitrary,
\cref{eq:decision-depth-unbounded} follows.
\end{proof}

\begin{corollary}[Opposite outcomes after an arbitrarily long common prefix]
\label{cor:arbitrary-depth-pairs}
Assume \(\kappa_\infty\sigma^2T^2=\pi^2\).  For every \(N\ge0\), there
exist translated drifts
\(\lambda_+,\lambda_-\in\mathcal D_\star\) whose reversion rates
\(q_+,q_-\in\mathcal P_{\mathrm{sym}}(\kappa_\infty)\) have the common
expansion
\[
 q_\pm(y)=\kappa_\infty y^2+
  \sum_{j=0}^{N}c_j^*|y|^{\beta_j}
  \mathbin{\pm}\varepsilon|y|^\beta+o(|y|^\beta),
 \qquad \beta_{N+1}<\beta<\beta_N,
\]
for some \(\varepsilon>0\), such that, for every \(x\in\mathbb R\),
\[
 \mathcal N^{(+)}_{1/2}(T,x)=\infty,
 \qquad
 \mathcal N^{(-)}_{1/2}(T,x)<\infty.
\]
The two reversion-rate germs agree through power \(\beta_N\), and their
terminal potentials agree through power \(\beta_N+2\).
\end{corollary}

\begin{proof}
Choose \(\beta\in(\beta_{N+1},\beta_N)\).  Outside a sufficiently large
compact set, set
\[
 \widetilde q_\pm(y)=\kappa_\infty y^2+
  \sum_{j=0}^{N}c_j^*|y|^{\beta_j}
  \mathbin{\pm}\varepsilon|y|^\beta.
\]
For large radius the leading quadratic term makes both tails positive.
Using a common smooth cutoff, complete them to strictly positive continuous
functions which agree on a neighbourhood of zero, and define
\(\lambda_\pm(y)=-\int_0^yq_\pm(u)\dd u\).  The same asymptotics as in the
proof of \cref{prop:decision-depth} put both drifts in
\(\mathcal D_\star\).  Part (i) of
\cref{thm:finite-critical-cascade}, with the tuned prefix through \(N\),
gives the two asserted outcomes.  Finally, twice integrating the difference
\(q_+-q_-=2\varepsilon|y|^\beta+o(|y|^\beta)\) gives
\[
 \Lambda_+(y)-\Lambda_-(y)
 =-\frac{2\varepsilon}{(\beta+1)(\beta+2)}|y|^{\beta+2}
   +o(|y|^{\beta+2}).
\]
Since \(\beta<\beta_N\), this difference is
\(o(|y|^{\beta_N+2})\), proving the final statement.
\end{proof}

\begin{remark}[The accumulation scale]
\label{rem:cascade-accumulation}
The recursion has a direct scaling origin.  A correction
\(|y|^\beta\) evaluated on the critical first-mode saddle
\(|u|\asymp|y|^3\) produces a bridge term of order \(|y|^{3\beta}\).
Twice integrating a predecessor \(|y|^\alpha\) in \(q\) produces a terminal
term of order \(|y|^{\alpha+2}\).  Cancellation therefore requires
\[
 3\beta=\alpha+2,
\]
which is exactly \(\beta_{n+1}=(\beta_n+2)/3\).  The fixed point of this
map is one, explaining both \(\beta_n\downarrow1\) and the lower endpoint in
\cref{def:Psym}.

At the critical coefficient,
\[
 \mathfrak L(d_*)=\frac14B_0v_*
 =\frac{\kappa_\infty}{12\sigma^2},
 \qquad
 v_*=\frac{\kappa_\infty}{3\sigma^2B_0}
     =\frac{\pi}{3\sqrt2\,\sigma^2\sqrt T}.
\]
Moreover, \(I_1=2\sqrt{2T}/\pi\), and hence
\[
 \sigma^2I_1v_*=\frac23.
\]
The coefficient recursion consequently gives
\[
 \frac{c_{n+1}^*}{c_n^*}
 =\frac{2}{\sigma^2(\beta_n+1)(\beta_n+2)
 I_{\beta_{n+1}}v_*^{\beta_{n+1}}}
 \longrightarrow
 \frac{2}{6\sigma^2I_1v_*}=\frac12.
\]
Thus the critical coefficients themselves decay geometrically to zero, with
the universal limiting ratio \(1/2\).  The identity
\(\sigma^2I_1v_*=2/3\) is the numerical trace of the midpoint geometry in
\cref{eq:first-mode-midpoint}.  Individual finite-depth ratios are not
universal; for the numerical normalisation
\(\sigma=1\), \(\kappa_\infty=3\), and \(T=T_c\), the first is
\(c_1^*/c_0^*=0.42376\ldots\).

The hierarchy is correspondingly non-uniform in depth.  Indeed,
\[
 \beta_n-\beta_{n+1}=\frac{2}{3^{n+2}},
\]
so, for any fixed \(K>1\), the two adjacent powers satisfy
\[
 \frac{r^{\beta_n}}{r^{\beta_{n+1}}}\ge K
 \quad\Longleftrightarrow\quad
 \log r\ge\frac{3^{n+2}}2\log K.
\]
Moreover, the limiting coefficient ratio implies
\(\log|c_n^*|=-n\log2+o(n)\).  Thus even a fixed multiplicative separation
of adjacent powers is postponed to state magnitudes that grow extremely
rapidly with the depth.  The unbounded-depth theorem is therefore a
structural statement about asymptotic decision rules, not a claim that high
levels should be visible at moderate state magnitudes.

At every finite depth, both competing powers satisfy
\(\beta_n+2=3\beta_{n+1}>3\).  Hence replacing \(x+y\) by \(y\), whose
first omitted term is cubic, cannot affect any finite-level classification.
The displacement of the cancelled quartic bridge term is
\[
 \mathfrak L(d_*)\big(|x+y|^4-|y|^4\big)
 =\frac{\kappa_\infty x}{3\sigma^2}y^3+O(y^2).
\]
Thus the limiting scale \(\beta_n+2\downarrow3\) is the first scale at
which the starting state can enter the leading endpoint balance.

There is a useful formal check just beyond the symmetric class.  If an odd
linear correction \(\gamma y\) were added to \(q\), then its first-mode
bridge contribution at the limiting saddle would be
\[
 \frac\gamma2 I_1v_*y^3=\frac{\gamma}{3\sigma^2}y^3,
\]
whereas twice integration contributes
\(-\gamma y^3/(6\sigma^2)\) through the terminal potential.  Combining
these terms with the preceding displacement gives the formal cubic
coefficient
\begin{equation}\label{eq:formal-accumulation-coefficient}
 \frac{\gamma+2\kappa_\infty x}{6\sigma^2}\,y^3.
\end{equation}
This calculation is not an accumulation-point theorem: the linear power is
the fixed point of the recursion, it breaks the symmetric germ, and the
finite-depth remainder estimates are not uniform in the depth.  It only
identifies the cubic quantity that a separate limiting theory would have to
control.  No infinite-depth classification is claimed here.
\end{remark}

\subsection{Equality-surface synthesis and scope}

The two entry transitions are summarised in
\cref{fig:four-thirds-phase-diagram}.  Beyond the second threshold, the
local continuation is \cref{thm:finite-critical-cascade}; on the whole class
\(\mathcal P_{\mathrm{sym}}(\kappa_\infty)\), its exhaustive form is
\cref{thm:Psym-complete}.

\begin{figure}[H]
\centering
\begin{tikzpicture}[x=1cm,y=1cm]
  \begin{scope}
    \node[font=\small] at (2.95,2.72)
      {\textup{(a)} Power transition (\(b>0\))};
    \fill[gray!18] (2.95,0.55) rectangle (5.9,2.15);
    \draw (0,0.55) rectangle (5.9,2.15);
    \draw[very thick,gray!55] (2.95,0.55) -- (2.95,2.15);
    \node[font=\small] at (1.48,1.30) {divergent};
    \node[font=\small] at (4.43,1.30) {finite};
    \node[font=\scriptsize,fill=white,inner sep=1pt] at (2.95,1.75)
      {split by \(b\)};
    \fill[white] (2.95,1.22) circle (3.2pt);
    \begin{scope}
      \clip (2.95,1.22) circle (3.2pt);
      \fill[gray!60] (2.95,1.05) rectangle (3.2,1.40);
    \end{scope}
    \draw[thick] (2.95,1.22) circle (3.2pt);
    \draw[->] (-0.05,0.30) -- (6.15,0.30)
      node[right,font=\scriptsize] {\(\alpha\)};
    \draw (2.95,0.24) -- (2.95,0.36);
    \node[font=\scriptsize] at (2.95,0.02) {\(4/3\)};
    \node[font=\scriptsize] at (1.48,0.80) {\(\alpha<4/3\)};
    \node[font=\scriptsize] at (4.43,0.80) {\(\alpha>4/3\)};
  \end{scope}

  \begin{scope}[xshift=7.25cm]
    \node[font=\small] at (2.95,2.72)
      {\textup{(b)} Coefficient transition (\(\alpha=4/3\))};
    \fill[gray!18] (2.95,0.55) rectangle (5.9,2.15);
    \draw (0,0.55) rectangle (5.9,2.15);
    \draw[very thick,gray!55] (2.95,0.55) -- (2.95,2.15);
    \node[font=\small] at (1.48,1.30) {divergent};
    \node[font=\small] at (4.43,1.30) {finite};
    \node[font=\scriptsize,fill=white,inner sep=1pt] at (2.95,1.75)
      {split by \(\operatorname{sgn}h\)};
    \fill[white] (2.95,1.22) circle (3.2pt);
    \begin{scope}
      \clip (2.95,1.22) circle (3.2pt);
      \fill[gray!60] (2.95,1.05) rectangle (3.2,1.40);
    \end{scope}
    \draw[thick] (2.95,1.22) circle (3.2pt);
    \draw[->] (-0.05,0.30) -- (6.15,0.30)
      node[right,font=\scriptsize] {\(b/b_*\)};
    \draw (2.95,0.24) -- (2.95,0.36);
    \node[font=\scriptsize] at (2.95,0.02) {\(1\)};
    \node[font=\scriptsize] at (1.48,0.80) {\(b<b_*\)};
    \node[font=\scriptsize] at (4.43,0.80) {\(b>b_*\)};
  \end{scope}
\end{tikzpicture}
\caption{The critical equality-surface phase diagram; shading denotes
finiteness.  Within the two-term family \cref{eq:intro-second-order},
panel \textup{(a)} records divergence for \(\alpha<4/3\) and finiteness for
\(\alpha>4/3\), for every \(b>0\); the line \(\alpha=4/3\) is resolved by
panel \textup{(b)}.  On that line,
\(b<b_*\) gives divergence and \(b>b_*\) gives finiteness.  At \(b=b_*\),
the split marker records non-universality: every eventually signed slowly
varying correction, and every signed regularly varying correction above
power \(10/9\), is divergent for positive sign and finite for negative sign.
At power \(10/9\) there is a second coefficient threshold, while smaller
corrections give divergence; see \cref{cor:second-critical-power}.  The
logarithmic family in \cref{cor:bstar-logarithmic} is an explicit
specialisation.  At coefficient equality, finite pure-power refinements
continue through the cascade in \cref{thm:finite-critical-cascade}, beginning
with \(28/27\) in \cref{cor:third-critical-power}.  Separate panels
avoid comparing \(b\) across
different powers, for which it has different physical dimensions.}
\label{fig:four-thirds-phase-diagram}
\end{figure}
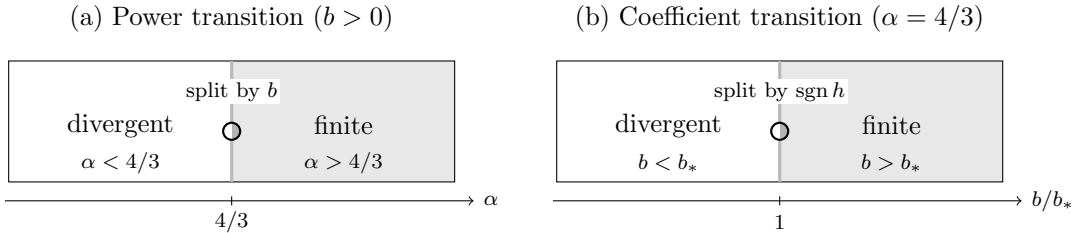

\begin{remark}[Scope of the equality classification]
\label{rem:equality-not-partition}
The one-sided hypotheses are strictly broader than the exact power
asymptotics used previously.  For example,
\[
 q(y)=\kappa_\infty y^2-\frac{b y^2}{\log(e+|y|)}
\]
satisfies \cref{eq:one-sided-finiteness}: for any fixed
\(\alpha\in(4/3,2)\), the logarithmic correction eventually dominates
\(|y|^\alpha\), while \(q(y)\ge cy^2\) for some \(c>0\).  Its critical
moment is therefore finite.

Signed regularly varying corrections are classified through the second
critical power by
\cref{thm:bstar-nonuniversal,cor:second-critical-power}; the same results
settle all smaller \(o(|y|^{10/9})\) corrections before coefficient equality.
For finite pure-power expansions, \cref{thm:Psym-complete} gives a complete
classification and \cref{prop:decision-depth} identifies the amount of tail
data it uses.  Arbitrary slowly varying remainders at a tuned coefficient and
irregular sign-changing tails are not part of
\(\mathcal P_{\mathrm{sym}}(\kappa_\infty)\).  A concrete oscillatory example
and the estimate missing from its analysis are recorded as an open problem
in \cref{sec:discussion}.

At \(b=b_*\), the quartic endpoint rate in
\cref{eq:four-thirds-endpoint-rate} cancels.
Signed regular variation is then decided down to power \(10/9\), and the
exact two-term tail is divergent.  At \(\gamma=\gamma_*\),
\cref{thm:finite-critical-cascade} classifies the next pure-power correction
when a finite expansion is given.  A slowly varying correction can dominate
every lower pure power, and irregular sign-changing tails without a coherent
stationary-phase contribution remain outside the classification.
\end{remark}

\begin{corollary}[Non-universality at leading order]
\label{cor:equality-nonuniversal}
Fix \(\kappa_\infty,\sigma,T>0\) with
\(\kappa_\infty\sigma^2T^2=\pi^2\).  There exist two drifts
\(\lambda_0,\lambda_1\in\mathcal D_\star\) satisfying
\[
 \lim_{|y|\to\infty}\frac{-\lambda_i'(\theta+y)}{y^2}
 =\kappa_\infty,\qquad i=0,1,
\]
and, for antiderivatives \(\Lambda_i'=\lambda_i\),
\[
 \lim_{|y|\to\infty}\frac{\Lambda_i(\theta+y)}{y^4}
 =-\frac{\kappa_\infty}{12},\qquad i=0,1,
\]
such that
\[
 \mathcal N^{(0)}_{1/2}(T,x)=\infty,
 \qquad
 \mathcal N^{(1)}_{1/2}(T,x)<\infty
 \quad\text{for every }x\in\mathbb R.
\]
\end{corollary}

\begin{proof}
For the divergent example take
\(\lambda_0(\theta+y)=-(\kappa_\infty/3)y^3\).
Then \(q_0(\theta+y)=\kappa_\infty y^2\), and
\cref{eq:one-sided-divergence-zero} and
\cref{thm:alpha-transition}(i) apply.

For the finite example, fix \(b>0\), choose \(R\) large enough that
\(\kappa_\infty y^2-b|y|^{3/2}>0\) for \(|y|\ge R\), and choose a smooth
strictly positive function \(q_1\) which agrees with
\(\kappa_\infty y^2-b|y|^{3/2}\) outside \([-R,R]\).  Define
\[
 \lambda_1(\theta+y):=-\int_0^yq_1(u)\dd u.
\]
Then \(\lambda_1\in C^1\), is strictly decreasing, and has its unique zero at
\(\theta\), so \(\lambda_1\in\mathcal D\).  For a smaller positive constant
in place of \(b\), the bounds in \cref{eq:one-sided-finiteness} hold with
\(\alpha=3/2\), and integration gives
\cref{eq:one-sided-Lambda-confinement}; hence
\cref{thm:alpha-transition}(ii) gives finiteness.
For both \(i=0,1\), \(q_i(\theta+y)/y^2\to\kappa_\infty\).  Integrating twice
gives the asserted \(\Lambda_i\)-asymptotic and
\(\lambda_i(\theta+y)\sim-\kappa_\infty y^3/3\); hence both drifts belong to
\(\mathcal D_\star\).
\end{proof}

\subsection{A reproducible finite-difference check}
\label{sec:bstar-numerics}

The rate in \cref{eq:four-thirds-endpoint-rate} can be normalised without
parameters:
\begin{equation}\label{eq:normalised-bstar-rate}
 \frac{12\sigma^2}{\kappa_\infty}\mathcal R(b)
 =\bigg(\frac{b_*}{b}\bigg)^3-1.
\end{equation}
Thus its zero is unique and occurs at the coefficient derived in
\cref{eq:b-star}; \cref{fig:bstar-rate} also checks the sign convention in
the endpoint comparison.

For a numerical check, we solve the transferred Feynman--Kac equation.  With
\(\Lambda(0)=0\), let
\begin{equation}\label{eq:numerical-pde}
 \partial_tv=\frac{\sigma^2}{2}\partial_{yy}v+\frac{q(y)}2v,
 \qquad
 v(0,y)=e^{\Lambda(y)/\sigma^2}.
\end{equation}
Then \cref{eq:transfer-Psi} gives
\(\mathcal N_{1/2}(t,x)=e^{-\Lambda(x)/\sigma^2}v(t,x)\).  This is the
gauge-equivalent equation whose potential is \(q/2\); in the pure cubic
case it is the inverted-oscillator equation used in
\cref{prop:cubic-blowup-rate}.

The ancillary script {\footnotesize\texttt{anc/numerical\_check.py}}, included with the
source submission, uses centred second differences on \((-L,L)\), zero
Dirichlet boundary data, and fourth-order Runge--Kutta time stepping.  We
validate it at
\(q(y)=3y^2\), \(\Lambda(y)=-y^4/4\), where
\cref{eq:cubic-mehler-integral} is exact.  The domain is \(L=15\), the time
step is at most \(5\times10^{-5}\), and the two finite-difference columns
halve the spatial mesh.  If \(\Delta\) denotes the fine-grid value of
\(\log N_T\) minus the Mehler value, the final column is
\(e^\Delta-1\), the relative error in \(N_T\), not in \(\log N_T\).

\begin{table}[H]
\centering
\small
\begin{tabular}{@{}rrrrr@{}}
\toprule
\(T/T_c\) & \(h=0.015\) & \(h=0.0075\) & Mehler & relative error in \(N_T\) \\
\midrule
0.50 & 0.297173 & 0.297166 & 0.297163 & \(2.44\times10^{-6}\) \\
0.75 & 1.346522 & 1.346007 & 1.345835 & \(1.71\times10^{-4}\) \\
0.90 & 7.677211 & 7.548137 & 7.508005 & \(4.09\times10^{-2}\) \\
\bottomrule
\end{tabular}
\caption{Mesh-halving validation against the exact Mehler integral; all
three central columns report \(\log N_T\).}
\label{tab:mehler-validation}
\end{table}

Mesh halving moves every row toward the exact value.  The remaining error
grows markedly near \(T_c\): the finest computation is accurate to four
significant digits at \(T/T_c=0.75\), but only to about four percent in
\(N_T\) at \(T/T_c=0.90\).  Thus the benchmark validates the implementation
and displays the expected loss of numerical resolution near blow-up; it is
not used as evidence for an asymptotic coefficient.

\begin{figure}[H]
\centering
\begin{tikzpicture}[x=5.4cm,y=1.05cm]
  \draw[->] (0.62,0) -- (1.70,0)
    node[right,font=\small] {\(b/b_*\)};
  \draw[->] (0.62,-0.95) -- (0.62,2.55)
    node[above,font=\small,align=center]
      {normalised\\endpoint rate};
  \draw[very thick,domain=0.68:1.62,samples=120,smooth]
    plot (\x,{1/(\x*\x*\x)-1});
  \draw[dashed,gray] (1,-0.88) -- (1,0);
  \fill (1,0) circle (1.6pt);
  \draw (1,-0.07) -- (1,0.07);
  \node[below,font=\small] at (1,-0.08) {\(1\)};
  \node[font=\small] at (0.95,1.35) {divergent};
  \node[font=\small] at (1.39,-0.28) {finite};
\end{tikzpicture}
\caption{The dimensionless quartic endpoint rate
\cref{eq:normalised-bstar-rate}.  It is positive for \(b<b_*\), vanishes at
\(b=b_*\), and is negative for \(b>b_*\).}
\label{fig:bstar-rate}
\end{figure}
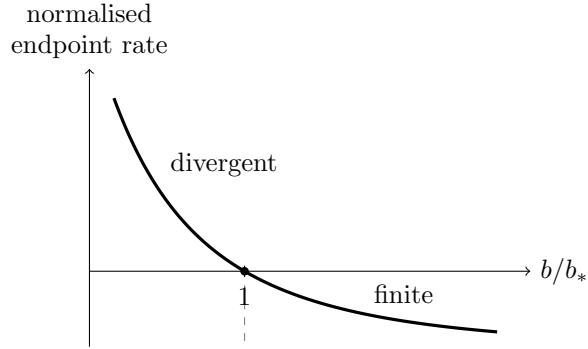

Numerically,
\[
 J_{4/3}=1.8214879859\ldots,
 \qquad
 b_*=1.187696\ldots\,\sigma^{1/3}\kappa_\infty^{5/6}.
\]
For \(\kappa_\infty=3\) and \(\sigma=1\), for example,
\[
 T_c=\frac{\pi}{\sqrt3}=1.813799\ldots,
 \qquad b_*=2.966924\ldots .
\]
If space and time have dimensions \(L\) and \(\tau\), then
\([b]=\tau^{-1}L^{-4/3}\), agreeing with
\(\Big[\sigma^{1/3}\kappa_\infty^{5/6}\Big]\).

\section{Horizon-independent moments of the two densities}
\label{sec:martingality}

The preceding results separate three statements that are sometimes
conflated:
\begin{enumerate}[label=\textup{(\roman*)}]
\item finiteness of the Novikov functional
  \(\mathcal N_{1/2}(T,x)\);
\item applicability of a classical sufficient criterion for a stochastic
  exponential;
\item the martingale property of that stochastic exponential.
\end{enumerate}

\begin{theorem}[No higher moments of the drift-removing density]
\label{thm:no-higher-moments}
Let \(\lambda\in\mathcal D\) be superlinear and satisfy
\(\rho_\sigma(\lambda)<1\).  For every \(m>1\), \(T>0\), and starting state \(x\),
\[
 \mathbb E^{\mathbb P^x}\big[Z_{0,T}^m\big]=\infty.
\]
Thus \(Z\) is a true martingale by \cref{thm:Z-martingale}, but its terminal
density belongs to no \(L^m(\mathbb P^x)\) with \(m>1\).
\end{theorem}

\begin{proof}
Let \(\mathbb Q_0\) be the Brownian law in \cref{cor:extended-transfer}, and write
\(D_T=\d\mathbb P^x/\d \mathbb Q_0=Z_{0,T}^{-1}\).  The localised
Girsanov--It\^o calculation, now between two equivalent probability
measures by \cref{thm:Z-martingale}, gives
\[
 D_T=\exp\bigg\{
  \frac{\Lambda(X_T)-\Lambda(x)}{\sigma^2}
  +\frac12\int_0^Tq(X_s)\dd s
  -\frac1{2\sigma^2}\int_0^T\lambda(X_s)^2\dd s
 \bigg\}.
\]
For \(r=m-1>0\),
\begin{equation}\label{eq:Zm-brownian}
 \mathbb E^{\mathbb P^x}[Z_{0,T}^m]
 =\mathbb E^{\mathbb Q_0}[D_T^{\,1-m}]
 =\mathbb E^{\mathbb Q_0}\bigg[
  \exp\bigg\{
   \frac{r(\Lambda(x)-\Lambda(X_T))}{\sigma^2}
   +\frac r2\int_0^T
     \bigg(\frac{\lambda(X_s)^2}{\sigma^2}-q(X_s)\bigg)\dd s
  \bigg\}\bigg].
\end{equation}

Put \(\delta=(1-\rho_\sigma(\lambda))/2>0\).  By the definition of the limsup,
outside a compact set
\[
 \frac{\lambda(y)^2}{\sigma^2}-q(y)
 \ge\delta\frac{\lambda(y)^2}{\sigma^2},
\]
while continuity gives a finite lower bound for the left-hand side on all
of \(\mathbb R\).  Restrict \cref{eq:Zm-brownian} to the returning excursion
\(A_R\) of \cref{lem:excursion}.  Its middle leg yields
\[
 \int_0^T
 \bigg(\frac{\lambda(X_s)^2}{\sigma^2}-q(X_s)\bigg)\dd s
 \ge \frac{\delta T}{3\sigma^2}
   \inf_{|y-\theta-R|\le1}\lambda(y)^2-C_T .
\]
The terminal term in \cref{eq:Zm-brownian} is bounded below on \(A_R\),
because \(X_T\) remains in a fixed neighbourhood of \(\theta\).
Combining these bounds with \cref{eq:excursion-bounds} gives
\[
 \mathbb E^{\mathbb P^x}\big[Z_{0,T}^m\big]
 \ge c\exp\bigg\{
  \frac{r\delta T}{6\sigma^2}
    \inf_{|y-\theta-R|\le1}\lambda(y)^2
  -\frac{12R^2+24R}{\sigma^2T}-C
 \bigg\}.
\]
Superlinearity makes the positive term dominate the quadratic tube cost as
\(R\to\infty\), proving the assertion.
\end{proof}

\begin{proposition}[Bounded forward density under Brownian motion]
\label{prop:forward-density-moments}
Let \(\lambda\in\mathcal D\) satisfy \(\rho_\sigma(\lambda)<1\), and set
\(D_T=\d\mathbb P^x/\d \mathbb Q_0=Z_{0,T}^{-1}\).  For every \(T>0\) and
starting state \(x\),
\begin{equation}\label{eq:D-infinity-upper}
 0<D_T\le
 \exp\bigg\{
  \frac{\Lambda(\theta)-\Lambda(x)}{\sigma^2}
  +T\sup_{y\in\mathbb R}\Psi_0(y)
 \bigg\}<\infty
 \qquad \mathbb Q_0\text{-a.s.}
\end{equation}
Thus \(D_T\in L^\infty(\mathbb Q_0)\); in particular,
\(\mathbb E^{\mathbb Q_0}\big[D_T^m\big]<\infty\) for every \(m>0\).
\end{proposition}

\begin{proof}
The proof of \cref{thm:subcritical-coefficient} with \(a=0\) shows that
\(\rho_\sigma(\lambda)<1\) implies \(\sup\Psi_0<\infty\).  The density formula in
the proof of \cref{thm:no-higher-moments} gives
\[
 D_T=\exp\bigg\{
  \frac{\Lambda(X_T)-\Lambda(x)}{\sigma^2}
  +\int_0^T\Psi_0(X_s)\dd s
 \bigg\}.
\]
By \cref{eq:Lambda-bounded}, \(\Lambda(X_T)\le\Lambda(\theta)\), and the
integral is bounded above by \(T\sup\Psi_0\).  This proves
\cref{eq:D-infinity-upper}; the moment assertion is immediate.
\end{proof}

The theorem holds under superlinearity and \(\rho_\sigma(\lambda)<1\), while the
proposition requires only \(\rho_\sigma(\lambda)<1\).  The class
\(\mathcal D_\star\), for which superlinearity holds and
\(\rho_\sigma(\lambda)=0\), is the principal special case in the phase diagram.
There the two results give a sharp directional asymmetry.  For every
\(m>1\), the order-\(m\) R\'enyi divergences satisfy
\[
 \mathrm R_m(\mathbb P^x\Vert \mathbb Q_0)
 =\frac1{m-1}\log\mathbb E^{\mathbb Q_0}\big[D_T^m\big]<\infty,
 \qquad
 \mathrm R_m(\mathbb Q_0\Vert\mathbb P^x)
 =\frac1{m-1}\log\mathbb E^{\mathbb P^x}\big[Z_{0,T}^m\big]=\infty.
\]
At infinite order the forward statement strengthens to
\[
 \mathrm R_\infty(\mathbb P^x\Vert \mathbb Q_0)
 =\log\|D_T\|_{L^\infty(\mathbb Q_0)}<\infty.
\]
Thus likelihood-ratio variance is infinite in the physical-to-Brownian
direction, while the reverse density has moments of every order and even a
constant envelope from \cref{eq:D-infinity-upper}.  The first statement is
proved directly from \cref{eq:Zm-brownian}; it is not a formal consequence
of divergence of \(\mathcal N_a\).

The higher-moment question is also part of the classical reverse-H\"older and
BMO theory of stochastic exponentials; see
\citet{lepinglememin1978,kazamaki}.  Any global reverse-H\"older \(R_p\)
bound on \([0,T]\) would, by taking the initial stopping time, imply
\(Z_{0,T}\in L^p(\mathbb P^x)\).  Hence \cref{thm:no-higher-moments} rules
out every such regime with \(p>1\), while
\cref{prop:forward-density-moments} shows that the reciprocal density is
essentially bounded under \(\mathbb Q_0\).  No full BMO characterization is
asserted here.

Novikov finiteness implies martingality, but the converse is false.  In the
examples above, statement (i) fails at an exact horizon for cubic reversion
(\cref{thm:cubic-boundary}), on every positive horizon for supercritical
reversion (\cref{thm:q-dichotomy}), and at every coefficient \(a\ge1/2\) for
sinh reversion (\cref{ex:smr-threshold}).  Statement (ii) fails with it, in
the strong sense that the failure survives partitioning
(\cref{thm:obstruction}).  Yet (iii) holds throughout: for every
\(\lambda\in\mathcal D\), \cref{thm:Z-martingale} produces a true martingale,
because the criterion it invokes asks about the non-explosion of an auxiliary
diffusion whose drift vanishes identically rather than about an exponential
moment under \(\mathbb P\).

The two questions concern different measures.  The critical-moment results are
statements about exponential functionals under the physical law
\(\mathbb P\), where the state is a nonlinear diffusion; the martingale
criterion is a statement about the auxiliary dynamics, where the state is
Brownian.  There is no tension between an infinite Novikov moment and a true
martingale, and the results above locate the exact point at which the first
fails while the second does not.

\section{Discussion}\label{sec:discussion}

The first main result is the exact spectral boundary.  The transfer identity
\cref{eq:transfer-Psi} turns the Novikov coefficient into a Brownian
Feynman--Kac problem with potential \(q/2\).  Asymptotically quadratic
reversion is therefore governed by the first bridge eigenvalue, giving
\(\kappa_\infty\sigma^2T^2=\pi^2\) and, for cubic drift, the constant
\(\pi^2/3\) with equality divergent.  Away from the Novikov coefficient, the
\(\lambda^2\)-term in the effective potential gives the horizon-independent
transition at \(a=1/2\) on \(\mathcal D_\star\).

The second main result is a complete equality classification on
\(\mathcal P_{\mathrm{sym}}(\kappa_\infty)\).  The \(4/3\) transition and
its coefficient \(b_*\) start the recursion
\(\beta_{n+1}=(\beta_n+2)/3\).  \Cref{thm:Psym-complete} proves that every
finite pure-power germ is decided by the first strict comparison, while
\cref{prop:decision-depth,cor:arbitrary-depth-pairs} show that the required
number of comparisons is finite pointwise but unbounded over the class.
This distinction is the structural content of the equality theory:
classification is complete, but no universal finite truncation of the tail
data suffices.  The sequence accumulates at power one and the endpoint rates
at cubic order; \cref{rem:cascade-accumulation} identifies the formal
limiting balance without asserting an infinite-depth theorem.  Slowly
varying and irregular sign-changing tails lie outside the finite-power class;
the regularly varying results describe part of that remaining territory.

A concrete open case is the oscillatory tail
\begin{equation}\label{eq:oscillatory-equality-example}
 q(y)=\kappa_\infty y^2+|y|^{3/2}\sin y
 \qquad (|y|\text{ sufficiently large}),
\end{equation}
with a smooth positive completion on a compact set.  It crosses the critical
parabola infinitely often and has no two-term asymptotic of the form used in
\cref{thm:alpha-transition}.  A formal stationary-phase calculation in the
conditional first-mode representation predicts divergence.  A proof would
require, for each fixed endpoint, an \(o(|u|)\) stationary-phase remainder as
the first-mode coordinate \(|u|\to\infty\), in particular along recurring
favourable phase intervals.  The present transverse estimates do not provide
that oscillatory control.  We leave this case as an open problem.

Moment criticality is distinct from loss of mass.  The drift-removing density
is a true martingale throughout \(\mathcal D\), while
\cref{thm:no-higher-moments} shows that on \(\mathcal D_\star\) it has no
higher moments under the physical law.  In the reverse direction,
\cref{prop:forward-density-moments} shows that the density under Brownian
motion is essentially bounded, and hence has every positive moment.  These
two results hold under the weaker hypotheses stated there, with
\(\mathcal D_\star\) as the principal special case.  The relevant phenomenon
is the directional asymmetry of the two densities; the exact-simulation
change of measure itself remains valid.

Constant volatility is structural in the reduction, rather than merely a
normalisation.  To see the residual explicitly, replace \(\sigma\) in
\cref{eq:sde} by a positive \(C^2\) function \(\varsigma\), set
\[
 F(x):=\int_\theta^x\frac{\dd z}{\varsigma(z)},
 \qquad Y=F(X),
\]
and write \(G=F^{-1}\) on the range of \(F\).  It\^o's formula gives
\[
 \dd Y_t=
 \left(
   \frac{\lambda(X_t)}{\varsigma(X_t)}
   -\frac12\varsigma'(X_t)
 \right)\dd t+\dd W_t.
\]
In the Lamperti coordinate, put
\[
 b(y):=\frac{\lambda(G(y))}{\varsigma(G(y))}
       -\frac12\varsigma'(G(y)),
 \qquad
 h(y):=\frac{\lambda(G(y))}{\varsigma(G(y))}.
\]
The transformed dynamics have drift \(b\), whereas the original functional
is built from \(h^2\).  Removing \(b\) and applying It\^o's formula therefore
leaves the running potential
\begin{align}\label{eq:lamperti-residual}
 \widetilde\Psi_{1/2}(y)
 &=-\frac12b'(y)+\frac12\big(h(y)^2-b(y)^2\big)\notag\\
 &=-\frac12b'(y)
   +\frac12b(y)\varsigma'(G(y))
   +\frac18\varsigma'(G(y))^2.
\end{align}
Here \(b'\) is the derivative in the Lamperti coordinate.  Returning to the
original coordinate gives the equivalent identity
\begin{equation}\label{eq:lamperti-residual-x}
 \widetilde\Psi_{1/2}(F(x))
 =-\frac12\lambda'(x)
  +\frac{\lambda(x)\varsigma'(x)}{\varsigma(x)}
  +\frac14\varsigma(x)\varsigma''(x)
  -\frac18\varsigma'(x)^2.
\end{equation}
Constant volatility removes the last three terms and recovers
\(q/2=-\lambda'/2\).  For variable volatility, the mixed term
\(\lambda\varsigma'/\varsigma\) can be of drift order rather than
derivative-of-drift order when the logarithmic volatility slope does not
decay at the compensating rate.  It is negative when volatility increases
away from the equilibrium in both tails, and positive when volatility
decreases outward.  Thus stochastic volatility can add either confinement
or reward before the curvature terms in
\cref{eq:lamperti-residual-x} are considered.  Extending the theory requires
a new state-dependent auxiliary operator, not just a change of notation.
\appendix
\section{Proof of the shifted Cameron--Martin criterion}
\label[appendix]{app:cameron-martin-proof}

\begin{proof}[Proof of \cref{lem:cameron-martin}]
The covariance operator of Brownian motion on \(L^2[0,T]\) has kernel
\(s\wedge r\).  Its eigenproblem is
\(\nu f''=-f\), with \(f(0)=0\) and \(f'(T)=0\), and hence
\begin{equation}\label{eq:bm-eigenvalues}
 \varphi_n(s)=\sqrt{\frac2T}\sin\frac{(2n-1)\pi s}{2T},
 \qquad
 \nu_n=\frac{4T^2}{(2n-1)^2\pi^2},\qquad n\ge1.
\end{equation}
Writing \(m_n=\langle u\mathbf1,\varphi_n\rangle\) and
\(\xi_n=\langle W,\varphi_n\rangle\), Parseval, monotone convergence, and
independence give
\begin{equation}\label{eq:bm-product}
 \mathbb E\bigg[e^{\frac{\gamma^2}{2}\|u\mathbf1+W\|_2^2}\bigg]
 =\lim_{N\to\infty}\prod_{n=1}^N
  \mathbb E\bigg[e^{\frac{\gamma^2}{2}(m_n+\xi_n)^2}\bigg],
 \qquad \xi_n\sim\mathcal N(0,\nu_n).
\end{equation}
By \cref{lem:gaussian-factor}, the first factor is infinite when
\(\gamma^2\nu_1\ge1\).  If \(\gamma^2\nu_1<1\), all factors are finite,
\(\sum_n\nu_n=T^2/2\), and
\[
 \sum_nm_n^2=\|u\mathbf1\|_2^2=u^2T,
\]
so both the determinant product and the shifted exponential product converge.
Thus finiteness is equivalent to
\(\gamma^2\nu_1<1\), which by \cref{eq:bm-eigenvalues} is
\(\gamma T<\pi/2\).
\end{proof}

\section{Brownian tubes and proof of the obstruction theorem}
\label[appendix]{app:tubes}

\begin{lemma}[Tube estimate]\label{lem:tube}
Let \(\varphi\colon[0,T]\to\mathbb R\) be piecewise linear with
\(\varphi(0)=x\), finitely many pieces, and slope \(\dot\varphi\).  Let
\[
 p_{\mathrm{tube}}:=\mathbb Q_0\big(\sup_{s\le T}|X_s-x|<1\big)>0.
\]
Then
\begin{equation}
\label{eq:tube-bound}
 \mathbb Q_0\big(\sup_{s\le T}|X_s-\varphi(s)|<1\big)
 \ge p_{\mathrm{tube}}
 \exp\bigg\{-\frac1{2\sigma^2}\int_0^T\dot\varphi(s)^2\dd s
 -\frac1{\sigma^2}\big(|\dot\varphi(T-)|
             +\operatorname{TV}(\dot\varphi)\big)\bigg\}.
\end{equation}
\end{lemma}

\begin{proof}
Write \(X=x+\sigma W\) under \(\mathbb Q_0\), set
\(\widetilde\varphi=\varphi-x\), and let
\(A=\{\sup_s|X_s-\varphi(s)|<1\}\).  By Cameron--Martin, under
\(\widetilde {\mathbb {Q}}\) with
\[
 \frac{\d\widetilde {\mathbb {Q}}}{\d \mathbb Q_0}
 =\exp\bigg(
  \frac1{\sigma^2}\int_0^T\dot\varphi\dd X
  -\frac1{2\sigma^2}\int_0^T\dot\varphi^2\dd s
 \bigg),
\]
the process \(X-\widetilde\varphi\) is again
\(x+\sigma\times\)Brownian motion.  Hence
\[
 \mathbb Q_0(A)=\mathbb E^{\widetilde {\mathbb {Q}}}\bigg[
  \mathbf1_A\exp\bigg(
  -\frac1{\sigma^2}\int_0^T\dot\varphi\dd X
  +\frac1{2\sigma^2}\int_0^T\dot\varphi^2\dd s
 \bigg)\bigg].
\]
Because the slope is piecewise constant, integration by parts gives
\[
 \int_0^T\dot\varphi\dd X
 =\int_0^T\dot\varphi\,\d(X-\varphi)
  +\int_0^T\dot\varphi^2\dd s.
\]
On \(A\),
\[
 \bigg|\int_0^T\dot\varphi\,\d(X-\varphi)\bigg|
 \le |\dot\varphi(T-)|+\operatorname{TV}(\dot\varphi).
\]
Finally \(\widetilde {\mathbb {Q}}(A)=p_{\mathrm{tube}}\), which proves
\cref{eq:tube-bound}.
\end{proof}

The same excursion serves both \cref{thm:supercritical-coefficient} and
\cref{thm:obstruction}, so we record it once.  Its third leg returns the path
to a fixed neighbourhood of \(\theta\), which is what keeps the terminal
contribution bounded uniformly in the excursion height.

\begin{lemma}[Returning excursion]\label{lem:excursion}
Let \(\lambda\in\mathcal D\), let \(x\in\mathbb R\) and \(T>0\) be fixed, and
set \(C_\Lambda:=\sup_{|y-\theta|\le1}|\Lambda(y)|+|\Lambda(x)|<\infty\).  For
\(R>|x-\theta|+2\) let \(\varphi_R\) rise linearly from \(x\) to
\(\theta+R\) on \([0,T/3]\), remain at \(\theta+R\) on \([T/3,2T/3]\), and
descend linearly to \(\theta\) on \([2T/3,T]\), and put
\(A_R:=\{\sup_{s\le T}|X_s-\varphi_R(s)|<1\}\).  Then, for all sufficiently
large \(R\),
\begin{equation}\label{eq:excursion-bounds}
 \mathbb Q_0(A_R)\ge p_{\mathrm{tube}}
 \exp\bigg\{-\frac{12R^2+24R}{\sigma^2T}\bigg\},
 \qquad
 \frac{\Lambda(X_T)-\Lambda(x)}{\sigma^2}\ge-\frac{C_\Lambda}{\sigma^2}
 \quad\text{on }A_R,
\end{equation}
and \(A_R\subseteq\{\sup_{s\le T}|X_s|\le|\theta|+R+1\}\), so that \(A_R\) is
an admissible localisation set in \cref{lem:local-transfer}.
\end{lemma}

\begin{proof}
The slopes of \(\varphi_R\) are bounded by \(6R/T\) for large \(R\), so
\[
 \int_0^T\dot\varphi_R^2\dd s\le\frac{24R^2}{T},
 \qquad
 |\dot\varphi_R(T-)|+\operatorname{TV}(\dot\varphi_R)
 \le\frac{24R}{T},
\]
and \cref{lem:tube} gives the first bound in
\cref{eq:excursion-bounds}.  On \(A_R\) we have \(|X_T-\theta|<1\), whence
\(\Lambda(X_T)\ge-\sup_{|y-\theta|\le1}|\Lambda(y)|\) and the second bound.
The inclusion is immediate from \(\sup_s|\varphi_R(s)-\theta|=R\) and the
tube half-width, using \(|x-\theta|<R\).
\end{proof}

\begin{proof}[Proof of \cref{thm:obstruction}]
Assume without loss of generality that \cref{eq:supercritical} holds in the
right tail along \(R_n\uparrow\infty\); the left-tail case is symmetric.  Let
\(\varphi_R\) and \(A_R\) be as in \cref{lem:excursion}.

For Novikov, use
\(F=(2\sigma^2)^{-1}\int_0^T\lambda(X_s)^2\dd s\) in
\cref{eq:transfer}.  The two integrated \(\lambda^2\) terms cancel.  Since
\(q\ge0\), on \(A_R\),
\[
 \frac12\int_0^Tq(X_s)\dd s\ge\frac{T}{6}q_*^+(R),
\]
since the middle leg keeps the path within distance one of \(\theta+R\) for
time \(T/3\).  Combining this with \cref{eq:excursion-bounds},
\begin{align*}
 \mathbb E_x\bigg[
  \exp\bigg\{\tfrac12\int_0^T\eta(X_s)^2\dd s\bigg\}\bigg]
 &\ge p_{\mathrm{tube}}
 \exp\bigg\{
  \frac{T}{6}q_*^+(R)
  -\frac{12R^2+24R}{\sigma^2T}
  -\frac{C_\Lambda}{\sigma^2}
 \bigg\}.
\end{align*}
Along \(R=R_n\), the exponent tends to infinity, proving
\cref{eq:novikov-diverges}.  Monotonicity gives the result for
\(a\ge1/2\), and the same argument on every positive subinterval proves the
partitioned statement.

For Kazamaki, \cref{eq:pathwise-identity} gives
\[
 -\frac12\int_0^T\eta\dd W
 =-\frac1{2\sigma^2}\int_0^T\lambda\dd X
  +\frac1{2\sigma^2}\int_0^T\lambda^2\dd s.
\]
Take this quantity as \(F\) in \cref{eq:transfer}.  Again the
\(\lambda^2\) terms cancel, and the \(\int\lambda\dd X\) terms combine
to
\[
 \frac1{2\sigma^2}\int_0^T\lambda\dd X
 =\frac{\Lambda(X_T)-\Lambda(x)}{2\sigma^2}
  +\frac14\int_0^Tq(X_s)\dd s.
\]
On the same tube the exponent is bounded below by
\(Tq_*^+(R)/12-C_\Lambda/(2\sigma^2)\), again using \(q\ge0\).  The tube cost
is unchanged, and supercriticality again forces the lower bound to infinity.
This proves \cref{eq:kazamaki-diverges}.
\end{proof}

\bibliography{refer}

@article{benes1971,
  author = {Bene{\v s}, V{\'a}clav E.},
  title = {Existence of Optimal Stochastic Control Laws},
  journal = {SIAM Journal on Control},
  volume = {9},
  number = {3},
  pages = {446--472},
  year = {1971},
  doi = {10.1137/0309034}
}

@article{andersenpiterbarg2007,
  author = {Andersen, Leif B. G. and Piterbarg, Vladimir V.},
  title = {Moment Explosions in Stochastic Volatility Models},
  journal = {Finance and Stochastics},
  volume = {11},
  number = {1},
  pages = {29--50},
  year = {2007},
  doi = {10.1007/s00780-006-0011-7}
}

@article{beskosroberts2005,
  author = {Beskos, Alexandros and Roberts, Gareth O.},
  title = {Exact Simulation of Diffusions},
  journal = {The Annals of Applied Probability},
  volume = {15},
  number = {4},
  pages = {2422--2444},
  year = {2005},
  doi = {10.1214/105051605000000485}
}

@article{beskospapaspiliopoulosroberts2006,
  author = {Beskos, Alexandros and Papaspiliopoulos, Omiros and
            Roberts, Gareth O.},
  title = {Retrospective Exact Simulation of Diffusion Sample Paths with
           Applications},
  journal = {Bernoulli},
  volume = {12},
  number = {6},
  pages = {1077--1098},
  year = {2006},
  doi = {10.3150/bj/1165269151}
}

@article{beskospapaspiliopoulosroberts2008,
  author = {Beskos, Alexandros and Papaspiliopoulos, Omiros and
            Roberts, Gareth O.},
  title = {A Factorisation of Diffusion Measure and Finite Sample Path
           Constructions},
  journal = {Methodology and Computing in Applied Probability},
  volume = {10},
  number = {1},
  pages = {85--104},
  year = {2008},
  doi = {10.1007/s11009-007-9060-4}
}

@article{cameronmartin1944,
  author = {Cameron, Robert H. and Martin, William T.},
  title = {Transformations of Wiener Integrals under Translations},
  journal = {Annals of Mathematics},
  volume = {45},
  number = {2},
  pages = {386--396},
  year = {1944},
  doi = {10.2307/1969276}
}

@book{chungzhao1995,
  author = {Chung, Kai Lai and Zhao, Zhongxin},
  title = {From {Brownian} Motion to {Schr{\"o}dinger}'s Equation},
  series = {Grundlehren der mathematischen Wissenschaften},
  volume = {312},
  publisher = {Springer},
  address = {Berlin},
  year = {1995},
  doi = {10.1007/978-3-642-57856-4}
}

@article{chensong2002,
  author = {Chen, Zhen-Qing and Song, Renming},
  title = {General Gauge and Conditional Gauge Theorems},
  journal = {The Annals of Probability},
  volume = {30},
  number = {3},
  pages = {1313--1339},
  year = {2002},
  doi = {10.1214/aop/1029867129}
}

@article{chenli2003,
  author = {Chen, Xia and Li, Wenbo V.},
  title = {Quadratic Functionals and Small Ball Probabilities for the
           {$m$}-Fold Integrated {Brownian} Motion},
  journal = {The Annals of Probability},
  volume = {31},
  number = {2},
  pages = {1052--1077},
  year = {2003},
  doi = {10.1214/aop/1048516545}
}

@article{fatalov2003,
  author = {Fatalov, Vladimir R.},
  title = {Asymptotics of Large Deviations of {Gaussian} Processes of
           {Wiener} Type for {$L^p$}-Functionals, {$p>0$}, and the
           Hypergeometric Function},
  journal = {Sbornik: Mathematics},
  volume = {194},
  number = {3},
  pages = {369--390},
  year = {2003},
  doi = {10.1070/SM2003v194n03ABEH000721}
}

@article{gaohannigtorcaso2003,
  author = {Gao, Fuchang and Hannig, Jan and Torcaso, Fred},
  title = {Integrated {Brownian} Motions and Exact {$L_2$}-Small Balls},
  journal = {The Annals of Probability},
  volume = {31},
  number = {3},
  pages = {1320--1337},
  year = {2003},
  doi = {10.1214/aop/1055425782}
}

@article{zhao1986,
  author = {Zhao, Zhongxin},
  title = {Green Function for {Schr\"odinger} Operator and Conditioned
           {Feynman--Kac} Gauge},
  journal = {Journal of Mathematical Analysis and Applications},
  volume = {116},
  number = {2},
  pages = {309--334},
  year = {1986},
  doi = {10.1016/S0022-247X(86)80001-4}
}

@article{kellerressel2011,
  author = {Keller-Ressel, Martin},
  title = {Moment Explosions and Long-Term Behavior of Affine Stochastic
           Volatility Models},
  journal = {Mathematical Finance},
  volume = {21},
  number = {1},
  pages = {73--98},
  year = {2011},
  doi = {10.1111/j.1467-9965.2010.00423.x}
}

@book{kazamaki,
  author = {Kazamaki, Norihiko},
  title = {Continuous Exponential Martingales and {BMO}},
  series = {Lecture Notes in Mathematics},
  volume = {1579},
  publisher = {Springer},
  year = {1994},
  doi = {10.1007/BFb0073585}
}

@article{lepinglememin1978,
  author = {L{\'e}pingle, Dominique and M{\'e}min, Jean},
  title = {Sur l'int\'egrabilit\'e uniforme des martingales exponentielles},
  journal = {Zeitschrift f{\"u}r Wahrscheinlichkeitstheorie und Verwandte
             Gebiete},
  volume = {42},
  number = {3},
  pages = {175--203},
  year = {1978},
  doi = {10.1007/BF00641409}
}

@book{kh,
  author = {Khasminskii, Rafail Z.},
  title = {Stochastic Stability of Differential Equations},
  edition = {2},
  publisher = {Springer},
  year = {2012}
}

@article{klebanerliptser2014,
  author = {Klebaner, Fima and Liptser, Robert},
  title = {When a Stochastic Exponential Is a True Martingale: Extension of
           the {Bene{\v s}} Method},
  journal = {Theory of Probability and Its Applications},
  volume = {58},
  number = {1},
  pages = {38--62},
  year = {2014},
  doi = {10.1137/S0040585X97986382}
}

@book{ks,
  author = {Karatzas, Ioannis and Shreve, Steven E.},
  title = {Brownian Motion and Stochastic Calculus},
  edition = {2},
  series = {Graduate Texts in Mathematics},
  volume = {113},
  publisher = {Springer},
  year = {1991}
}

@book{ls,
  author = {Liptser, Robert S. and Shiryaev, Albert N.},
  title = {Statistics of Random Processes I: General Theory},
  edition = {2},
  series = {Stochastic Modelling and Applied Probability},
  volume = {5},
  publisher = {Springer},
  year = {2001}
}

@article{mu,
  author = {Mijatovi{\'c}, Aleksandar and Urusov, Mikhail},
  title = {On the Martingale Property of Certain Local Martingales},
  journal = {Probability Theory and Related Fields},
  volume = {152},
  number = {1--2},
  pages = {1--30},
  year = {2012},
  doi = {10.1007/s00440-010-0314-7}
}

@article{musiela1985,
  author = {Musiela, Marek},
  title = {Divergence, Convergence and Moments of Some Integral Functionals
           of Diffusions},
  journal = {Zeitschrift f{\"u}r Wahrscheinlichkeitstheorie und Verwandte
             Gebiete},
  volume = {70},
  number = {1},
  pages = {49--65},
  year = {1985},
  doi = {10.1007/BF00532237}
}

@article{musiela1986,
  author = {Musiela, Marek},
  title = {On {Kac} Functionals of One-dimensional Diffusions},
  journal = {Stochastic Processes and their Applications},
  volume = {22},
  number = {1},
  pages = {79--88},
  year = {1986},
  doi = {10.1016/0304-4149(86)90115-8}
}

@article{novikov,
  author = {Novikov, Alexander A.},
  title = {On an Identity for Stochastic Integrals},
  journal = {Theory of Probability and its Applications},
  volume = {17},
  number = {4},
  pages = {717--720},
  year = {1972},
  doi = {10.1137/1117088},
  note = {Russian original: Teor. Veroyatnost. i Primenen. \textbf{17}
          (1972); English translation issue dated 1973}
}

@article{ruf,
  author = {Ruf, Johannes},
  title = {A New Proof for the Conditions of Novikov and Kazamaki},
  journal = {Stochastic Processes and their Applications},
  volume = {123},
  number = {2},
  pages = {404--421},
  year = {2013},
  doi = {10.1016/j.spa.2012.09.011}
}

@article{stummer1993,
  author = {Stummer, Wolfgang},
  title = {The {Novikov} and Entropy Conditions of Multidimensional Diffusion
           Processes with Singular Drift},
  journal = {Probability Theory and Related Fields},
  volume = {97},
  number = {4},
  pages = {515--542},
  year = {1993},
  doi = {10.1007/BF01192962}
}

@article{stummer1997,
  author = {Stummer, Wolfgang},
  title = {On Exponential Moments of Two {Brownian} Functionals},
  journal = {Statistics \& Probability Letters},
  volume = {31},
  number = {3},
  pages = {233--237},
  year = {1997},
  doi = {10.1016/S0167-7152(96)00035-1}
}

@article{stummersturm2000,
  author = {Stummer, Wolfgang and Sturm, Karl-Theodor},
  title = {On Exponentials of Additive Functionals of {Markov} Processes},
  journal = {Stochastic Processes and their Applications},
  volume = {85},
  number = {1},
  pages = {45--60},
  year = {2000},
  doi = {10.1016/S0304-4149(99)00064-2}
}

@article{takeda2002,
  author = {Takeda, Masayoshi},
  title = {Conditional Gaugeability and Subcriticality of Generalized
           {Schr\"odinger} Operators},
  journal = {Journal of Functional Analysis},
  volume = {191},
  number = {2},
  pages = {343--376},
  year = {2002},
  doi = {10.1006/jfan.2001.3864}
}

@article{larssonruf2019,
  author = {Larsson, Martin and Ruf, Johannes},
  title = {Stochastic Exponentials and Logarithms on Stochastic Intervals
           --- A Survey},
  journal = {Journal of Mathematical Analysis and Applications},
  volume = {476},
  number = {1},
  pages = {2--12},
  year = {2019},
  doi = {10.1016/j.jmaa.2018.11.040}
}

@article{dandapaniprotter2022,
  author = {Dandapani, Aditi and Protter, Philip},
  title = {Strict Local Martingales and the {Khasminskii} Test for
           Explosions},
  journal = {Stochastic Processes and their Applications},
  volume = {150},
  pages = {716--728},
  year = {2022},
  doi = {10.1016/j.spa.2019.03.009}
}

@book{binghamgoldieteugels1987,
  author = {Bingham, Nicholas H. and Goldie, Charles M. and
            Teugels, Jef L.},
  title = {Regular Variation},
  series = {Encyclopedia of Mathematics and its Applications},
  volume = {27},
  publisher = {Cambridge University Press},
  address = {Cambridge},
  year = {1987},
  doi = {10.1017/CBO9780511721434}
}

@book{bogachev1998,
  author = {Bogachev, Vladimir I.},
  title = {Gaussian Measures},
  series = {Mathematical Surveys and Monographs},
  volume = {62},
  publisher = {American Mathematical Society},
  address = {Providence, RI},
  year = {1998},
  doi = {10.1090/surv/062}
}

\end{document}